\documentclass[11pt]{amsart}
\usepackage{amsmath, amscd, amssymb}
\usepackage[enableskew,vcentermath]{youngtab}
\usepackage{ytableau}
\usepackage[mathscr]{eucal}
\usepackage[all,cmtip]{xy}
\usepackage{tikz-cd}
\usepackage{color}
\usepackage{amscd}
\usepackage{xfrac}
\usepackage{mathrsfs}

\usepackage{geometry}
\usepackage{amsfonts}
\usepackage[colorlinks,linkcolor=blue,citecolor=blue, pdfstartview=FitH]{hyperref}
\usepackage{graphicx}
\usepackage{cite}

\usepackage{MnSymbol}
\usepackage{mathtools}
\mathtoolsset{showonlyrefs=true}

\numberwithin{equation}{section}

\theoremstyle{definition}
\newtheorem{thm}{Theorem}[section]
\newtheorem{theorem}[thm]{Theorem}

\newtheorem{lemma}[thm]{Lemma}
\newtheorem{corollary}[thm]{Corollary}
\newtheorem{proposition}[thm]{Proposition}

\newtheorem{remark}[thm]{Remark}

\newtheorem{definition}[thm]{Definition}

\newtheorem{example}[thm]{Example}

\newtheorem{defn-thm}[thm]{Definition-Theorem}

\newenvironment{observe}{\noindent\textit{Observation}.}{\hfill $\blacktriangleleft$\par}
\newtheorem*{theorem*}{Theorem}
\newtheorem*{proposition*}{Proposition}

\newcommand{\sC}{{\mathcal C}}

\newcommand{\sF}{{\mathcal F}}

\newcommand{\sH}{{\mathcal H}}

\newcommand{\sL}{{\mathcal L}}

\newcommand{\sN}{{\mathcal N}}
\newcommand{\sO}{{\mathcal O}}

\newcommand{\CC}{{\mathbb C}}

\newcommand{\Q}{{\mathbb Q}}

\newcommand{\Z}{{\mathbb Z}}

\newcommand{\SL}{\operatorname{SL}}

\newcommand{\Hom}{\operatorname{Hom}}

\renewcommand{\Im}{\operatorname{Im}}

\newcommand{\Nm}{\operatorname{Nm}}

\newcommand{\Pic}{\operatorname{Pic}}

\newcommand{\td}{\operatorname{td}}

\newcommand{\tr}{{\operatorname{tr}}}

\newcommand{\id}{{\mathrm{id}}}

\newcommand{\Supp}{{\operatorname{Supp}}}

\newcommand{\codim}{{\operatorname{codim}}}
\newcommand{\ch}{{\operatorname{ch}}}

\newcommand{\CH}{{\operatorname{CH}}}

\newcommand{\btheorem}{\begin{theorem}}
\newcommand{\etheorem}{\end{theorem}}
\newcommand{\bproposition}{\begin{proposition}}
\newcommand{\eproposition}{\end{proposition}}
\newcommand{\bdefinition}{\begin{definition}}
\newcommand{\edefinition}{\end{definition}}
\newcommand{\bcorollary}{\begin{corollary}}
\newcommand{\ecorollary}{\end{corollary}}
\newcommand{\bproof}{\begin{proof}}
\newcommand{\eproof}{\end{proof}}
\newcommand{\bremark}{\begin{remark}}
\newcommand{\eremark}{\end{remark}}
\newcommand{\eexample}{\end{example}}
\newcommand{\bexample}{\begin{example}}
\newcommand{\elemma}{\end{lemma}}
\newcommand{\blemma}{\begin{lemma}}
\newcommand{\bobserve}{\begin{observe}}
\newcommand{\eobserve}{\end{observe}}

\renewcommand{\phi}{\varphi}
\newcommand{\ee}{\end{eqnarray*}}
\newcommand{\be}{\begin{eqnarray*}}

\newcommand{\beq}{\begin{equation}}
\newcommand{\eeq}{\end{equation}}
\newcommand{\bd}{\begin{enumerate}}
\newcommand{\ed}{\end{enumerate}}
\newcommand{\bti}{\begin{tikzcd}}
\newcommand{\eti}{\end{tikzcd}}

\renewcommand{\tilde}{\widetilde}

\newcommand{\Coh}{{\operatorname{Coh}}}

\renewcommand{\bf}[1]{\mathbf{#1}}

\newcommand{\rk}{\operatorname{rk}}

\title[Multiplicativity of the perverse filtration in the $\SL_n$ case]{Multiplicativity of the perverse filtration in the $\SL_n$ case}

\author{Xueqing Wen}
\address{Mathematical Science Research Center, Chongqing University of Technology, No. 69, Hongguang Avenue, Banan District, Chongqing, 400054, China.}
\email{wenxq@cqut.edu.cn}

\date{}

\begin{document}

\begin{abstract}
We prove that the perverse filtration of the canonical $\mathrm{SL}_n$ Hitchin system is multiplicative when the rank and degree are coprime. Over the elliptic locus of an even-degree twisted Hitchin system, we compare the Fourier projectors for compactified Pryms with the endoscopic correspondence. We then use the square-zero property of the variant cohomology and a
parabolic Higgs moduli reduction to pass to the full canonical moduli space.
\end{abstract}

\maketitle
\tableofcontents

\section{Introduction}

\subsection{The \texorpdfstring{$P=W$}{P=W} conjecture}

Let $\Sigma$ be a smooth projective complex curve of genus $g\geq2$, and fix coprime integers $n\geq2$ and $e$. Let $\mathcal M_{\mathrm{Dol}}^{\mathrm{GL}_n,e}$ denote the moduli space of stable Higgs bundles of rank $n$ and degree $e$ on $\Sigma$, and let $\mathcal M_{\mathrm B}^{\mathrm{GL}_n,e}$ be the corresponding twisted character variety, with central monodromy $\exp(2\pi\sqrt{-1}e/n)\operatorname{id}$ around a puncture. Nonabelian Hodge theory gives a real-analytic diffeomorphism between these spaces \cite{Hit87a,Don87,Cor88,Sim92}, and hence an identification of their cohomology rings. The Hitchin map \cite{Hit87S}
\[
  h_{\mathrm{GL}_n}:\mathcal M_{\mathrm{Dol}}^{\mathrm{GL}_n,e} \longrightarrow \bigoplus_{i=1}^{n}H^0(\Sigma,K_\Sigma^{\otimes i})
\]
defines a perverse filtration $P_\bullet$ on the Dolbeault side, while the mixed Hodge structure defines a weight filtration $W_\bullet$ on the Betti side. We use complex coefficients throughout, extending the weight filtration from rational cohomology.

We use the indexing convention of \cite[Subsection~1.3.3]{MSY23}. For a proper morphism $f:X\to S$ with equidimensional base $S$, the perverse filtration is
\begin{equation}
  \label{eq:perverse-filtration-convention} P_kH^m(X,\mathbb C) := \operatorname{Im}\left\{ H^m\!\left(S,{}^p\tau_{\leq k+\dim S}Rf_*\mathbb C_X\right) \longrightarrow H^m(X,\mathbb C) \right\},
\end{equation}
where ${}^p\tau_{\leq\bullet}$ denotes truncation for the middle perverse $t$-structure.

In their study of the mixed Hodge polynomials of character varieties, Hausel and Rodriguez-Villegas found a curious Poincar\'e duality and formulated the curious hard Lefschetz conjecture, which they proved in rank two \cite{HRV08}. The Lefschetz operator is cup product by a cohomology class of degree two and weight four. On the Dolbeault side, a Lefschetz symmetry for the perverse filtration is supplied by relative hard Lefschetz, since the Hitchin map is projective. This suggests a relation between the two filtrations. De Cataldo, Hausel, and Migliorini proposed the $P=W$ conjecture in \cite{dCHM12}. Under the nonabelian Hodge identification, it asserts that
\begin{equation}
  \label{eq:intro-PW} P_kH^j(\mathcal M_{\mathrm{Dol}}^{\mathrm{GL}_n,e}) =W_{2k}H^j(\mathcal M_{\mathrm B}^{\mathrm{GL}_n,e}) =W_{2k+1}H^j(\mathcal M_{\mathrm B}^{\mathrm{GL}_n,e}).
\end{equation}
They proved the conjecture for $\mathrm{GL}_2$, $\mathrm{SL}_2$, and $\mathrm{PGL}_2$. Mellit later proved curious hard Lefschetz in arbitrary rank \cite{Mel25}.

The two filtrations have different general properties. Relative hard Lefschetz holds for projective morphisms from smooth varieties, whereas the weight filtration on the cohomology of any complex algebraic variety is compatible with cup product. The $P=W$ conjecture relates these properties: it interprets curious hard Lefschetz as relative hard Lefschetz for the Hitchin map, and predicts multiplicativity of the perverse filtration \cite{dCHM12}. The latter does not hold for a general projective morphism.

The $P=W$ conjecture for $\mathrm{GL}_n$ has been proved by different methods. Maulik--Shen \cite{MS24} use vanishing cycles and a parabolic support theorem in global Springer theory; Hausel--Mellit--Minets--Schiffmann \cite{HMMS} use algebras of Hecke operators; and Maulik--Shen--Yin \cite{MSY23} use Fourier--Mukai transforms for compactified Jacobian fibrations. All three use Markman's theorem that the cohomology ring is generated by the K\"unneth components of the Chern classes of a universal bundle \cite{Mar02}. The normalized Chern character defines the Chern filtration $C_\bullet$: each K\"unneth component of $\operatorname{ch}_k$ is assigned Chern degree $k$, and $C_jH^*$ is the span of products whose total Chern degree is at most $j$. Shende's calculation of the weights of these classes identifies $C_k$ with $W_{2k}$ on the Betti side \cite{She17}. Thus $P=W$ is equivalent to $P=C$ on the Dolbeault side.

De Cataldo--Maulik--Shen proved that the universal classes lie in the perverse steps predicted by $P=W$ \cite[Theorems~0.4--0.5]{CMS1}. Together with Mellit's curious hard Lefschetz theorem, this reduces $P=W$ to the multiplicativity of $P_\bullet$ \cite[Theorem~0.6]{CMS1}. Maulik--Shen--Yin \cite{MSY23} prove this multiplicativity over the locus of integral spectral curves by Fourier methods, and use the parabolic reduction of Hausel--Mellit--Minets--Schiffmann \cite[Section~8]{HMMS} to obtain $P=C$ on the full moduli space.

\subsection{The \texorpdfstring{$\mathrm{SL}_n$}{SLn} Hitchin system}

We consider the corresponding multiplicativity problem for $\mathrm{SL}_n$. Fix a line bundle $L$ on $\Sigma$ with $\gcd(n,\deg L)=1$. Let $\mathbf M_{n,L}^{K_\Sigma}$ be the moduli space of stable Higgs bundles $(E,\theta)$ of rank $n$ with $\det E\simeq L$ and $\operatorname{tr}\theta=0$. It is smooth, and its Hitchin map is
\[
  h^{K_\Sigma}:\mathbf M_{n,L}^{K_\Sigma} \longrightarrow \mathbf B^{K_\Sigma} :=\bigoplus_{i=2}^{n}H^0(\Sigma,K_\Sigma^{\otimes i}).
\]
The group $\Gamma=\operatorname{Pic}^0(\Sigma)[n]$ acts by tensor product. Write $\widehat\Gamma=\operatorname{Hom}(\Gamma,\mathbb C^*)$. The universal classes generate $H^*(\mathbf M_{n,L}^{K_\Sigma})^\Gamma$, but not the variant part, namely the sum of the nontrivial character summands \cite{CMS2,MS21}. In prime rank, de Cataldo--Maulik--Shen proved $P=W$ on the variant part \cite{CMS2}; together with the $\mathrm{GL}_n$ theorem, this gives $P=W$ for $\mathrm{SL}_n$ in prime rank.

We prove multiplicativity in arbitrary rank.

\begin{theorem}[Theorem~\ref{main theorem}]\label{thm:intro-multiplicativity}
Let $\Sigma$ be a smooth projective complex curve of genus $g\geq2$, let $n\geq2$, and let $L$ be a line bundle with $\gcd(n,\deg L)=1$. The perverse filtration associated with the canonical $\mathrm{SL}_n$ Hitchin map is multiplicative:
\[
  P_iH^a(\mathbf M_{n,L}^{K_\Sigma}) \cup P_jH^b(\mathbf M_{n,L}^{K_\Sigma}) \subseteq P_{i+j}H^{a+b}(\mathbf M_{n,L}^{K_\Sigma})
\]
for all $i,j,a,b$.
\end{theorem}

In contrast to $\mathrm{GL}_n$, multiplicativity alone does not imply $P=W$ on the variant cohomology, since the generators of this part are not yet understood.

\subsection{Fourier transforms and endoscopy}

We first work over the elliptic locus of the $D$-twisted Hitchin system, where $D$ is effective of even degree greater than $2g-2$. Write $h:\mathbf M\to\mathbf B$ for this restriction. The spectral curves are integral, and the Hitchin fibers are compactified Pryms. We use the Fourier construction of Maulik--Shen--Yin for dualizable abelian fibrations \cite{MSY23}, but their hypotheses do not hold directly in this setting.

Although $\mathbf M$ is smooth, the coarse moduli space on the $\mathrm{PGL}_n$ side need not be. We therefore work with quotient stacks. After an \'etale base change admitting a smooth section of the spectral family, the degree and norm corrections identify the pulled-back Hitchin fibration with a norm-zero compactified Prym fibration $\check P$. The Poincar\'e kernel gives a Fourier--Mukai equivalence between $\check P$ and the smooth quotient stack $[\check P/\Gamma]$ by \cite[Theorem~4.7]{GS22}. We use the Riemann--Roch transformation of Krishna--Sreedhar \cite[Definition~10.8 and Theorem~10.11]{KS20}, together with the fixed-locus formula of Edidin--Graham \cite[Theorem~5.5]{EG08}, to convert these kernels into correspondences involving the inertia sectors. The normalization and its compatibility with composition are established in Subsection~\ref{subsec:cohomological-functors}. Following Arinkin's support arguments \cite{Ari13}, we prove codimension estimates for the relevant convolution kernels. Together with Adams operations, these estimates give Fourier vanishing and show that the filtration defined by the Fourier projectors is multiplicative. Using the multisection construction of \cite[Section~4]{MSY23}, we then construct kernels on auxiliary $\mu_n$-gerbes and obtain global projectors $\mathfrak p_k^L$ on $Rh_*\mathbb C_{\mathbf M}$.

The main difficulty is the failure of full support. In \cite[Section~2.5.2]{MSY23}, full support allows the images of the Fourier projectors to be identified with perverse truncations by restricting to the locus of smooth spectral curves. For $\mathrm{SL}_n$, the $\Gamma$-action gives a decomposition
\[
  Rh_*\mathbb C_{\mathbf M} =\bigoplus_{\kappa\in\widehat\Gamma} (Rh_*\mathbb C_{\mathbf M})_\kappa,
\]
and the nontrivial character summands have zero restriction to that locus. Thus restricting to smooth spectral curves gives no information about the projectors on these summands. We instead compare Fourier projectors along each endoscopic support.

For $\gamma\in\Gamma$, let $\kappa_\gamma(\delta)=\langle\delta,\gamma\rangle_\Gamma^{-1}$ be the character determined by the Weil pairing, and let $\mathbf B_\gamma\subset\mathbf B$ be the corresponding endoscopic locus. The support theorem \cite{Ngo10,deC17,MS21} identifies $\mathbf B_\gamma$ as the unique strict support of the nonzero perverse cohomology sheaves of $(Rh_*\mathbb C_{\mathbf M})_{\kappa_\gamma}$. Set $a_\gamma=\operatorname{codim}_{\mathbf B}\mathbf B_\gamma$; this is also the age of the $\gamma$-sector. After passing to a suitable \'etale cover $\mathbb V$ of a dense nodal open subset of $\mathbf B_\gamma$, the spectral curves admit a simultaneous normalization. Let $h(\gamma):\check{\mathbb P}(\gamma)\to\mathbb V$ be its norm-zero Prym fibration. This is a smooth proper group scheme, not necessarily connected. Under the compatible degree and norm corrections in \eqref{identifications}, the endoscopic correspondence of Maulik--Shen \cite[Section~3]{MS21} gives an isomorphism
\[
  [\mathscr C_{\mathbb V}]_{\#,\kappa_\gamma}: \bigl((Rh_*\mathbb C_{\mathbf M})|_{\mathbb V}\bigr)_{\kappa_\gamma} \xrightarrow{\ \simeq\ } \bigl(Rh(\gamma)_*\mathbb C_{\check{\mathbb P}(\gamma)}\bigr)_{\kappa_\gamma} [-2a_\gamma].
\]
Let $\mathfrak p_{k,\gamma}$ and $\mathfrak p_{\mathbb V,k}(\gamma)$ denote the Fourier projectors for the ambient norm-zero Prym and the Prym of the normalization, respectively. We prove the following compatibility.

\begin{theorem}[Proposition~\ref{prop:comparison-of-Fourier-projectors}]\label{thm:intro-endoscopic-projectors}
For every $k$, we have
\begin{equation}
  \label{eq:intro-endoscopic-projectors} [\mathscr C_{\mathbb V}]_{\#,\kappa_\gamma} \circ\mathfrak p_{k,\gamma}|_{\mathbb V} = \mathfrak p_{\mathbb V,k-a_\gamma}(\gamma) \circ[\mathscr C_{\mathbb V}]_{\#,\kappa_\gamma}
\end{equation}
as morphisms on the $\kappa_\gamma$-summand of the Hitchin direct image over $\mathbb V$.
\end{theorem}

We obtain \eqref{eq:intro-endoscopic-projectors} by comparing the Poincar\'e kernels of the nodal spectral curves and their normalizations, using the geometry of Franco--Hanson--Horn--Oliveira \cite{FHHO26}. The comparison introduces Riemann--Roch correction terms of positive degree, which disappear from the projector formula by relative Fourier vanishing. On the normalization side, the images of the Fourier projectors are the perverse truncations, by the decomposition for abelian schemes \cite{DM91,Ku93}, with the connected components taken into account. To pass from the local ambient projectors to the global projectors $\mathfrak p_k^L$, we use Lemma~\ref{lem:twisted-untwisted-projector-comparison}, which compares their images over the nodal open. The support theorem then determines their action on perverse cohomology over all of $\mathbf B_\gamma$. This identifies the images of $\mathfrak p_k^L$ with the perverse truncations and proves multiplicativity over the entire elliptic locus (Theorem~\ref{multiplicativity of ellptic locus}).

To pass to the full canonical Hitchin system, we first prove that the variant cohomology is a square-zero ideal under cup product (Proposition~\ref{square zero}):
\[
  \left( \bigoplus_{\kappa\neq1} H^*(\mathbf M_{n,L}^{K_\Sigma})_\kappa \right)^2=0.
\]
The invariant part is generated by universal classes, with its perverse filtration given by their Chern degrees (Proposition~\ref{prop:stable-filtered-generation}). It therefore suffices to bound multiplication by each universal class on the full cohomology. For the normalized class $c_k(\eta)$ defined in \eqref{eq:normalized-tautological-class}, where $\eta\in H^*(\Sigma)$, the required bound is
\begin{equation}
  \label{eq:intro-tautological-action} c_k(\eta)\cup: P_jH^*(\mathbf M_{n,L}^{K_\Sigma}) \longrightarrow P_{j+k}H^*(\mathbf M_{n,L}^{K_\Sigma}).
\end{equation}

We prove \eqref{eq:intro-tautological-action} by adapting the parabolic reduction of Hausel--Mellit--Minets--Schiffmann \cite[Section~8]{HMMS}, as used in \cite[Sections~5.4.2--5.4.3]{MSY23}, to fixed determinant and trace-free Higgs fields. Full flags at two distinct points $p,q$ give auxiliary moduli spaces with twist $K_\Sigma+p+q$, whose degree $2g$ is even and greater than $2g-2$. On the weakly parabolic elliptic locus, the perverse bounds for universal classes and multiplicativity give the analogue of \eqref{eq:intro-tautological-action}. Restriction in cohomology from the weakly parabolic space to this open locus is a strict filtered injection, while restriction to the strongly parabolic locus is a strict filtered isomorphism. These two maps transfer the bound to the strongly parabolic space. The zero-residue correspondences then transfer it to $\mathbf M_{n,L}^{K_\Sigma}$ (Propositions~\ref{prop:SL-parabolic-restrictions} and~\ref{prop:SL-zero-residue-splitting}). These correspondences commute with multiplication by universal classes. The square-zero property then completes the proof.

Most recently, in independent work, Larsen~\cite[Theorem~0.1]{Lar26} extends the Fourier approach of Maulik--Shen--Yin to compactified Prym fibrations with stacky duals, proving multiplicativity and motivic decomposition for good compactified Prym fibrations. Both approaches use normalization of nodal curves to compare Fourier projectors with the endoscopic correspondence. The Hitchin applications of this result concern the elliptic locus under the determinant-degree hypothesis of \cite[Proposition~2.13]{Lar26}. Theorem~\ref{thm:intro-multiplicativity} concerns the full coprime $K_\Sigma$-valued moduli space; the passage from the elliptic locus uses the square-zero property of the variant cohomology and the two-point parabolic Higgs moduli reduction.

\subsection{Organization of the paper}

Section~2 develops the Riemann--Roch formalism for quotient stacks and constructs Fourier projectors for compactified Pryms. We prove Fourier vanishing and multiplicativity of the filtration defined by these projectors, and extend the construction to twisted compactified Pryms. 

In Section~3, we compare the Fourier projectors with the endoscopic correspondence over the nodal locus. Using the support theorem, we identify their images with perverse truncations and obtain multiplicativity over the elliptic locus for twists of even degree greater than $2g-2$. 

Section~4 proves the square-zero property of the variant cohomology and completes the proof of Theorem~\ref{thm:intro-multiplicativity} by adapting the parabolic reduction of Hausel--Mellit--Minets--Schiffmann \cite[Section~8]{HMMS}, as used in \cite[Subsections~5.4.2--5.4.3]{MSY23}, to two-point parabolic Higgs moduli spaces with fixed determinant and trace-free Higgs fields.

\subsection*{Acknowledgements}
The author wants to thank Bin Wang and Yaoxiong Wen for helpful discussions.

\smallskip\noindent
\textbf{AI disclosure.}
AI tools were used to help the author work through the details of Grothendieck--Riemann--Roch theory for stacks. They were also used to assist in developing and refining mathematical arguments, identifying relevant references, and improving the exposition. All AI-assisted material included in the paper was checked by the author, who takes full responsibility for the contents of the paper.

\section{Projectors and Multiplicativity}

\subsection{Higgs moduli and compactified Pryms}

Let $\Sigma$ be a smooth irreducible projective curve of genus $g\ge2$. Fix an effective divisor $D$ such that $D=K_\Sigma$ or $d=\deg D>2g-2$, and a line bundle $L$ on $\Sigma$.

An $\mathrm{SL}_n$-Higgs bundle with determinant $L$ and Higgs field twisted by $D$ is a pair $(E, \theta)$, where $E$ is a rank-$n$ vector bundle with $\det E\simeq L$ and $\theta:E\to E\otimes\sO_\Sigma(D)$ is an $\sO_\Sigma$-linear morphism with $\tr(\theta)=0$. The pair is stable if every proper nonzero $\theta$-invariant subbundle $F$ satisfies
\[
  \frac{\deg F}{\rk F}< \frac{\deg E}{\rk E}.
\]

Let $\mathbf{M}_{n,L}^{D}$ denote the moduli space of stable $\mathrm{SL}_n$-Higgs bundles on $\Sigma$ with determinant $L$ and Higgs field twisted by $D$. We assume that $\operatorname{gcd}(n, \deg L)=1$, so $\mathbf{M}_{n,L}^{D}$ is a smooth quasi-projective variety. The associated Hitchin map is
\[
  h^{D}:\mathbf{M}_{n,L}^{D}\to \mathbf{B}^{D}:=\bigoplus_{i=2}^{n}H^0\big(\Sigma, \sO_{\Sigma}(iD)\big)
\]
which sends a Higgs bundle $(E, \theta)$ to $\big((-1)^i\tr(\wedge^i\theta)\big)_{i=2}^n$. The Hitchin map is projective and flat.

Let $\Gamma$ denote the finite group $\Pic^0(\Sigma)[n]$ of order $n^{2g}$. It acts on $\mathbf{M}_{n,L}^{D}$ by tensoring Higgs bundles with $n$-torsion line bundles. This action preserves the fibers of the Hitchin map $h^D$.

We write $\widehat\Gamma=\Hom(\Gamma,\CC^*)$ and use the Weil pairing with the convention $\kappa_\gamma(\delta)=\langle\delta,\gamma\rangle_\Gamma^{-1}$ for $\gamma,\delta\in\Gamma$. See \cite[Subsection~1.3]{MS21} for this convention.

The moduli stack of stable $D$-twisted $\operatorname{PGL}_n$-Higgs bundles of topological type $\deg L\bmod n$ is the quotient stack
\[
  [\mathbf{M}_{n,L}^{D}/\Gamma].
\]
We also regard this quotient stack as an orbifold.

The spectral equations parametrized by $\mathbf{B}^{D}$ define a universal curve in the total space of $\sO_{\Sigma}(D)$:
\[
  \mathbf{C}^D\longrightarrow \mathbf{B}^{D},
\]
which is a flat family of spectral curves, finite flat of degree $n$ over $\Sigma$. The arithmetic genus is
\begin{equation}
  \label{genus of spectural curves} g_{\mathbf C^D}=(n-1)(g-1)+\frac{n(n-1)}{2}d+g.
\end{equation}
We denote the spectral map by
\begin{equation}
  \label{the spectral curve map} \pi: \mathbf C^D\longrightarrow \Sigma\times \mathbf B^D.
\end{equation}
For any $b\in \mathbf{B}^{D}$, the fiber $\mathbf C^D_b$ is called the spectral curve associated to $b$.

Let $\mathbf{B}^{D,ell}\subset\mathbf{B}^D$ be the open locus of integral spectral curves. Restricting the Hitchin map gives
\[
  \xymatrix{ \mathbf{M}_{n,L}^{D, ell} \ar[dr]^{h^{D,ell}} & & [\mathbf{M}_{n,L}^{D, ell}/\Gamma] \ar[dl] \\ & \mathbf{B}^{D,ell}.& }
\]
Unless otherwise specified, we write
\[
  \mathbf{M}=\mathbf{M}_{n,L}^{D, ell}, \qquad [\mathbf{M}/\Gamma]=[\mathbf{M}_{n,L}^{D,ell}/\Gamma], \qquad \mathbf{B}=\mathbf{B}^{D,ell}, \qquad h=h^{D,ell}, \qquad \mathbf C=\mathbf C^D\times_{\mathbf B^D}\mathbf B.
\]

For $\gamma\in\Gamma$, let $\mathbf M_\gamma$ denote its fixed locus in $\mathbf M$, and put $\mathbf B_\gamma=h(\mathbf M_\gamma)$. Since $\Gamma$ is a finite abelian group, the inertia stack of $[\mathbf{M}/\Gamma]$ is given by
\[
  I[\mathbf{M}/\Gamma]=\bigsqcup_{\gamma\in \Gamma}[\mathbf{M}_\gamma/\Gamma].
\]
The Chen--Ruan cohomology of $[\mathbf{M}/\Gamma]$ is given by
\[
  \operatorname{H}^*_{\operatorname{CR}}([\mathbf{M}/\Gamma], \mathbb C)=\bigoplus_{\gamma\in \Gamma}\operatorname{H}^{*-2a_\gamma}([\mathbf{M}_\gamma/\Gamma], \mathbb{C}),
\]
where $a_\gamma$ is the \emph{age} of $[\mathbf{M}_\gamma/\Gamma]$.

Let $\gamma\in\Gamma$. For a point $x\in\mathbf M_\gamma$, write the eigenvalues of $\gamma$ on $T_x\mathbf M$ as
\[
  \nu_{\operatorname{ord}(\gamma)}^{b_1},\ldots,\nu_{\operatorname{ord}(\gamma)}^{b_{\dim\mathbf M}}, \qquad 0\leq b_i<\operatorname{ord}(\gamma),
\]
where $\nu_{\operatorname{ord}(\gamma)}$ is a primitive $\operatorname{ord}(\gamma)$-th root of unity. The age of the component of $\mathbf M_\gamma$ containing $x$ is
\[
  \frac{1}{\operatorname{ord}(\gamma)}\sum_{i=1}^{\dim\mathbf M}b_i.
\]
The following lemma shows that this number is independent of the component; we denote its common value by $a_\gamma$.

\begin{lemma}\label{lem:age computation}
For any $\gamma\in \Gamma$, we have
\begin{align}
  \label{computation of age} a_\gamma=\frac{n^2d}{2}\left(1-\frac{1}{\operatorname{ord}(\gamma)}\right).
\end{align}
Moreover, $2a_\gamma=\codim_{\mathbf M}(\mathbf M_\gamma)$.
\end{lemma}

\begin{proof}
The case \(\gamma=1\) is immediate. Fix \(x=(E,\theta)\in\mathbf M_\gamma\). Choose an isomorphism of Higgs bundles \(E\simeq E\otimes\gamma\). Its \(\operatorname{ord}(\gamma)\)-th power is scalar by stability, so, after rescaling, it gives the cyclic cover structure determined by \(\gamma\).

Let \(\nu_{\operatorname{ord}(\gamma)}\) be a primitive \(\operatorname{ord}(\gamma)\)-th root of unity. The cyclic-cover description shows that, locally on \(\Sigma\), \(E\) is the direct sum of \(\operatorname{ord}(\gamma)\) rank-\(n/\operatorname{ord}(\gamma)\) summands on which \(\gamma\) acts through the distinct characters of \(\langle\gamma\rangle\). Hence, for \(1\leq j\leq \operatorname{ord}(\gamma)-1\), the \(\nu_{\operatorname{ord}(\gamma)}^j\)-eigensubbundle \(\mathcal A_j\subset\operatorname{End}(E)\) has rank
\[
  \frac{n^2}{\operatorname{ord}(\gamma)}.
\]
Since the trace is invariant, \(\mathcal A_j\) is contained in \(\operatorname{End}_0(E)\).

The tangent space \(T_x\mathbf M\) is the first hypercohomology of the trace-free deformation complex
\[
  \operatorname{End}_0(E) \xrightarrow{[\theta,-]} \operatorname{End}_0(E)\otimes\sO_\Sigma(D).
\]
Its \(\nu_{\operatorname{ord}(\gamma)}^j\)-eigenspace is therefore the first hypercohomology of
\[
  \mathcal A_j \xrightarrow{[\theta,-]} \mathcal A_j\otimes\sO_\Sigma(D).
\]
Stability gives the vanishing of the zeroth hypercohomology. By Serre duality, the second hypercohomology is dual to the space of trace-free Higgs morphisms
\[
  (E,\theta)\longrightarrow (E,\theta)\otimes K_\Sigma(-D),
\]
which vanishes by simplicity when \(D=K_\Sigma\), and by stability and \(\deg K_\Sigma(-D)<0\) when \(d>2g-2\). Riemann--Roch consequently gives
\[
  \dim (T_x\mathbf M)_{\nu_{\operatorname{ord}(\gamma)}^j} = \frac{n^2}{\operatorname{ord}(\gamma)}d \qquad \text{for }1\leq j\leq \operatorname{ord}(\gamma)-1.
\]
Thus the age is independent of \(x\) and equals
\[
  a_\gamma = \sum_{j=1}^{\operatorname{ord}(\gamma)-1} \frac{j}{\operatorname{ord}(\gamma)}\frac{n^2d}{\operatorname{ord}(\gamma)} = \frac{n^2d}{2}\left(1-\frac{1}{\operatorname{ord}(\gamma)}\right).
\]
The tangent space of $\mathbf M_\gamma$ is the invariant subspace of $T_x\mathbf M$. Hence its normal space is the sum of the nontrivial eigenspaces computed above and has dimension
\[
  \sum_{j=1}^{\operatorname{ord}(\gamma)-1} \frac{n^2d}{\operatorname{ord}(\gamma)} =2a_\gamma.
\]
Since the fixed locus of a finite-order automorphism of a smooth variety is smooth in characteristic zero, this proves the codimension statement.
\end{proof}

Put
\begin{equation}
  \label{eq:spectral-sheaf-degree} d_L=\deg L+\frac{n(n-1)}{2}d, \qquad \Lambda_L=L_{\mathbf B}\otimes\det(\pi_*\sO_{\mathbf C})^{-1}.
\end{equation}

Fix an \'etale cover
\begin{equation}
  \label{first etale cover} B\longrightarrow \mathbf{B},
\end{equation}
so that the base change $C:=\mathbf{C}\times_{\mathbf{B}}B\rightarrow B$ admits a section $s:B\to C$ in the smooth locus of each fiber. Put $s_\Sigma=\pi\circ s$. After refining this cover, choose $\mathscr N_L\in\Pic^0(\Sigma\times B/B)$ with
\begin{equation}
  \label{eq:norm-correction-line} \mathscr N_L^{\otimes n} \simeq\Lambda_{L,B}^{-1}\otimes \sO_{\Sigma\times B}(d_L\operatorname{Im}s_\Sigma).
\end{equation}

The family $C\to B$ consists of integral projective curves with planar singularities, and the spectral map $C\to\Sigma\times B$ is finite flat of degree $n$.

Let $\overline{J}_{C}^0 \to B$ denote the compactified Jacobian of degree $0$ for the family $C \to B$. The norm defines a morphism $\Nm: \overline{J}_{C}^0 \to B \times J_\Sigma^0$, where $J_\Sigma^0$ is the Jacobian of $\Sigma$. The norm morphism is flat.

Define the compactified Prym $\check{P}:=\Nm^{-1}(\sO_{\Sigma})\subset \overline{J}_{C}^0$ as the kernel of the norm map. Tensor product defines a $\Gamma$-action on $\check P$ over $B$.

The BNR correspondence~\cite{BNR89} identifies $\mathbf M$ with $\Nm^{-1}(\Lambda_L)\subset\overline{\mathbf J}_{\mathbf C}^{d_L}$, since $\det(\pi_*\mathscr H)=\Nm(\mathscr H)\otimes\det(\pi_*\sO_{\mathbf C})$. Tensoring first by $\sO_C(-d_L\operatorname{Im}s)$ corrects the degree, and then by $\pi^*\mathscr N_L$ corrects the norm. This gives a $\Gamma$-equivariant identification
\begin{align}
  \label{moduli and prym isomorphism} \check{P}\cong M:= \bf{M}\times_{\bf B} B.
\end{align}

We write $g_P=g_{\mathbf C^D}-g$ for the relative dimension of $\check P\to B$, so that
\[
  g_P=(n-1)(g-1)+\frac{n(n-1)}{2}d.
\]
Write $j$ for the inclusion $\check{P}\hookrightarrow \overline{J}^0_{C}$.

\begin{lemma}\label{lem:H-fixed-dimensions}
Assume $d>2g-2$. Then
\begin{equation}
  \label{dimension of moduli} \dim\mathbf M=(n^2-1)d, \qquad \dim\mathbf B= \left(\frac{n(n+1)}{2}-1\right)d-(n-1)(g-1).
\end{equation}
For every $\gamma\in\Gamma$ with nonempty fixed locus,
\begin{equation}
  \label{dimension of MH} \dim\mathbf M_\gamma=\dim\mathbf M-2a_\gamma, \qquad \dim\mathbf B_\gamma=\dim\mathbf B-a_\gamma,
\end{equation}
and, for every geometric point $b\to B$,
\begin{equation}
  \label{eq:fixed-Prym-fiber-dimension} \dim(\check P_\gamma)_b\leq g_P-a_\gamma.
\end{equation}
\end{lemma}

\begin{proof}
The first two formulas follow from Riemann--Roch; equivalently, they are the dimension formulas in \cite[Proposition~2.4.9]{deC17}. The first formula in \eqref{dimension of MH} follows from Lemma~\ref{lem:age computation}. The endoscopic base has codimension $a_\gamma$, and the fixed Hitchin fibration is the finite quotient of the corresponding cyclic endoscopic Hitchin fibration; see \cite[Section~3]{MS21}. The latter is flat of relative dimension $g_P-a_\gamma$. This gives the second formula in \eqref{dimension of MH} and \eqref{eq:fixed-Prym-fiber-dimension}.
\end{proof}

\begin{lemma}\label{etale cover of J bar}
The natural map
\[
  \eta: J_{\Sigma}^0\times \check{P}\longrightarrow \overline{J}_{C}^0
\]
given by $\eta(L_B,\sF)=\pi^*L_B\otimes\sF$, where $L_B$ is a degree-zero line bundle on $\Sigma\times B$, is finite \'etale and surjective. Moreover,
\[
  \eta_*\sO_{J_{\Sigma}^0\times \check{P}}\cong \bigoplus_{L_i\in \Gamma} \operatorname{Nm}^*(L_{i})_B.
\]
Here $L_i\in \Gamma =J_{\Sigma}^0[n]$ is regarded as an element of $(J^0_{\Sigma})^\vee\cong J^0_{\Sigma}$ and $(L_{i})_B$ denotes its pullback to $J^0_{\Sigma}\times B$. The inclusion $j:\check P\hookrightarrow\overline J_C^0$ factors as
\[
  \check{P} \hookrightarrow J_{\Sigma}\times \check{P}\longrightarrow \overline{J}_{C}^{0}.
\]
Hence $j$ has finite Tor-dimension.
\end{lemma}

\begin{proof}
The norm identity $\Nm(F\otimes\pi^*L_B)=\Nm(F)\otimes L_B^{\otimes n}$ gives a Cartesian square
\[
\begin{tikzcd}[ampersand replacement=\&]
  J_\Sigma^0\times\check P \arrow[r,"\eta"]\arrow[d,"{\operatorname{pr}_1}"'] \&\overline J_C^0\arrow[d,"\Nm"]\\
  J_\Sigma^0\times B\arrow[r,"{[n]\times\id_B}"']\&J_\Sigma^0\times B.
\end{tikzcd}
\]
Thus $\eta$ is the base change of multiplication by $n$, and is finite \'etale and surjective. The character decomposition of $[n]_*\sO_{J_\Sigma^0}$ is the direct sum of the $n$-torsion line bundles on the dual abelian variety. Flat base change therefore gives the displayed formula for $\eta_*\sO$.

Finally, $j$ is the composite of the zero section $\check P\hookrightarrow J_\Sigma^0\times\check P$, a regular immersion, and the \'etale morphism $\eta$. It has finite Tor-dimension.
\end{proof}

Consequently,
\begin{align}
  [\overline{J}^0_{C}/J^0_{\Sigma}]\cong [\check{P}/\Gamma] \cong [M/\Gamma].
\end{align}
The structure morphisms are
\[
  \xymatrix{ \check{P} \ar[dr]_{\check{h}} & & [\check{P}/\Gamma] \ar[dl]^{\hat{h}} \\ & B.& }
\]
The inertia stack is
\begin{equation}
  I[\check{P}/\Gamma]=\bigsqcup_{\gamma\in \Gamma}[\check{P}_{\gamma}/\Gamma].
\end{equation}

Let $\mathscr{P}$ be the normalized Poincaré sheaf on $\overline{J}_{C}^0 \times_B \overline{J}_{C}^0$ constructed by Arinkin \cite{Ari11,Ari13}. Following \cite{GS22}, let $\mathscr P_1$ be the restriction of $\mathscr P$ to $\check P\times_B\overline J_{C}^0$. By \cite[Proposition~4.4]{GS22}, this sheaf is $J_\Sigma^0$-equivariant and descends to a sheaf $\mathscr F$ on $\check P\times_B[\check P/\Gamma]$. We also define
\begin{equation}
  \label{the sheaf R} \mathscr{R}=(j\times j)^*\mathscr{P},
\end{equation}
as in \cite[Section~5.3]{FHHO26}. The isomorphism $(\id\times q)^*\mathscr F\simeq\mathscr R$, where $q:\check P\to[\check P/\Gamma]$ is the quotient map, equips $\mathscr R$ with the required $\Gamma$-linearization on the second factor.

We use the following Fourier--Mukai equivalence.

\begin{proposition}[Theorem~4.7 of \cite{GS22}]
The Fourier--Mukai transform
\[
  \Phi_{\mathscr{F}}: D^b\mathrm{Coh}(\check{P})\longrightarrow D^b\mathrm{Coh}([\check{P}/\Gamma])
\]
is an equivalence of derived categories.
\end{proposition}

Set
\[
  \mathscr{F}^{-1}= R\mathcal{H}om(\mathscr{F}, p_2^*\omega_{[\check{P}/\Gamma]/B}[g_P]).
\]
The factors are interchanged after taking the derived dual, so $\mathscr F^{-1}$ is a kernel on $[\check P/\Gamma]\times_B\check P$. Then $\Phi_{\mathscr F^{-1}}$ is inverse to $\Phi_{\mathscr F}$. The sheaf $\mathscr F$ is maximal Cohen--Macaulay and flat over both factors; see \cite[Lemma~4.3 and Proposition~4.4]{GS22}. The relative product is Gorenstein, since it is the flat pullback of the diagonal of the smooth base inside the smooth absolute product. Relative duality for flat families of maximal Cohen--Macaulay sheaves, as in \cite[Lemma~2.1]{Ari13}, shows that $\mathscr F^{-1}[-g_P]$ is also a coherent sheaf flat over both factors.

\subsection{Cohomological correspondences and Riemann--Roch}\label{subsec:cohomological-functors}

Let $\mathcal V$ be a separated finite-type Deligne--Mumford quotient stack whose coarse moduli space is quasi-projective. Let
\begin{equation}
  \label{KS's tau functor} \tau^{\mathrm{orb}}: G_0(\mathcal V)_{\CC} \longrightarrow \CH_*(I\mathcal V)_{\CC} \longrightarrow H^{\mathrm{BM}}_{2*}(I\mathcal V,\CC)
\end{equation}
be the Riemann--Roch transformation of \cite[Definition~10.8]{KS20}, followed by the cycle class map. Here $G_0$ denotes the Grothendieck group of coherent sheaves. We use the equivariant convention for Chow groups; see \cite[Theorem~10.1]{KS20} and \cite[Subsection~6.4]{EG05}. Borel--Moore homology and proper pushforwards are taken in the same stack convention. By \cite[Theorem~10.11]{KS20}, for a proper morphism $u:\mathcal V\to\mathcal W$ between such stacks, with induced inertia morphism $Iu$, we have
\begin{equation}
  \label{eq:orbifold-tau-proper-covariance} (Iu)_*\tau^{\mathrm{orb}}(\alpha) =\tau^{\mathrm{orb}}(Ru_*\alpha), \qquad \alpha\in G_0(\mathcal V)_{\CC}.
\end{equation}
We will give an explicit formula for $\tau^{\mathrm{orb}}$ and prove its compatibility with composition in Proposition~\ref{prop:normalized-tau-composition}.

Let $S$ be a smooth quasi-projective variety. For $i=1,2,3$, let $f_i:X_i\to S$ be proper and flat, with $X_i$ smooth and quasi-projective, and let $G_i$ be a finite abelian group acting on $X_i$ over $S$. Set $\mathcal X_i=[X_i/G_i]$. For $g\in G_i$, let $X_{i,g}$ be its fixed locus and put $\mathcal X_{i,g}=[X_{i,g}/G_i]$, with structure morphism $h_{i,g}:\mathcal X_{i,g}\to S$. The fixed loci are smooth. We assume that every nonempty sector is equidimensional and has constant integral age, and write
\[
  d_{i,g}:=\dim\mathcal X_{i,g},\qquad a_i(g)\in\Z.
\]

Set
\[
  \mathcal W_{ij}:=\mathcal X_i\times_S\mathcal X_j, \qquad \mathcal M_{ij}:=\mathcal X_i\times\mathcal X_j,
\]
and write $j_{ij}:\mathcal W_{ij}\hookrightarrow\mathcal M_{ij}$ for the natural closed immersion. For $g\in G_i$ and $h\in G_j$, put
\[
  \mathcal W_{ij}^{g,h}:=\mathcal X_{i,g}\times_S\mathcal X_{j,h}, \qquad \mathcal M_{ij}^{g,h}:=\mathcal X_{i,g}\times\mathcal X_{j,h},
\]
with inclusions
\[
  \iota_{ij}^{g,h}:\mathcal M_{ij}^{g,h}\hookrightarrow\mathcal M_{ij}, \qquad \jmath_{ij}^{g,h}:\mathcal W_{ij}^{g,h} \hookrightarrow\mathcal M_{ij}^{g,h}.
\]
Thus $I\mathcal W_{ij}=\bigsqcup_{g,h}\mathcal W_{ij}^{g,h}$. The absolute products are smooth, whereas the relative products and their sectors may be singular. We use $\operatorname{pr}_1,\operatorname{pr}_2$ for the projections from products.

Let $N_{i,g}$ be the normal bundle of $i_{i,g}:\mathcal X_{i,g}\hookrightarrow\mathcal X_i$. Since $\langle g\rangle$ acts trivially on $X_{i,g}$, a $G_i$-equivariant coherent sheaf decomposes into eigensheaves $E=\bigoplus_{\xi\in\widehat{\langle g\rangle}}E_\xi$. Following \cite[Subsection~5.2]{EG08}, set
\[
  t_g([E]):=\sum_\xi\xi(g)[E_\xi].
\]
This operation extends to bounded complexes and supported $K$-classes. On a pair of sectors, $t_{g,h}$ weights the $(\xi,\eta)$-eigensummand by $\xi(g)\eta(h)$. Define
\[
  \lambda_{-1}(N_{i,g}^{\vee}) :=\sum_{q\geq0}(-1)^q[\wedge^qN_{i,g}^{\vee}], \qquad \Lambda_{i,g}:=t_g\bigl(\lambda_{-1}(N_{i,g}^{\vee})\bigr),
\]
and
\[
  \Theta_{i,g} :=\operatorname{td}(T_{\mathcal X_{i,g}}) \ch(\Lambda_{i,g})^{-1}.
\]
Here $\ch$ and $\operatorname{td}$ denote the ordinary equivariant Chern character and Todd class. Since $N_{i,g}$ has no eigenvalue $1$, $\ch(\Lambda_{i,g})$ has nonzero degree-zero term on each connected component. Hence $\Theta_{i,g}$ is invertible.

For a closed immersion $\mathcal Z\hookrightarrow\mathcal M$ into a smooth quotient Deligne--Mumford stack, write
\[
  \ch_{\mathcal Z}^{\mathcal M}: K_0^{\mathcal Z}(\mathcal M)_{\CC} \longrightarrow H^{\mathrm{BM}}_{2*}(\mathcal Z,\CC)
\]
for the localized Chern character with support on $\mathcal Z$, after capping its bivariant class with $[\mathcal M]$ and applying the cycle class map; see \cite[Section~18.1]{Ful98}. For quotient stacks, we apply this construction to the mixed quotients defining equivariant Chow groups in \cite{EG98}. Here $K_0^{\mathcal Z}(\mathcal M)$ consists of classes of perfect complexes on $\mathcal M$ that are exact off $\mathcal Z$. For such a complex $A^\bullet$, let $[A^\bullet]_{\mathcal Z}$ denote its class in $G_0(\mathcal Z)$ under d\'evissage. Ordinary equivariant Riemann--Roch, followed by the cycle class map and denoted by $\tau_{\mathrm{BM}}$, satisfies
\begin{equation}
  \label{eq:ordinary-RR-with-support} \tau_{\mathrm{BM}}([A^\bullet]_{\mathcal Z}) =\operatorname{td}(T_{\mathcal M}|_{\mathcal Z}) \cap\ch_{\mathcal Z}^{\mathcal M}([A^\bullet]) \in H^{\mathrm{BM}}_{2*}(\mathcal Z,\CC);
\end{equation}
see \cite[Theorem~18.3]{Ful98}, \cite{EG00}, and \cite[Subsection~2.3, Equation~(23)]{MSY23}.

For $F^\bullet\in D^b\Coh(\mathcal W_{ij})$, the complex $Rj_{ij*}F^\bullet$ is perfect on $\mathcal M_{ij}$. Its derived restriction by $\iota_{ij}^{g,h}$ is perfect on $\mathcal M_{ij}^{g,h}$ and exact off $\mathcal W_{ij}^{g,h}$. Set
\[
  \ch^{\mathrm{loc}}_{g,h}(F^\bullet) :=\ch_{\mathcal W_{ij}^{g,h}}^{\mathcal M_{ij}^{g,h}} \left(t_{g,h}\left[ L(\iota_{ij}^{g,h})^*Rj_{ij*}F^\bullet \right]\right).
\]
Using~\eqref{KS's tau functor}, define the target-normalized transformation by the first equality below. Its fixed-sector expression is
\begin{equation}
  \label{eq:normalized-tau-supported-definition}
\begin{aligned}
  \tau^{\mathrm{nor}}_S([F^\bullet])_{g,h} &:=\operatorname{pr}_2^*\Theta_{j,h}^{-1} \cap\tau^{\mathrm{orb}}([F^\bullet])_{g,h} \\
  &=\operatorname{pr}_1^*\Theta_{i,g} \cap\ch_{\mathcal W_{ij}^{g,h}}^{\mathcal M_{ij}^{g,h}} \left(t_{g,h}\left[ L(\iota_{ij}^{g,h})^*Rj_{ij*}F^\bullet \right]\right)\in H^{\mathrm{BM}}_{2*}(\mathcal W_{ij}^{g,h},\CC).
\end{aligned}
\end{equation}

Formula~\eqref{eq:normalized-tau-supported-definition} is the fixed-sector form of equivariant Lefschetz--Riemann--Roch. Let $\ell_{ij}^{g,h}:\mathcal W_{ij}^{g,h}\hookrightarrow\mathcal W_{ij}$ be the sector inclusion and let $\mathfrak m_{(g,h)}$ be the maximal ideal of $R(G_i\times G_j)_{\CC}$ obtained by evaluating characters at $(g,h)$. After localization at $\mathfrak m_{(g,h)}$, equivariant concentration implies that $(\ell_{ij}^{g,h})_*$ is an isomorphism. The self-intersection formula gives its inverse:
\begin{equation}
  \label{eq:supported-concentration-inverse} \bigl((\ell_{ij}^{g,h})_*\bigr)^{-1} [F^\bullet]_{\mathfrak m_{(g,h)}} = \frac{ \bigl[L(\iota_{ij}^{g,h})^*Rj_{ij*}F^\bullet \bigr]_{\mathcal W_{ij}^{g,h}} }{ \lambda_{-1}\!\left( \operatorname{pr}_1^*N_{i,g}^{\vee} \oplus\operatorname{pr}_2^*N_{j,h}^{\vee} \right) }.
\end{equation}
The numerator is interpreted by d\'evissage, and the denominator is the equivariant $K$-theoretic Euler class of the normal bundle. Applying $t_{g,h}$, ordinary equivariant Riemann--Roch~\eqref{eq:ordinary-RR-with-support}, and the module formula of \cite[Theorem~3.1(c)]{EG00} gives \eqref{eq:normalized-tau-supported-definition}; see also \cite[Theorem~5.5]{EG08} and \cite[Definitions~6.6, 6.8, and~10.8]{KS20}. The same calculation with three factors gives \eqref{eq:normalized-orbifold-tau-binary}.

We also use the module and isomorphism pullback formulas for the localized Chern character. For $\alpha\in K_0^{\mathcal Z}(\mathcal M)_{\CC}$, $V\in K^0(\mathcal M)$, and an isomorphism $u:\mathcal M\to\mathcal M$ preserving $\mathcal Z$ scheme-theoretically, they read
\begin{align}
  \ch_{\mathcal Z}^{\mathcal M}(V\cdot\alpha) &=\ch(V|_{\mathcal Z})\cap \ch_{\mathcal Z}^{\mathcal M}(\alpha), \label{eq:localized-ch-module}\\
  (u|_{\mathcal Z})^*\ch_{\mathcal Z}^{\mathcal M}(\alpha) &=\ch_{\mathcal Z}^{\mathcal M}(u^*\alpha). \label{eq:localized-ch-isomorphism-pullback}
\end{align}
These follow from the bivariant product and pullback properties in \cite[Section~18.1]{Ful98}, applied equivariantly.

Define the spaces of cohomological correspondences by
\begin{equation}
  \label{eq:CR-correspondence-space} \mathrm{Corr}^k_{S,\mathrm{CR}}(\mathcal X_i,\mathcal X_j) :=\bigoplus_{g\in G_i,\,h\in G_j} H^{\mathrm{BM}}_{2d_{j,h}-2k+2a_j(h)-2a_i(g)} (\mathcal W_{ij}^{g,h},\CC).
\end{equation}
For binary correspondences, set
\begin{align}
  &\mathrm{Corr}^k_{S,\mathrm{CR}} (\mathcal X_1,\mathcal X_2;\mathcal X_3) \nonumber\\
  &\quad:=\bigoplus_{g_1\in G_1,\,g_2\in G_2,\,g_3\in G_3} H^{\mathrm{BM}}_{2d_{3,g_3}-2k +2a_3(g_3)-2a_1(g_1)-2a_2(g_2)} \bigl(\mathcal X_{1,g_1}\times_S\mathcal X_{2,g_2} \times_S\mathcal X_{3,g_3},\CC\bigr). \label{eq:binary-CR-correspondence-space}
\end{align}
The associated Chen--Ruan direct image complexes are
\begin{equation}
  \label{eq:CR-direct-image-complex} \mathcal H_{\mathrm{CR}}(\mathcal X_i/S) :=\bigoplus_{g\in G_i} Rh_{i,g*}\CC_{\mathcal X_{i,g}}[-2a_i(g)].
\end{equation}
Sectorwise realization gives
\begin{align}
  \mathrm{Corr}^k_{S,\mathrm{CR}}(\mathcal X_i,\mathcal X_j) &\simeq \Hom_{D_c^b(S)}\left( \mathcal H_{\mathrm{CR}}(\mathcal X_i/S), \mathcal H_{\mathrm{CR}}(\mathcal X_j/S)[2k]\right), \label{eq:CR-realization}\\
  \mathrm{Corr}^k_{S,\mathrm{CR}} (\mathcal X_1,\mathcal X_2;\mathcal X_3) &\simeq \Hom_{D_c^b(S)}\left( \mathcal H_{\mathrm{CR}}(\mathcal X_1/S) \otimes^L\mathcal H_{\mathrm{CR}}(\mathcal X_2/S), \mathcal H_{\mathrm{CR}}(\mathcal X_3/S)[2k]\right). \label{eq:binary-CR-realization}
\end{align}
These are the realizations of \cite[Subsections~2.2.2--2.2.3]{MSY23} and \cite[Appendix~A]{Yun11}, applied equivariantly to the smooth sectors. We therefore regard cohomological correspondences as morphisms in $D^b_{\mathrm{c}}(S)$.

For a binary kernel
\[
  F^\bullet\in D^b\Coh(\mathcal X_1\times_S\mathcal X_2\times_S\mathcal X_3),
\]
the same normalization gives the binary correspondence
\begin{align}
  \label{eq:normalized-orbifold-tau-binary} \tau^{\mathrm{nor}}([F^\bullet])_{g_1,g_2,g_3} =\operatorname{pr}_3^*\Theta_{3,g_3}^{-1} \cap \tau^{\mathrm{orb}}([F^\bullet])_{g_1,g_2,g_3}\in \bigoplus_{d\in\Z}\mathrm{Corr}^d_{S,\mathrm{CR}} (\mathcal X_1,\mathcal X_2;\mathcal X_3).
\end{align}

\begin{proposition}\label{Degree and realization of composition}
Let $\alpha\in \mathrm{Corr}^p_{S,\mathrm{CR}}(\mathcal X_1,\mathcal X_2)$ and $\beta\in \mathrm{Corr}^q_{S,\mathrm{CR}}(\mathcal X_2,\mathcal X_3)$. Then
\[
  \beta\circ\alpha \in \mathrm{Corr}^{p+q}_{S,\mathrm{CR}} (\mathcal X_1,\mathcal X_3).
\]
For \(g\in G_1\) and \(k\in G_3\), its \((g,k)\)-component is
\begin{equation}
  \label{eq:CR-composition-sectorwise} (\beta\circ\alpha)_{g,k} = \sum_{h\in G_2} p_{13*}\, \delta_h^! \bigl( \alpha_{g,h}\boxtimes\beta_{h,k} \bigr),
\end{equation}
where
\[
  \delta_h: \mathcal X_{1,g}\times_S\mathcal X_{2,h}\times_S\mathcal X_{3,k} \longrightarrow (\mathcal X_{1,g}\times_S\mathcal X_{2,h}) \times (\mathcal X_{2,h}\times_S\mathcal X_{3,k})
\]
is the base change of the absolute diagonal of \(\mathcal X_{2,h}\). Under~\eqref{eq:CR-realization}, \eqref{eq:CR-composition-sectorwise} realizes the composition of morphisms in \(D_c^b(S)\).
\end{proposition}

\begin{proof}
The diagonal of the smooth Deligne--Mumford stack \(\mathcal X_{2,h}\) is l.c.i. of codimension \(d_{2,h}\). Hence \(\delta_h^!\) lowers Borel--Moore degree by \(2d_{2,h}\). The degree of the summand indexed by \(h\) is
\[
\begin{aligned}
  & \bigl( 2d_{2,h}-2p+2a_2(h)-2a_1(g) \bigr) + \bigl( 2d_{3,k}-2q+2a_3(k)-2a_2(h) \bigr) -2d_{2,h} \\
  &\qquad= 2d_{3,k} -2(p+q) +2a_3(k) -2a_1(g).
\end{aligned}
\]
Since \(p_{13}\) is proper, its pushforward preserves Borel--Moore degree. This proves the first assertion. Compatibility with realization is \cite[Equation~(16)]{MSY23}; the corresponding cohomological correspondence formalism is recalled in \cite[Appendix~A]{Yun11}.
\end{proof}

\begin{proposition}\label{prop:normalized-tau-composition}
Let $F_{12}^\bullet\in D^b\Coh(\mathcal W_{12})$ and $F_{23}^\bullet\in D^b\Coh(\mathcal W_{23})$. Write $\mathcal W_{123}:=\mathcal X_1\times_S\mathcal X_2\times_S\mathcal X_3$ and let $p_{ij}$ be its projections to $\mathcal W_{ij}$. Assume that
\[
  E_{123}^\bullet :=Lp_{12}^*F_{12}^\bullet\otimes^L Lp_{23}^*F_{23}^\bullet, \qquad F_{13}^\bullet:=Rp_{13*}E_{123}^\bullet
\]
are bounded coherent. Then
\begin{equation}
  \label{eq:normalized-tau-composition} \tau^{\mathrm{nor}}_S([F_{13}^\bullet]) =\tau^{\mathrm{nor}}_S([F_{23}^\bullet]) \circ\tau^{\mathrm{nor}}_S([F_{12}^\bullet]).
\end{equation}
\end{proposition}

\begin{proof}
Set $\mathcal M_{123}:=\mathcal X_1\times\mathcal X_2\times\mathcal X_3$, write $j_{123}:\mathcal W_{123}\hookrightarrow\mathcal M_{123}$, and let $q_{ij}$ be the projections from $\mathcal M_{123}$. The partial diagonals $s_1=s_2$ and $s_2=s_3$ in $S^3$ are Tor-independent regular embeddings. Their inverse images in $\mathcal M_{123}$ remain Tor-independent because the morphism $\mathcal M_{123}\to S^3$ is flat. Thus derived base change \cite[Tag~08IB]{Stacks} and the projection formula give
\begin{equation}
  \label{eq:ambient-relative-convolution-comparison} Lq_{12}^*Rj_{12*}F_{12}^\bullet \otimes^L Lq_{23}^*Rj_{23*}F_{23}^\bullet \simeq Rj_{123*}E_{123}^\bullet
\end{equation}
with support on $\mathcal W_{123}$.

For $(g,h,k)\in G_1\times G_2\times G_3$, put
\[
  \mathcal M_{123}^{g,h,k} :=\mathcal X_{1,g}\times\mathcal X_{2,h}\times\mathcal X_{3,k}, \qquad \mathcal W_{123}^{g,h,k} :=\mathcal X_{1,g}\times_S\mathcal X_{2,h}\times_S\mathcal X_{3,k}.
\]
Define $\ch^{\mathrm{loc}}_{g,h,k}(E_{123}^\bullet)$ by applying the eigenvalue-weighted localized Chern character to the derived restriction of $Rj_{123*}E_{123}^\bullet$ to $\mathcal M_{123}^{g,h,k}$, with support on $\mathcal W_{123}^{g,h,k}$. The operation $t_g$, defined by weighting eigensheaves, commutes with derived pullback and is multiplicative for derived tensor products; see also \cite[Proposition~5.4]{EG08}. Let $\delta_h^!$ be the refined Gysin pullback of \eqref{eq:CR-composition-sectorwise}, induced by the absolute diagonal of $\mathcal X_{2,h}$ and taking values on $\mathcal W_{123}^{g,h,k}$.

The product and pullback formulas for the bivariant localized Chern character, as in \cite[Section~18.1]{Ful98} and \cite[Tags~0FBF, 0FB3, and~0FBB]{Stacks}, applied to \eqref{eq:ambient-relative-convolution-comparison}, yield
\begin{equation}
  \label{eq:localized-chern-composition} \ch^{\mathrm{loc}}_{g,h,k}(E_{123}^\bullet) =\delta_h^!\left( \ch^{\mathrm{loc}}_{g,h}(F_{12}^\bullet) \boxtimes \ch^{\mathrm{loc}}_{h,k}(F_{23}^\bullet) \right).
\end{equation}

The map $p_{13}:\mathcal W_{123}\to\mathcal W_{13}$ is proper, being a base change of $\mathcal X_2\to S$. The relative products are separated finite quotient stacks with quasi-projective coarse moduli spaces, since their atlases are closed in the corresponding quasi-projective absolute products. Thus \eqref{eq:orbifold-tau-proper-covariance} applies to $p_{13}$. Its inertia map sends $(g,h,k)$ to $(g,k)$; using $p_{13}$ also for the induced sector morphisms, we obtain
\[
  \tau^{\mathrm{orb}}_S([F_{13}^\bullet])_{g,k} =\sum_{h\in G_2}p_{13*} \tau^{\mathrm{orb}}_S([E_{123}^\bullet])_{g,h,k}.
\]
Substitute the fixed-sector expressions and cancel the invertible outer $\Theta$-factors using the projection formula. Writing $q_2$ for the projection to $\mathcal X_{2,h}$, we obtain
\begin{equation}
  \label{eq:supported-fixed-point-pushforward} \ch^{\mathrm{loc}}_{g,k}(F_{13}^\bullet) =\sum_{h\in G_2}p_{13*}\left( q_2^*\Theta_{2,h}\cap \ch^{\mathrm{loc}}_{g,h,k}(E_{123}^\bullet) \right)
\end{equation}
in $H^{\mathrm{BM}}_{2*}(\mathcal W_{13}^{g,k},\CC)$.

Multiplying by $\operatorname{pr}_1^*\Theta_{1,g}$ and substituting \eqref{eq:localized-chern-composition}, we conclude that
\[
  \tau^{\mathrm{nor}}_S([F_{13}^\bullet])_{g,k} =\sum_{h\in G_2}p_{13*}\delta_h^!\left( \tau^{\mathrm{nor}}_S([F_{12}^\bullet])_{g,h} \boxtimes \tau^{\mathrm{nor}}_S([F_{23}^\bullet])_{h,k} \right).
\]
Indeed, the middle factor $\Theta_{2,h}$ is the source normalization of the second kernel. By \eqref{eq:CR-composition-sectorwise}, the last expression is the composition of cohomological correspondences; see \cite[Subsection~2.2.2, Equation~(16)]{MSY23}. Summing over $g,k$ proves the proposition.
\end{proof}

\begin{lemma}\label{lem:normalized-diagonal-is-identity}
Let $\Delta_{\mathcal X_i/S}: \mathcal X_i \longrightarrow \mathcal X_i\times_S\mathcal X_i$ be the relative diagonal. Then
\begin{equation}
  \label{eq:normalized-diagonal-class} \tau^{\mathrm{nor}} \bigl( [\sO_{\Delta_{\mathcal X_i/S}}] \bigr) = \mathbf 1_{\mathcal X_i}.
\end{equation}
\end{lemma}

\begin{proof}
The induced map on inertia sends the \(g\)-sector diagonally to the \((g,g)\)-sector. Proper covariance gives
\[
  \tau^{\mathrm{orb}} \bigl( [\sO_{\Delta_{\mathcal X_i/S}}] \bigr)_{g,g} = (\Delta_{\mathcal X_{i,g}/S})_* \left( \Theta_{i,g}\cap[\mathcal X_{i,g}] \right),
\]
and all components with unequal sector labels vanish. The target normalization and the projection formula cancel \(\Theta_{i,g}\), leaving
\[
  (\Delta_{\mathcal X_{i,g}/S})_*[\mathcal X_{i,g}].
\]
Summing over \(g\) proves~\eqref{eq:normalized-diagonal-class}.
\end{proof}

\begin{lemma}\label{lem:support-bound-normalized-tau}
Let
\[
  F^\bullet\in D^b\Coh(\mathcal X_i\times_S\mathcal X_j),
\]
and fix \(g\in G_i\) and \(h\in G_j\). Suppose that the union of the supports of the cohomology sheaves of \(L(\iota_{ij}^{g,h})^*Rj_{ij*}F^\bullet\) has dimension at most \(d_F\). Then
\begin{align}
  \tau^{\mathrm{nor}}([F^\bullet])_{g,h} \in \bigoplus_{ k\geq d_{j,h}+a_j(h)-a_i(g)-d_F} \mathrm{Corr}^k_{S,\mathrm{CR}} (\mathcal X_{i,g},\mathcal X_{j,h}). \label{eq:support-bound-pair}
\end{align}
For a binary kernel whose supported derived restriction in the smooth absolute product to the \((g_1,g_2,g_3)\)-sector has support of dimension at most \(d_F\), we have
\begin{align}
  \tau^{\mathrm{nor}}([F^\bullet])_{g_1,g_2,g_3} \in \bigoplus_{ k\geq d_{3,g_3} +a_3(g_3)-a_1(g_1)-a_2(g_2)-d_F} \mathrm{Corr}^k_{S,\mathrm{CR}} (\mathcal X_{1,g_1},\mathcal X_{2,g_2};\mathcal X_{3,g_3}). \label{eq:support-bound-binary}
\end{align}
\end{lemma}

\begin{proof}
The localized Chern character of a supported perfect complex is a cycle class on its support. Hence it has no Borel--Moore component of degree greater than \(2d_F\). Multiplication by an orbifold Todd class does not increase Borel--Moore degree.

A class in the \((g,h)\)-component of \(\mathrm{Corr}^d_{S,\mathrm{CR}}(\mathcal X_i,\mathcal X_j)\) has Borel--Moore degree
\[
  2d_{j,h}-2d+2a_j(h)-2a_i(g).
\]
If this component is nonzero, then
\[
  2d_{j,h}-2d+2a_j(h)-2a_i(g)\leq 2d_F,
\]
which is equivalent to~\eqref{eq:support-bound-pair}. The binary statement follows from the same calculation.
\end{proof}

\subsection{Fourier correspondences and character decomposition}\label{subsec:Fourier-character-decomposition}

We apply the preceding construction to
\[
  \check{P} \qquad \text{and} \qquad [\check{P}/\Gamma],
\]
over $B$. The inertia stacks are given by
\[
  I(\check{P}\times_B[\check{P}/\Gamma])=\bigsqcup_{\gamma\in \Gamma}\check{P}\times_B[\check{P}_\gamma/\Gamma]\qquad I[\check{P}/\Gamma]=\bigsqcup_{\gamma\in \Gamma}[\check{P}_\gamma/\Gamma].
\]

\begin{lemma}\label{lem:Poincare-sector-actual-pullback}
Use the notation above with $S=B$, $\mathcal X_1=\check P$, and $\mathcal X_2=[\check P/\Gamma]$. For $\gamma\in\Gamma$, let
\[
  \ell_\gamma: \check P\times_B[\check P_\gamma/\Gamma] \longrightarrow \check P\times_B[\check P/\Gamma]
\]
be the natural inclusion. Then there is a canonical isomorphism
\begin{equation}
  \label{eq:Poincare-sector-base-change} L(\iota_{12}^{1,\gamma})^*Rj_{12*}\mathscr F \simeq R\jmath_{12*}^{1,\gamma}L\ell_\gamma^*\mathscr F \quad\text{in }D^b\Coh(\check P\times [\check P_\gamma/\Gamma]).
\end{equation}
Consequently,
\begin{equation}
  \label{components of tau of F} \tau^{\mathrm{nor}}_{B}(\mathscr{F})_{1, \gamma}=\tau^{\mathrm{nor}}_B(L\ell_\gamma^*\mathscr F)_{1, \gamma}.
\end{equation}
\end{lemma}

\begin{proof}
The projection $p_2:\mathcal W_{12}\to[\check P/\Gamma]$ is flat, being a base change of $\check P\to B$. On the quotient atlas, the inclusion $\check P_\gamma\hookrightarrow\check P$ is regular, since both spaces are smooth. Its local defining regular sequence remains regular after pullback to $\mathcal W_{12}$ by flatness. The corresponding Koszul resolution therefore shows that the Cartesian square
\[
\begin{tikzcd}
  \mathcal W_{12}^{1,\gamma} \arrow[r,"\ell_\gamma"] \arrow[d,"\jmath_{12}^{1,\gamma}"'] & \mathcal W_{12}\arrow[d,"j_{12}"] \\
  \mathcal M_{12}^{1,\gamma} \arrow[r,"\iota_{12}^{1,\gamma}"'] & \mathcal M_{12}
\end{tikzcd}
\]
is Tor-independent; it also shows that $\ell_\gamma$ is regular. Derived base change gives \eqref{eq:Poincare-sector-base-change}; see \cite[Tags~00LM and~08IB]{Stacks}. This construction is equivariant on the quotient atlas.
\end{proof}

We define
\begin{align}
  \mathfrak{F} =\tau^{\mathrm{nor}}([\mathscr{F}])&\in \bigoplus_{i\in \mathbb{Z}}\operatorname{Corr}^{i}_{B,\rm CR}(\check{P},[\check{P}/\Gamma]) \\
  \mathfrak{F}^{-1} =\tau^{\mathrm{nor}}([\mathscr{F}^{-1}])&\in \bigoplus_{i\in \mathbb{Z}}\operatorname{Corr}^{i}_{B,\rm CR}([\check{P}/\Gamma], \check{P}).
\end{align}
For any $\gamma\in \Gamma$ and $i\in \mathbb{Z}$, we index their components by shifting the correspondence degree by $g_P$:
\begin{align}
  \mathfrak{F}_{\gamma, i}=(\mathfrak{F})_{\gamma, i}&\in \operatorname{Corr}^{i-g_P}_{B,\rm CR}(\check{P},[\check{P}_\gamma/\Gamma]) \\
  \mathfrak{F}_{\gamma, i}^{-1}=(\mathfrak{F}^{-1})_{\gamma, i}&\in \operatorname{Corr}^{i-g_P}_{B,\rm CR}([\check{P}_\gamma/\Gamma], \check{P}).
\end{align}

Here $\gamma$ labels an inertia sector on the quotient side, and $\mathfrak F_{\gamma,i}^{-1}$ denotes the component $(\mathfrak F^{-1})_{\gamma,i}$ of the inverse transform.

We also write
\begin{equation}
  \mathfrak{F}_i=\sum_{\gamma \in \Gamma}\mathfrak{F}_{\gamma, i} \qquad \mathfrak{F}_i^{-1}=\sum_{\gamma \in \Gamma}\mathfrak{F}_{\gamma, i}^{-1}.
\end{equation}

\begin{proposition}\label{inverse to each other}
We have
\[
  \mathfrak{F}^{-1}\circ \mathfrak{F}=\mathrm{id} \qquad \mathfrak{F}\circ \mathfrak{F}^{-1}=\mathrm{id}
\]
\end{proposition}

\begin{proof}
Apply Proposition~\ref{prop:normalized-tau-composition} and Lemma~\ref{lem:normalized-diagonal-is-identity} to the inverse Fourier--Mukai kernels.
\end{proof}

Let $q:\check P\longrightarrow[\check P/\Gamma]$ be the quotient map. For \(\delta\in\Gamma\), let $T_\delta:\check P\longrightarrow\check P$ denote tensor product by \(\pi^*\delta\). We can also regard \(\pi^*\delta\) as a section of \(\check P\to B\) and set
\[
  \mathscr L_\delta := L(\pi^*\delta\times\id)^*\mathscr F \in \operatorname{Pic}([\check P/\Gamma]).
\]
The derived pullback is an ordinary pullback, since the Poincaré sheaf is flat over the first factor.

\begin{lemma}\label{the scalar on trivial line bundle}
The pullback \(q^*\mathscr L_\delta\) is the trivial line bundle on \(\check P\). Under the \(\Gamma\)-linearization induced by \(\mathscr F\), the action of an element \(\eta\in\Gamma\) on this trivial line bundle is multiplication by the scalar
\begin{equation}
  \label{eq:Poincare-character-line} \langle\delta,\eta\rangle_\Gamma.
\end{equation}
\end{lemma}

\begin{proof}
Pulling \(\mathscr L_\delta\) back along \(q\) gives
\[
  q^*\mathscr L_\delta \simeq \mathscr R \big|_{\{\pi^*\delta\}\times_B\check P}.
\]
For \(F\in\check P\), we have \(\operatorname{Nm}(F)\simeq\sO_\Sigma\). The compatibility of the determinant-of-cohomology construction with pullback and norm therefore identifies the last restriction with the trivial line bundle. The normalization of \(\mathscr P\) fixes its trivialization along the identity section.

To determine its linearization, let $\mathscr L_\delta^\Sigma$ be the normalized Poincar\'e line on $J_\Sigma^0$ corresponding to $\delta$. The determinant-of-cohomology norm identity of \cite[Proposition~4.4]{GS22} identifies the ambient slice with $\Nm^*\mathscr L_\delta^\Sigma$. Its pullback by $\eta$ is therefore the pullback from $J_\Sigma^0$ of $[n]^*\mathscr L_\delta^\Sigma$. The normalized trivialization of this last line has deck character $\eta\mapsto\langle\eta,\delta\rangle_\Gamma$, by Cartier duality for the multiplication-by-$n$ isogeny. In the quotient presentation of Lemma~\ref{etale cover of J bar}, an element $\eta\in\Gamma$ acts on $J_\Sigma^0\times\check P$ by $(L_B,F)\mapsto(L_B\otimes\eta^{-1},F\otimes\pi^*\eta)$. The character on the trivialized slice over $\check P$ is consequently $\langle\eta^{-1},\delta\rangle_\Gamma =\langle\delta,\eta\rangle_\Gamma$. This proves \eqref{eq:Poincare-character-line} for the linearization fixed by descent.
\end{proof}

\begin{lemma}\label{lem:Poincare-sector-character} \label{sectorwise isomorphism}
The normalization of the Poincaré sheaf determines an isomorphism
\begin{equation}
  \label{eq:Poincare-translation} (T_\delta\times\id)^*\mathscr F \simeq \mathscr F\otimes p_2^*\mathscr L_\delta .
\end{equation}
Consequently, for every \(\gamma\in\Gamma\) and every \(i\),
\begin{equation}
  \label{eq:Poincare-sector-eigenclass} \mathfrak F_{\gamma,i} \circ T_{\delta^{-1}}^* = \kappa_\gamma(\delta)^{-1}\mathfrak F_{\gamma,i}.
\end{equation}
It follows that the morphism realized by $\mathfrak F_{\gamma,i}$ vanishes on
\[
  \bigoplus_{\kappa\neq\kappa_\gamma} (R\check h_*\CC_{\check P})_\kappa,
\]
and the image of the morphism realized by $\mathfrak F_{\gamma,i}^{-1}$ is contained in $(R\check h_*\CC_{\check P})_{\kappa_\gamma}$.
\end{lemma}

\begin{proof}
Arinkin's theorem of the square \cite[Lemma~6.5]{Ari13}, applied to the line bundle \(\pi^*\delta\), gives
\[
  (T_\delta\times\id)^*\mathscr P_1 \simeq \mathscr P_1 \otimes p_2^* L(\pi^*\delta\times\id)^*\mathscr P .
\]
The \(J_\Sigma^0\)-equivariant structure on \(\mathscr P_1\) is constructed from the determinant-of-cohomology isomorphism of \cite[Proposition~4.4]{GS22}. To check compatibility with this equivariant structure, consider the union of the loci where at least one factor represents a line bundle. The determinant-of-cohomology construction gives the compatible isomorphism on this union, as in \cite[proof of Proposition~4.4]{GS22}. Its complement has codimension at least two: each non-locally-free locus has fiberwise codimension at least one, and the relative products are flat over $B$. Uniqueness of the maximal Cohen--Macaulay extension gives the equivariant isomorphism on the whole product. It descends to \eqref{eq:Poincare-translation}. Moreover, $T_\delta$ is a $B$-automorphism of the first factor, so $T_\delta\times\id$ commutes with $\ell_\gamma$. Applying $L\ell_\gamma^*$ to \eqref{eq:Poincare-translation} gives the corresponding sectorwise translation formula, since $\mathscr L_\delta$ is a line bundle.

For an equivariant bundle $E$, write $\widetilde{\ch}(E)_\gamma=\ch(t_\gamma([E|_{[\check P_\gamma/\Gamma]}]))$. Here $t_\gamma$ is the eigenvalue-weighting operation defined in Subsection~\ref{subsec:cohomological-functors}. By Lemma~\ref{the scalar on trivial line bundle}, the inertial Chern character of \(\mathscr L_\delta\) on the \(\gamma\)-sector is
\[
  \widetilde{\ch}(\mathscr L_\delta)_\gamma = \langle\delta,\gamma\rangle_\Gamma.
\]

The automorphism \(T_\delta\times\id\) extends to the smooth ambient product and preserves the relative product scheme-theoretically. Equations~\eqref{eq:localized-ch-isomorphism-pullback} and \eqref{eq:localized-ch-module}, together with \eqref{eq:Poincare-translation}, give
\[
\begin{aligned}
  (T_\delta\times\id)^* \ch^{\mathrm{loc}}_{1,\gamma}(\mathscr F) &= \ch^{\mathrm{loc}}_{1,\gamma} \bigl((T_\delta\times\id)^*\mathscr F\bigr)\\
  &= \operatorname{pr}_2^* \widetilde{\ch}(\mathscr L_\delta)_\gamma \cap \ch^{\mathrm{loc}}_{1,\gamma}(\mathscr F)\\
  &= \langle\delta,\gamma\rangle_\Gamma \ch^{\mathrm{loc}}_{1,\gamma}(\mathscr F).
\end{aligned}
\]
Since \(T_\delta^*T_{\check P}\simeq T_{\check P}\), the source Todd factor in the definition of \(\tau^{\mathrm{nor}}\) is invariant under \(T_\delta\). Hence
\[
  (T_\delta\times\id)^* \tau^{\mathrm{nor}}([\mathscr F])_\gamma = \langle\delta,\gamma\rangle_\Gamma \tau^{\mathrm{nor}}([\mathscr F])_\gamma= \kappa_\gamma(\delta)^{-1} \tau^{\mathrm{nor}}([\mathscr F])_\gamma.
\]
Under the realization of cohomological correspondences, pullback along \(T_\delta\times\id\) corresponds to precomposition by \(T_{\delta^{-1}}^*\), which proves \eqref{eq:Poincare-sector-eigenclass}.

For $\kappa\in\widehat\Gamma$, let
\[
  e_\kappa= \frac{1}{|\Gamma|}\sum_{\delta\in\Gamma} \kappa(\delta)^{-1}T_\delta^*
\]
be the projector onto $(R\check h_*\CC_{\check P})_\kappa$. Equation~\eqref{eq:Poincare-sector-eigenclass} gives
\[
  \mathfrak F_{\gamma,i}\circ e_\kappa =
\begin{cases}
  \mathfrak F_{\gamma,i},&\kappa=\kappa_\gamma,\\
  0,&\kappa\neq\kappa_\gamma.
\end{cases}
\]
Let $q_\gamma$ be the projector onto the $\gamma$-sector of $\mathcal H_{\mathrm{CR}}([\check P/\Gamma]/B)$. By the nondegeneracy of the Weil pairing, $q_\gamma\mathfrak F=\mathfrak F e_{\kappa_\gamma}$. Proposition~\ref{inverse to each other} then gives
\[
  \mathfrak F^{-1}q_\gamma =e_{\kappa_\gamma}\mathfrak F^{-1}.
\]
Taking homogeneous components proves the assertion for $\mathfrak F_{\gamma,i}^{-1}$.
\end{proof}

\subsection{Codimension estimates}\label{subsec:Prym-codimension-estimates}

From this point to the end of this section, assume $d>2g-2$.

For an integral spectral curve $\mathbf C_b^D$, let $\tilde{\mathbf C_b^D}$ be its normalization and set
\[
  \delta(b) = g_a(\mathbf C_b^D) - g(\tilde{\mathbf C_b^D}).
\]

For an irreducible subvariety $Z$ of the integral locus, let $\delta_Z$ be $\delta(b)$ for general $b\in Z$. For $d>2g-2$, \cite[Proposition~4.5.1]{deC17} gives the Severi inequality
\begin{equation}
  \label{delta inequality} \operatorname{codim}_{\mathbf B^D}( Z) \ge \delta_{ Z},
\end{equation}
for every such $Z$.

The inequality is unchanged after the \'etale base change $B\to\mathbf B$.

\begin{lemma}\label{lem:pull back to product}
For a geometric point $F$ of $[\check{P}/\Gamma]$, let $b$ be its image in $B$. Then the restriction $\mathscr{F}_F$ to $\check{P}_b$, pulled back along
\[
  p_2: J_\Sigma^0\times \check{P}_b\longrightarrow \check{P}_b,
\]
is isomorphic to the pullback of $\mathscr{P}_{j(F^\prime)}$ along
\[
  \eta_b: J_\Sigma^0\times \check{P}_b\longrightarrow \overline{J}_{C_b}^0,
\]
where $F^\prime\in\check P_b$ is any lift of $F$ along $\check P_b\to[\check P_b/\Gamma]$.
\end{lemma}

\begin{proof}
Let \(\mathscr P_2\) be the restriction of the normalized Poincaré sheaf to \(\overline J_C^0\times_B\check P\). Symmetry of the normalized Poincaré sheaf, together with the equivariant structure used in its descent, equips \(\mathscr P_2\) with a \(J_\Sigma^0\)-equivariant structure for tensor product on the first factor. Pulling this equivariant sheaf back along \(\eta\times\id\) and restricting the second factor to a lift \(F'\in\check P_b\) gives
\[
  p_2^*(\mathscr F_F) \simeq \eta_b^*(\mathscr P_{j(F')}).
\]
Changing \(F'\) by an element of \(\Gamma\) changes both sides by the same linearization isomorphism. The cocycle condition therefore makes the identification independent of the chosen lift.
\end{proof}

Following \cite[Subsection~3.3]{MSY23}, write $[\check P/\Gamma]_B^k$ for the relative $k$-fold product over $B$. Let $\hat p_{1,\ldots,k}$ be the projection from $[\check P/\Gamma]_B^k\times_B\check P$ to the first $k$ factors, and let $\hat p_{i,k+1}$ be the projection to $[\check P/\Gamma]\times_B\check P$, in that order. Write $\hat p_{k+1,i}$ for the same projection with the two factors interchanged. Define
\[
  \mathscr K_k:=R\hat p_{1,\ldots,k,*} \left(\bigotimes_{i=1}^k{}^L\hat p_{k+1,i}^*\mathscr F\right).
\]
We also use mixed versions in which any factor $L\hat p_{k+1,i}^*\mathscr F$ is replaced by $L\hat p_{i,k+1}^*\mathscr F^{-1}$. All tensor products are derived. Since the factors are pulled back from separate copies of $[\check P/\Gamma]$, these are relative external tensor products over the common $\check P$-factor. The flatness of $\mathscr F$ and $\mathscr F^{-1}[-g_P]$ over that factor makes these products bounded coherent, and the projection $\hat p_{1,\ldots,k}$ is proper. In the fiberwise calculations below, auxiliary tensor products on a singular fiber are taken in $D^-\Coh$.

As in \cite[Subsection~3.3]{MSY23}, define
\[
  \mathscr K :=\mathscr F\circ\sO_{\Delta^{\mathrm{sm}}_{\check P/B}} \circ(\mathscr F^{-1}\boxtimes\mathscr F^{-1}) \in D^b\Coh\bigl([\check P/\Gamma] \times_B[\check P/\Gamma]\times_B[\check P/\Gamma]\bigr).
\]
Here $\Delta^{\mathrm{sm}}_{\check P/B}$ is the small relative diagonal in $\check P\times_B\check P\times_B\check P$. Thus $\mathscr K$ transports the tensor product on $\check P$ to the dual side. Let $p_{123}$, $p_{14}$, $p_{24}$, and $p_{43}$ denote the indicated projections from $[\check P/\Gamma]^3\times_B\check P$, with the order of the factors retained in the subscripts. Expanding the convolution gives
\begin{equation}
  \label{eq:convolution-kernel-four-factor-form} \mathscr K\simeq Rp_{123*}\left( Lp_{14}^*\mathscr F^{-1}\otimes^L Lp_{24}^*\mathscr F^{-1}\otimes^L Lp_{43}^*\mathscr F\right).
\end{equation}

\begin{proposition}\label{prop:codimQ}
For $\gamma_1,\ldots,\gamma_k\in\Gamma$, the restriction of $\mathscr K_k$ or any of its mixed versions to
\[
  [\check P_{\gamma_1}/\Gamma]\times_B\cdots \times_B[\check P_{\gamma_k}/\Gamma]
\]
has support of dimension at most
\[
  \dim B+(k-1)g_P-\sum_{i=1}^k a_{\gamma_i}.
\]
The same bound holds for the union of the supports of the cohomology sheaves of the corresponding derived restriction in the smooth absolute product.
\end{proposition}

\begin{proof}
It is enough to estimate the union of the supports of the cohomology sheaves of the restriction of \(\mathscr K_k\). Let \(S\) be an irreducible component of this support, and let \(Z\subset B\) be the closure of its image. We work over a general point \(b\in Z\).

Let \(\widetilde C_b\) be the normalization of \(C_b\). If \((F_1,\ldots,F_k)\in S_b\), proper base change and the definitions of \(\mathscr K_k\) and its mixed versions give
\[
  \mathbb H^*\left( \check P_b, \mathscr F_{F_1}^{\epsilon_1}\otimes^L\cdots\otimes^L\mathscr F_{F_k}^{\epsilon_k} \right)\neq0,
\]
for suitable \(\epsilon_i\in\{1,-1\}\). Pulling back to $J_\Sigma^0\times\check P_b$ and pushing forward along the finite étale map
\[
  \eta_b:J_\Sigma^0\times\check P_b\longrightarrow\overline J_{C_b}^0
\]
and using Lemmas~\ref{etale cover of J bar} and~\ref{lem:pull back to product}, we obtain lifts \(F_i'\) and a line bundle \(L\in\Gamma\) such that
\[
  \mathbb H^*\left( \overline J_{C_b}^0, \mathscr P_{F_1'}^{\epsilon_1}\otimes^L\cdots\otimes^L \mathscr P_{F_k'}^{\epsilon_k}\otimes\Nm^*L \right)\neq0
\]
Here $\Nm^*L$ is a line bundle on the compactified Jacobian; its associated line bundle on $C_b$ is $\pi_b^*L$, where $\pi_b:C_b\to\Sigma$ is the spectral map. The argument of \cite[Corollary~7.3 and Proposition~7.4]{Ari13} therefore gives
\begin{equation}
  \label{eq:mixed-Poincare-relation-on-regular-locus} \bigotimes_{i=1}^k \left(F_i'|_{C_b^{\mathrm{reg}}}\right)^{\epsilon_i} \simeq \pi_b^*L^{-1}|_{C_b^{\mathrm{reg}}}.
\end{equation}

Let $\check{P}_b^\circ$ be the open subset parametrizing line bundles on $C_b$. Tensor product gives an action $\mu: \check{P}_b^\circ\times [\check{P}_{\gamma_1}/\Gamma]_b \longrightarrow [\check{P}_{\gamma_1}/\Gamma]_b$, and hence a map
\[
  \mu \times \mathrm{id}:\check{P}_b^\circ\times \bigl([\check{P}_{\gamma_1}/\Gamma]_b\times \cdots \times [\check{P}_{\gamma_k}/\Gamma]_b\bigr) \longrightarrow [\check{P}_{\gamma_1}/\Gamma]_b\times \cdots \times [\check{P}_{\gamma_k}/\Gamma]_b
\]
Since $\mu\times\mathrm{id}$ is flat, to bound $\dim S_b$ it suffices to consider the dimension of its inverse image
\begin{equation}
  \label{inverse image} (\mu \times \mathrm{id})^{-1}\Bigl(S\cap \bigl([\check{P}_{\gamma_1}/\Gamma]_b\times \cdots \times [\check{P}_{\gamma_k}/\Gamma]_b\bigr)\Bigr),
\end{equation}
whose points are tuples $(\tilde{L}; F_1, \cdots, F_k)$ satisfying, by the preceding argument,
\[
  \tilde{L}^{\epsilon_1}\otimes\bigotimes_{i=1}^k \left(F_i'|_{C_b^{\mathrm{reg}}}\right)^{\epsilon_i} \simeq \pi_b^*L^{-1}|_{C_b^{\mathrm{reg}}}.
\]
For fixed $F_1,\ldots,F_k$, the pullback of $\tilde L^{\epsilon_1}$ to $\tilde{C_b}$ differs from a fixed line bundle by $\sO_{\tilde{C_b}}(E)$, where $E$ is supported over the singular points of $C_b$. Thus $\tilde L$ lies in a countable union of translates of the affine part of $\check{P}_b^\circ$, which has dimension at most $g_a(C_b)-g(\tilde{C_b})$. This also bounds the dimension of each fiber of the projection from~\eqref{inverse image} to $[\check{P}_{\gamma_1}/\Gamma]_b\times \cdots \times [\check{P}_{\gamma_k}/\Gamma]_b$, and $\dim \check{P}_b^\circ=g_a(C_b)-g$. Thus
\begin{equation}
  \label{eq:prop-codimQ-fiber-bound} \operatorname{codim}(S_b) \geq g(\widetilde C_b)-g.
\end{equation}
The same calculation is valid after the finite quotient by \(\Gamma\).

By Lemma~\ref{lem:H-fixed-dimensions}, the fiber of the indicated product has dimension at most $\sum_{i=1}^k(g_P-a_{\gamma_i})$. Combining this with \eqref{eq:prop-codimQ-fiber-bound} and the Severi inequality on $B$ gives
\begin{align*}
  \dim S &=\dim Z+\dim S_b\\
  &\leq\dim Z+\sum_{i=1}^k(g_P-a_{\gamma_i}) -\bigl(g(\widetilde C_b)-g\bigr)\\
  &\leq\dim B-\delta(b)+kg_P-\sum_{i=1}^k a_{\gamma_i} -\bigl(g_P-\delta(b)\bigr)\\
  &=\dim B+(k-1)g_P-\sum_{i=1}^k a_{\gamma_i}.
\end{align*}
Here $g(\widetilde C_b)-g=g_P-\delta(b)$ by definition. Derived pullback to a sector cannot enlarge the set-theoretic intersection of the support with that sector, so the same estimate holds for the supported derived restriction in the smooth absolute product.
\end{proof}

For Fourier vanishing (FV), we also use the following analogue of \cite[Corollary 3.5]{MSY23}.

Consider the closed embedding $\iota: [\check{P}/\Gamma] \times_B \check{P} \hookrightarrow [\check{P}/\Gamma] \times \check{P}$. For a positive integer $N$, the complex
\[
  (\iota_*\mathscr{F}^{-1})^{\otimes^L N}\in D^b\operatorname{Coh}([\check{P}/\Gamma] \times \check{P})
\]
is bounded and exact off $[\check{P}/\Gamma] \times_B \check{P} \subset [\check{P}/\Gamma] \times \check{P}$. Define
\[
  \tilde{\mathscr{K}}(N)=\mathscr{F}\circ (\iota_*\mathscr{F}^{-1})^{\otimes^L N} \in D^b\operatorname{Coh}([\check{P}/\Gamma]\times [\check{P}/\Gamma]).
\]
Here convolution uses the supported perfect complexes on the smooth absolute products. The projection on this support is a base change of the proper map $\check P\to B$, so the resulting convolution is bounded coherent.

\begin{proposition}\label{codim for FV}
For every positive integer $N$ and $\gamma_1,\gamma_2\in\Gamma$, the restriction of $\widetilde{\mathscr K}(N)$ to $[\check P_{\gamma_1}/\Gamma]\times[\check P_{\gamma_2}/\Gamma]$ has support of dimension at most
\[
  \dim B+g_P-a_{\gamma_1}-a_{\gamma_2}.
\]
Its support is contained in the corresponding relative product over $B$.
\end{proposition}

\begin{proof}
Let $S$ be an irreducible component of the support of the restriction of $\widetilde{\mathscr K}(N)$ to the $(\gamma_1,\gamma_2)$-sector, and let $Z\subset B$ be the closure of its image. Fix a general point $b\in Z$ and $(F_1,F_2)\in S_b$. Write $\iota_b:\check P_b\hookrightarrow\check P$ for the fiber inclusion. Proper base change gives
\[
  \mathbb H^*\left( \check P_b, \bigl(L\iota_b^*\iota_{b*}(\mathscr F^{-1})_{F_1}\bigr)^{\otimes^L N} \otimes^L\mathscr F_{F_2} \right)\neq0.
\]
The immersion $\iota_b$ is regular, since $\check P\to B$ is flat and $B$ is smooth. Its normal bundle is trivial. The Koszul calculation, or \cite[proof of Proposition~3.3, Step~3]{MSY23}, therefore gives a finite truncation filtration of $L\iota_b^*\iota_{b*}(\mathscr F^{-1})_{F_1}$ whose graded pieces are finite sums of shifts of $(\mathscr F^{-1})_{F_1}$. Tensor this filtration $N$ times in $D^-\Coh(\check P_b)$ and then with $\mathscr F_{F_2}$. If all hypercohomology of $(\mathscr F^{-1})_{F_1}^{\otimes^L N}\otimes^L\mathscr F_{F_2}$ vanished, the finite sequence of distinguished triangles would force the preceding hypercohomology to vanish. Hence
\[
  \mathbb H^q\left( \check P_b, (\mathscr F^{-1})_{F_1}^{\otimes^L N} \otimes^L\mathscr F_{F_2} \right)\neq0
\]
for some $q$. This group is finite-dimensional: the complex is bounded above with coherent cohomology and $\check P_b$ is proper.

Pull back to $J_\Sigma^0\times\check P_b$ and push forward along $\eta_b$. Lemmas~\ref{etale cover of J bar} and~\ref{lem:pull back to product}, also applied to the dual kernel, give lifts $F_1',F_2'$ and $L\in\Gamma$ with nonzero hypercohomology on $\overline J_{C_b}^0$. The central-extension argument of \cite[Proposition~7.2 and Corollary~7.3]{Ari13} applies to a nonzero finite-dimensional hypercohomology group and gives
\[
  \left(F_1'|_{C_b^{\mathrm{reg}}}\right)^{-N} \otimes F_2'|_{C_b^{\mathrm{reg}}} \simeq\pi_b^*L|_{C_b^{\mathrm{reg}}}.
\]
Multiplication by $N$ on the Prym of $\widetilde C_b$ is finite onto its image. The action argument in the proof of Proposition~\ref{prop:codimQ} consequently gives $\operatorname{codim}(S_b)\geq g(\widetilde C_b)-g$. Lemma~\ref{lem:H-fixed-dimensions} and \eqref{delta inequality} then yield
\begin{align*}
  \dim S &=\dim Z+\dim S_b\\
  &\leq\dim B-\delta(b)+2g_P-a_{\gamma_1}-a_{\gamma_2} -\bigl(g(\widetilde C_b)-g\bigr)\\
  &=\dim B+g_P-a_{\gamma_1}-a_{\gamma_2}.
\end{align*}
The support is contained in the relative product by the definition of the convolution.
\end{proof}

We identify supported $K$-theory on the smooth absolute product with $G_0$ of the relative product by dévissage. Convolution with $\mathscr F$ is defined on this $G_0$-group: flatness of $\mathscr F$ over the intermediate factor makes convolution with each coherent sheaf bounded, and the projection on its support is proper. Equation~\eqref{eq:ambient-relative-convolution-comparison} identifies it with ambient convolution. Hence Proposition~\ref{prop:normalized-tau-composition} extends by additivity to these supported classes, including the Adams classes below; compare \cite[Subsection~3.5.3]{MSY23}.

Let
\[
  \psi^N: K_{[\check P/\Gamma]\times_B\check P} \bigl([\check P/\Gamma]\times\check P\bigr) \longrightarrow K_{[\check P/\Gamma]\times_B\check P} \bigl([\check P/\Gamma]\times\check P\bigr)
\]
be the localized Adams operation, and set
\[
  \mathscr A_N := [\mathscr F]\circ \psi^N\bigl([\iota_*\mathscr F^{-1}]\bigr).
\]

\begin{lemma}\label{lem:Adams-support-for-FV}
For $\gamma_1,\gamma_2\in\Gamma$, the restriction of $\mathscr A_N$ to the $(\gamma_1,\gamma_2)$-sector belongs to the subgroup of supported $K$-theory generated by classes with support of dimension at most $\dim B+g_P-a_{\gamma_1}-a_{\gamma_2}$.
\end{lemma}

\begin{proof}
Represent $\iota_*\mathscr F^{-1}$ by a supported perfect complex. Over $\CC$, the antisymmetrizing idempotent realizes every derived exterior power as a direct summand of the corresponding derived tensor power. Thus each product of exterior powers of total degree $N$ is a direct summand of its $N$-fold derived tensor power. Convolution with $\mathscr F$ and derived restriction to a sector preserve these direct summands. Proposition~\ref{codim for FV} therefore gives the required support bound for each such summand. The Newton identities express the Adams class as an integral linear combination of their classes, which proves the lemma.
\end{proof}

\subsection{Fourier vanishing and multiplicative projectors}\label{multiplicity}

We use the preceding codimension estimates to prove Fourier vanishing by adapting the argument of \cite[Subsection~3.5.3]{MSY23}.

\begin{proposition}[Fourier vanishing]\label{Prop:FV vanishing}
For any $\gamma_1$ and $\gamma_2$ in $\Gamma$, we have
\begin{equation}
  \label{FV vanishing for gamma factor} \mathfrak{F}_{\gamma_1, i}\circ \mathfrak{F}^{-1}_{\gamma_2, j}=0 \qquad \text{for all } i+j< 2g_P.
\end{equation}
Summing over $\gamma_1,\gamma_2\in\Gamma$ gives
\begin{equation}
  \label{FV vanishing} \mathfrak{F}_{i}\circ \mathfrak{F}^{-1}_{j}=0 \qquad \text{for all } i+j< 2g_P.
\end{equation}
\end{proposition}

\begin{proof}
By Lemma~\ref{sectorwise isomorphism},
\[
  \mathfrak F_{\gamma_1,i} =\mathfrak F_{\gamma_1,i}e_{\kappa_{\gamma_1}}, \qquad \mathfrak F_{\gamma_2,j}^{-1} =e_{\kappa_{\gamma_2}}\mathfrak F_{\gamma_2,j}^{-1}.
\]
The nondegeneracy of the Weil pairing shows that the composition vanishes unless \(\gamma_1=\gamma_2\). We may therefore write \(\gamma_1=\gamma_2=\gamma\) and \(a=a_\gamma\).

Let \(N\equiv1\pmod n\). On the $\gamma$-sector, Adams operations replace an eigenvalue $\zeta$ by $\zeta^N$. Since $\zeta^n=1$, this congruence preserves every eigenvalue, and therefore $t_\gamma\psi^N=\psi^N t_\gamma$. The support estimate of Lemma~\ref{lem:Adams-support-for-FV}, together with Lemma~\ref{lem:support-bound-normalized-tau}, implies that the \((\gamma,\gamma)\)-sector of \(\tau^{\mathrm{nor}}(\mathscr A_N)\) has no component of negative correspondence degree.

Let \(\widehat{\mathfrak F}_{\gamma,j}^{-1}\) be the component of Chern codimension \(j+\dim B-a\) in the eigenvalue-weighted localized Chern character of \(\iota_*\mathscr F^{-1}\), before multiplication by the fixed invertible normalization class in \eqref{eq:normalized-tau-supported-definition}. Compatibility of localized Chern characters with Adams operations gives
\begin{equation}
  \label{eq:Adams-scaling-for-FV-completion} \widehat{\mathfrak F}_{\gamma,j}^{-1,(N)} =N^{j+\dim B-a}\widehat{\mathfrak F}_{\gamma,j}^{-1}.
\end{equation}
Apply the supported $K$-theory form of Proposition~\ref{prop:normalized-tau-composition} established above to $\mathscr A_N$. The inverse of the fixed normalization has only nonnegative correspondence degrees. Precomposing with it therefore preserves the absence of negative-degree components and gives, for every $m<2g_P$,
\[
  0= \sum_{i+j=m} N^{j+\dim B-a} \mathfrak F_{\gamma,i} \circ \widehat{\mathfrak F}_{\gamma,j}^{-1}
\]
for all positive \(N\equiv1\pmod n\). The powers of \(N\) are distinct as \(j\) varies. The Vandermonde argument therefore gives
\[
  \mathfrak F_{\gamma,i} \circ \widehat{\mathfrak F}_{\gamma,j}^{-1}=0 \qquad\text{for }i+j<2g_P.
\]

Finally, let $\mathsf L_{\gamma,\ell}$ be the diagonal correspondence given by the degree-$2\ell$ part of the fixed normalization class multiplying the localized Chern character in \eqref{eq:normalized-tau-supported-definition}. Then
\[
  \mathfrak F_{\gamma,j}^{-1} = \sum_{\ell\geq0} \widehat{\mathfrak F}_{\gamma,j-\ell}^{-1} \circ\mathsf L_{\gamma,\ell}.
\]
If \(i+j<2g_P\), then \(i+j-\ell<2g_P\) for every \(\ell\geq0\), so each term vanishes. This proves~\eqref{FV vanishing for gamma factor}; summing over the sectors gives~\eqref{FV vanishing}.
\end{proof}

\begin{proposition}\label{convolution kenel}
Set $\mathfrak C=\tau^{\mathrm{nor}}_B([\mathscr K])$. Then
\begin{equation}
  \label{convolution kernel map} \mathfrak C\in\bigoplus_{i\geq-g_P} \operatorname{Corr}^{i}_{B,\mathrm{CR}} \bigl([\check P/\Gamma],[\check P/\Gamma];[\check P/\Gamma]\bigr).
\end{equation}
Moreover,
\begin{equation}
  \label{to get cup product} \mathfrak F^{-1}\circ\mathfrak C\circ (\mathfrak F\otimes\mathfrak F) = [\Delta^{\mathrm{sm}}_{\check P/B}] \in\operatorname{Corr}^{0}_{B}(\check P,\check P;\check P).
\end{equation}
The right-hand side realizes the ordinary cup product on $R\check h_*\CC_{\check P}$.
\end{proposition}

\begin{proof}
Applying the mixed-kernel assertion of Proposition~\ref{prop:codimQ} to the four-factor expression \eqref{eq:convolution-kernel-four-factor-form}, we obtain
\[
  \dim\bigl(\Supp(\mathscr K)\cap ([\check P_{\gamma_1}/\Gamma]\times_B [\check P_{\gamma_2}/\Gamma]\times_B [\check P_{\gamma_3}/\Gamma])\bigr) \leq\dim B+2g_P-a_{\gamma_1}-a_{\gamma_2}-a_{\gamma_3}.
\]
The same bound applies to the supported ambient restriction $L(\iota_{123}^{\gamma_1,\gamma_2,\gamma_3})^* Rj_{123*}\mathscr K$ used to define $\tau^{\mathrm{nor}}$.

Since $\dim[\check P_{\gamma_3}/\Gamma] =\dim B+g_P-2a_{\gamma_3}$, the binary estimate in Lemma~\ref{lem:support-bound-normalized-tau} gives the bound
\[
  \dim[\check P_{\gamma_3}/\Gamma]+a_{\gamma_3} -a_{\gamma_1}-a_{\gamma_2} -\bigl(\dim B+2g_P -a_{\gamma_1}-a_{\gamma_2}-a_{\gamma_3}\bigr)=-g_P
\]
This proves \eqref{convolution kernel map} on every sector.

Finally, the identities for the inverse Fourier--Mukai kernels and associativity of convolution give
\[
  \mathscr F^{-1}\circ\mathscr K\circ (\mathscr F\boxtimes\mathscr F) \simeq\sO_{\Delta^{\mathrm{sm}}_{\check P/B}}.
\]
Apply the binary version of Proposition~\ref{prop:normalized-tau-composition}. The calculation in the proof of Lemma~\ref{lem:normalized-diagonal-is-identity}, applied to the small diagonal of the scheme $\check P$, gives \eqref{to get cup product}.
\end{proof}

For every $k$ and $\gamma\in\Gamma$, define
\begin{align*}
  \mathfrak{p}_{k} & =\sum_{i\leq k}\mathfrak{F}^{-1}_{ i}\circ \mathfrak{F}_{ 2g_P-i}\in \operatorname{Corr}^{0}_{B,\rm CR}(\check{P}, \check{P}) \qquad \mathfrak{q}_{k+1} & =\sum_{i\geq k+1}\mathfrak{F}^{-1}_{ i}\circ \mathfrak{F}_{ 2g_P-i} \in \operatorname{Corr}^{0}_{B,\rm CR}(\check{P}, \check{P}). \\
  \mathfrak{p}_{k, \gamma} & =\sum_{i\leq k}\mathfrak{F}^{-1}_{\gamma,i}\circ \mathfrak{F}_{\gamma,2g_P-i}\in \operatorname{Corr}^{0}_{B,\rm CR}(\check{P}, \check{P}).
\end{align*}
These are endomorphisms of $R\check h_*\CC_{\check P}$. Then
\[
  \mathfrak{p}_{k}=\sum_{\gamma\in \Gamma}\mathfrak{p}_{k, \gamma}
\]
By Proposition~\ref{inverse to each other}, $\mathfrak p_k+\mathfrak q_{k+1}$ is the identity morphism.

\begin{proposition}\label{prop:local-multiplicative-projectors}
\begin{enumerate}
\item[\textup{(i)}] The morphisms $\mathfrak p_k$ and $\mathfrak q_{k+1}$ are orthogonal projectors:
\[
  \mathfrak p_k^2=\mathfrak p_k, \qquad \mathfrak q_{k+1}^2=\mathfrak q_{k+1}, \qquad \mathfrak p_k\circ\mathfrak q_{k+1} = \mathfrak q_{k+1}\circ\mathfrak p_k = 0.
\]
\item[\textup{(ii)}] For any $k, l \geq 0$,
\[
  \mathfrak q_{k+l+1} \circ [\Delta^{\mathrm{sm}}_{\check P/B}] \circ (\mathfrak p_k\otimes\mathfrak p_l) = 0 .
\]
Equivalently, the cup product restricts to a morphism
\[
  \cup: \Im \mathfrak{p}_k \otimes \Im \mathfrak{p}_l\longrightarrow \Im \mathfrak{p}_{k+l}.
\]
\end{enumerate}
\end{proposition}

\begin{proof}
The degree-zero part of the inverse identity gives $\mathfrak p_k+\mathfrak q_{k+1}=1$. Moreover,
\[
  \mathfrak q_{k+1}\circ\mathfrak p_k =\sum_{\substack{r\geq k+1\\i\leq k}} \mathfrak F_r^{-1}\circ\mathfrak F_{2g_P-r} \circ\mathfrak F_i^{-1}\circ\mathfrak F_{2g_P-i}=0,
\]
since $(2g_P-r)+i<2g_P$ and \eqref{FV vanishing} applies to every middle factor. Since $\mathfrak p_k+\mathfrak q_{k+1}=1$, this also gives idempotence and the other orthogonality relation.

For (ii), \eqref{to get cup product} gives
\begin{equation}
  \label{to be zero} \mathfrak q_{k+l+1}\circ[\Delta^{\mathrm{sm}}_{\check P/B}]\circ(\mathfrak p_k\otimes\mathfrak p_l)=\mathfrak q_{k+l+1}\circ \mathfrak{F}^{-1}\circ \mathfrak{C}\circ \bigl( (\mathfrak{F}\circ \mathfrak{p}_k)\otimes (\mathfrak{F}\circ \mathfrak{p}_l)\bigr).
\end{equation}
Using \eqref{FV vanishing} again, we obtain
\begin{align*}
  \mathfrak q_{k+l+1}\circ \mathfrak{F}^{-1}& =\bigl(\sum_{i\geq k+l+1}\mathfrak{F}^{-1}_{ i}\circ \mathfrak{F}_{ 2g_P-i}\bigr)\circ \bigl(\sum_{j}\mathfrak{F}^{-1}_j\bigr) \\
  & =\bigl(\sum_{i\geq k+l+1}\mathfrak{F}^{-1}_{ i}\circ \mathfrak{F}_{ 2g_P-i}\bigr)\circ \bigl(\sum_{j\geq k+l+1}\mathfrak{F}^{-1}_j\bigr) \in \bigoplus_{i\geq k+l+1-g_P}\operatorname{Corr}^{i}_{B,\rm CR}([\check{P}/\Gamma], \check{P}).
\end{align*}
Similarly,
\[
  \mathfrak{F}\circ \mathfrak{p}_k\in \bigoplus_{i\geq g_P-k}\operatorname{Corr}^{i}_{B,\rm CR}(\check{P}, [\check{P}/\Gamma]) \qquad \mathfrak{F}\circ \mathfrak{p}_l\in \bigoplus_{i\geq g_P-l}\operatorname{Corr}^{i}_{B,\rm CR}(\check{P}, [\check{P}/\Gamma])
\]
By~\eqref{convolution kernel map} and Proposition~\ref{Degree and realization of composition}, the right-hand side of~\eqref{to be zero} has correspondence degrees
\[
  \geq g_P-k+g_P-l-g_P+k+l+1-g_P=1.
\]
The left-hand side has degree $0$, so it vanishes.
\end{proof}

\subsection{Twisted compactified Prym}\label{subsec:twisted-compactified-prym}

We return to $\mathbf B$ and follow the multisection construction of \cite[Section~4]{MSY23}. Fix $x\in\Sigma$ and put $\mathbf T=\pi^{-1}(\{x\}\times\mathbf B)$, which is a finite flat multisection of degree $n$. Retain $d_L$ and $\Lambda_L$ from \eqref{eq:spectral-sheaf-degree}. The BNR correspondence identifies $\check{\mathbf P}^{L}:=\Nm^{-1}(\Lambda_L) \subset\overline{\mathbf J}_{\mathbf C}^{d_L}$ with $\mathbf M$. We denote its structure map by $\check h^L$. We also consider $\widehat{\mathbf P}^{\,0}:= [\overline{\mathbf J}_{\mathbf C}^{0}/J_\Sigma^0]$, with structure map $\widehat h^0$.

For $a\in\{0,d_L\}$, let $\overline{\mathcal J}_{\mathbf C}^{a}$ be the stack of pairs $(\mathscr H,\alpha)$, where, over a $\mathbf B$-scheme $S$, $\mathscr H$ is an $S$-flat family of rank-one torsion-free sheaves of degree $a$ on $\mathbf C_S$, and $\alpha:\det p_{S*}(\mathscr H|_{\mathbf T_S})\simeq\sO_S$; here $p_S:\mathbf T_S\to S$ is the projection. By \cite[Proposition~4.1]{MSY23}, this is a $\mu_n$-gerbe over $\overline{\mathbf J}_{\mathbf C}^{a}$, carrying a universal sheaf. Define
\[
  \check{\mathcal P}^{L} :=\check{\mathbf P}^{L} \times_{\overline{\mathbf J}_{\mathbf C}^{d_L}} \overline{\mathcal J}_{\mathbf C}^{d_L}, \qquad \widehat{\mathcal P}^{\,0} :=[\overline{\mathcal J}_{\mathbf C}^{0}/J_\Sigma^0].
\]
We use line bundles rigidified at $x$ to define the $J_\Sigma^0$-action; tensoring by their pullbacks preserves $\alpha$. The gerbe projections are denoted by
\[
  \check{\varpi}:\check{\mathcal P}^{L}\to\check{\mathbf P}^{L}, \qquad \hat{\varpi}:\widehat{\mathcal P}^{\,0}\to\widehat{\mathbf P}^{\,0}.
\]
Both are $\mu_n$-gerbes. Rigidification quotients each automorphism group by the central $\mu_n$. Applying the tensor product construction of Lemma~\ref{etale cover of J bar} to the degree-zero norm fiber gives a finite $\Gamma$-quotient presentation of $\widehat{\mathcal P}^{\,0}$. Thus both stacks are nonsingular separated Deligne--Mumford stacks, proper and flat over $\mathbf B$, with quasi-projective coarse spaces. The quotient-stack argument of \cite[Subsection~4.4]{MSY23}, followed by the finite $\Gamma$-quotient, applies also to their relative products.

Write $\mathscr U_{\mathbf C}^{L}$ for the universal sheaf on $\mathbf C\times_{\mathbf B}\check{\mathcal P}^{L}$. Retain the section $s:B\to C$ and the line bundle $\mathscr N_L$ fixed in \eqref{first etale cover} and \eqref{eq:norm-correction-line}.

\begin{proposition}\label{prop:twisted-prym-etale-comparison}
The degree and norm correction defines a $\Gamma$-equivariant morphism
\[
  \check\sigma_L:\check{\mathcal P}^{L}_B\longrightarrow\check P, \qquad (\mathscr H,\alpha)\longmapsto [\mathscr H\otimes\sO_C(-d_L\operatorname{Im}s) \otimes\pi^*\mathscr N_L].
\]
There is a line bundle $\mathscr L_L$ on $\check{\mathcal P}^{L}_B$ such that
\begin{equation}
  \label{eq:twisted-universal-sheaf-comparison}
\begin{split}
  \mathscr U_{\mathbf C}^{L}|_B\simeq{}& (\id_C\times\check\sigma_L)^*\check{\mathscr U}_C \otimes p_C^*\bigl(\sO_C(d_L\operatorname{Im}s) \otimes\pi^*\mathscr N_L^{-1}\bigr) \otimes p_{\check{\mathcal P}}^*\mathscr L_L,
\end{split}
\end{equation}
where the $p$'s are the projections and $\check{\mathscr U}_C$ is normalized along $s$. For every geometric point $b\to B$,
\begin{equation}
  \label{eq:twisting-line-homologically-trivial} c_1(\mathscr L_L|_{\check{\mathcal P}^{L}_b})=0 \quad\text{in }H^2(\check{\mathcal P}^{L}_b,\Q).
\end{equation}
\end{proposition}

\begin{proof}
By \eqref{eq:norm-correction-line}, the corrected family has degree zero and trivial relative norm. Applying the normalization construction of \cite[Proposition~4.4(i)]{MSY23} along $s$, and restricting to the norm fiber, gives \eqref{eq:twisted-universal-sheaf-comparison}. The construction commutes with tensoring by $\pi^*\gamma$.

On nonsingular spectral fibers, the last assertion is \cite[Proposition~4.4(ii)]{MSY23}. For an arbitrary integral fiber $C_b$, the BNR pushforward of $\mathscr U_{\mathbf C}^{L}|_b$ is a vector bundle on $\Sigma\times\check{\mathcal P}^{L}_b$ whose determinant is trivial along $x$. Its evaluation at any $y\in\Sigma$ therefore has zero first Chern class. Choose $y$ so that $\pi_b^{-1}(y)=\{z_1,\ldots,z_n\}$ consists of distinct smooth points. Since $C_b^{\mathrm{reg}}$ is connected, the evaluation line bundles of the universal sheaf have the same first Chern class there. Hence
\[
  0=\sum_{j=1}^n c_1\!\left( \mathscr U_{\mathbf C}^{L}|_{\{z_j\}\times\check{\mathcal P}^{L}_b} \right) =n\,c_1\!\left( \mathscr U_{\mathbf C}^{L}|_{\{s(b)\}\times\check{\mathcal P}^{L}_b} \right).
\]
Evaluating \eqref{eq:twisted-universal-sheaf-comparison} at $s(b)$ proves \eqref{eq:twisting-line-homologically-trivial}.
\end{proof}

We descend the Fourier construction from $B$ to $\mathbf B$ as follows.

\smallskip
\noindent\textbf{Step 1. The Fourier--Mukai kernels and equivalence.}

For a $\mu_n$-gerbe $\mathcal X$, we write $D^b\Coh(\mathcal X)_{(w)}$ for the subcategory on which its central stabilizer acts by $\lambda\mapsto\lambda^w$, where $w$ is taken modulo $n$. On a product, the weights refer to the two central stabilizers separately. Consider the mixed Poincar\'e sheaf $\mathscr P_{0,d_L}$ on $\overline{\mathcal J}_{\mathbf C}^{d_L} \times_{\mathbf B}\overline{\mathcal J}_{\mathbf C}^{0}$ constructed as in \cite[Subsection~4.2]{MSY23}; its subscripts record the central weights $(0,d_L)$.

Choose an $x$-rigidified line bundle $\mathscr N_x\in\Pic^0(\Sigma)$ with
\begin{equation}
  \label{eq:global-norm-correction-for-kernel} \mathscr N_x^{\otimes n} \simeq L^{-1}\otimes \sO_\Sigma\bigl(d_Lx-\tfrac{n(n-1)}2D\bigr).
\end{equation}
It exists because the line on the right has degree zero. This is the norm correction at the fixed point $x$, whereas $\mathscr N_L$ is the correction at $s_\Sigma$: writing $\mathscr N_x$ also for its pullback to $\Sigma\times B$,
\[
  (\mathscr N_L\otimes\mathscr N_x^{-1})^{\otimes n} \simeq\sO_{\Sigma\times B} \bigl(d_L(\operatorname{Im}s_\Sigma-\{x\}\times B)\bigr).
\]
Since $\det\pi_*\sO_{\mathbf C} =\sO_\Sigma(-\tfrac{n(n-1)}2D)$, pulled back to $\Sigma\times\mathbf B$, the family $\mathscr U_{\mathbf C}^{L}\otimes\pi^*\mathscr N_x$ has relative norm $\sO_\Sigma(d_Lx)$ and degree $d_L$.

Pull back $\mathscr P_{0,d_L}$ along $(\mathscr H_1,\mathscr H_2)\mapsto (\mathscr H_1\otimes\pi^*\mathscr N_x,\mathscr H_2)$, and restrict the first factor to $\check{\mathcal P}^{L}$. Equivalently, we apply the determinant-of-cohomology construction to the corrected first family. By \cite[Proposition~4.4]{GS22}, tensoring the second family by $\pi^*N$, for $N\in J_\Sigma^0$, changes this kernel by $N|_x^{\otimes d_L}$. The $x$-rigidification trivializes this line compatibly with tensor product, giving a $J_\Sigma^0$-linearization. Descent and the relative-adjoint construction give
\begin{equation}
  \label{eq:global-twisted-Prym-kernel}
\begin{split}
  \mathscr F_{L,0}&\in D^b\Coh(\check{\mathcal P}^{L}\times_{\mathbf B} \widehat{\mathcal P}^{\,0})_{(0,d_L)},\\
  \mathscr F_{L,0}^{-1}&\in D^b\Coh(\widehat{\mathcal P}^{\,0}\times_{\mathbf B} \check{\mathcal P}^{L})_{(-d_L,0)}.
\end{split}
\end{equation}

Let
\[
  \widehat\sigma_0:\widehat{\mathcal P}^{\,0}_B \xrightarrow{\ \hat\varpi\ } \widehat{\mathbf P}^{\,0}_B \xrightarrow{\ \sim\ } [\check P/\Gamma]
\]
be the composite of the gerbe projection with the degree-zero quotient identification. Applying the determinant calculation of \cite[Corollary~4.5]{MSY23} to \eqref{eq:twisted-universal-sheaf-comparison} gives
\begin{equation}
  \label{eq:twisted-Prym-kernel-local-form}
\begin{split}
  \mathscr F_{L,0}|_B &\simeq(\check\sigma_L\times\widehat\sigma_0)^*\mathscr F \otimes p_2^*\mathscr M_0,\\
  \mathscr F_{L,0}^{-1}|_B &\simeq(\widehat\sigma_0\times\check\sigma_L)^*\mathscr F^{-1} \otimes p_1^*\mathscr M_0^{-1},
\end{split}
\end{equation}
for a line bundle $\mathscr M_0$ of central weight $d_L$ on $\widehat{\mathcal P}^{\,0}_B$. Its restriction to every geometric fiber satisfies
\begin{equation}
  \label{eq:external-kernel-line-fiberwise-trivial} c_1(\mathscr M_0|_{\widehat{\mathcal P}^{\,0}_b})=0 \quad\text{in }H^2(\widehat{\mathcal P}^{\,0}_b,\Q).
\end{equation}
On a geometric fiber, pull back to the norm-zero gerbe atlas. The determinant calculation expresses $\mathscr M_0|_b$, up to a constant line, as the pairing of the universal degree-zero spectral sheaf with
\[
  \sO_{C_b}(d_Ls(b))\otimes \pi_b^*\bigl((\mathscr N_L|_{\Sigma\times\{b\}})^{-1} \otimes\mathscr N_x\bigr).
\]
This line has degree $d_L$. The evaluation argument in Proposition~\ref{prop:twisted-prym-etale-comparison} applies also to the degree-zero norm gerbe: its universal evaluation lines at smooth points of $C_b$ all have the same first Chern class, and the determinant rigidification makes $n$ times that class zero. They therefore have zero rational first Chern class. Representing the displayed line by a divisor supported in $C_b^{\mathrm{reg}}$ expresses its determinant pairing as a tensor product of such evaluation lines and their inverses. This proves \eqref{eq:external-kernel-line-fiberwise-trivial} on the atlas; the finite quotient trace proves it on $\widehat{\mathcal P}^{\,0}_b$.

Weight-zero descent identifies $D^b\Coh(\check{\mathcal P}^{L})_{(0)}$ with $D^b\Coh(\check{\mathbf P}^{L})$. Over $B$, tensoring by $\mathscr M_0^{-1}$ similarly identifies the dual weight-$d_L$ category with the weight-zero category. Thus \eqref{eq:twisted-Prym-kernel-local-form} reduces the adjunction morphisms to those of $\mathscr F$ and $\mathscr F^{-1}$. The cones of these adjunction morphisms vanish after the faithfully flat base change $B\to\mathbf B$, and therefore vanish globally. Consequently,
\begin{equation}
  \label{eq:twisted-Prym-derived-equivalence} \Phi_{\mathscr F_{L,0}}: D^b\Coh(\check{\mathcal P}^{L})_{(0)} \xrightarrow{\ \simeq\ } D^b\Coh(\widehat{\mathcal P}^{\,0})_{(d_L)}
\end{equation}
has inverse $\Phi_{\mathscr F_{L,0}^{-1}}$, as in \cite[Theorem~4.7]{GS22} and \cite[Proposition~4.3]{MSY23}.

\smallskip
\noindent\textbf{Step 2. Borel--Moore homology and coefficient systems.}

Gerbe pullback gives an isomorphism
\[
  \check{\varpi}^*:H_*^{\mathrm{BM}}(\check{\mathbf P}^{L},\CC) \xrightarrow{\ \simeq\ } H_*^{\mathrm{BM}}(\check{\mathcal P}^{L},\CC),
\]
whose inverse, in our stack-cycle convention, is $n\check{\varpi}_*$. On the dual side, we compare the inertia stacks. Write $\widehat{\mathbf P}^{\,0}_\gamma$ for the $\gamma$-sector of $I\widehat{\mathbf P}^{\,0}$, with structure map $\widehat h^0_\gamma$, and let $\mathcal I_\gamma$ be its inverse image under $I\hat{\varpi}$. By \cite[Construction~2.8]{GWZ20}, this map factors as
\begin{equation}
  \label{eq:twisted-inertia-lift-torsor} \mathcal I_\gamma \xrightarrow{\ \rho_\gamma\ } \mathcal T_\gamma \xrightarrow{\ q_\gamma\ } \widehat{\mathbf P}^{\,0}_\gamma,
\end{equation}
where $\rho_\gamma$ rigidifies the central $\mu_n$ and $q_\gamma$ is a representable finite \'etale $\mu_n$-torsor. Over a point of $\widehat{\mathbf P}^{\,0}_\gamma$, the torsor parametrizes the lifts of $\gamma$ to the gerbe.

For $w\in\Z/n$, let $\mathcal L_{0,\gamma}^{(w)}$ be the local system of functions on this torsor satisfying $f(\lambda t)=\lambda^w f(t)$. Then
\begin{equation}
  \label{eq:twisted-character-coefficients} (q_\gamma)_*\CC_{\mathcal T_\gamma} =\bigoplus_{w\in\Z/n}\mathcal L_{0,\gamma}^{(w)}, \qquad \mathcal L_{0,\gamma}^{(w)}\otimes \mathcal L_{0,\gamma}^{(-w)}\simeq\CC.
\end{equation}
Let $(w)$ denote the eigenspace of character $\lambda\mapsto\lambda^w$ for the pullback action of $\mu_n$ on the lifts. Then
\[
  H_*^{\mathrm{BM}}(\mathcal I_\gamma,\CC)_{(w)} \xrightarrow[\ \simeq\ ]{\ n(\rho_\gamma)_*\ } H_*^{\mathrm{BM}}(\mathcal T_\gamma,\CC)_{(w)} \simeq H_*^{\mathrm{BM}} (\widehat{\mathbf P}^{\,0}_\gamma,\mathcal L_{0,\gamma}^{(w)}).
\]
For a kernel of central weight $w$, changing an inertia lift gives $t_{\lambda\widetilde\gamma}([K]) =\lambda^w t_{\widetilde\gamma}([K])$. Its inertial class therefore has coefficients in $\mathcal L_{0,\gamma}^{(w)}$. In particular, after the age shifts, the Fourier kernel and its inverse give correspondences between
\begin{equation}
  \label{eq:twisted-sector-target} R\check h^L_*\CC_{\check{\mathbf P}^{L}}, \qquad \bigoplus_{\gamma\in\Gamma} R\widehat h^0_{\gamma*} \mathcal L_{0,\gamma}^{(d_L)}[-2a_\gamma].
\end{equation}
The inverse carries the dual local systems on its source.

\smallskip
\noindent\textbf{Step 3. Comparison with the previous Riemann--Roch construction.}

We use \eqref{KS's tau functor} and the target normalization of Subsection~\ref{subsec:cohomological-functors}. The classes $\Theta_{i,g}$ in \eqref{eq:normalized-tau-supported-definition} are pulled back from the rigidifications because a $\mu_n$-gerbe has zero relative tangent complex. In particular, the ages remain $a_\gamma$. Compatibility with proper pushforward is given by \eqref{eq:orbifold-tau-proper-covariance}; see \cite[Theorem~10.11]{KS20}.

For a relative product containing $m$ gerbe factors, let $\boldsymbol\gamma$ be the tuple of sector labels on all factors, with entry $1$ on an unquotiented Prym factor, and let $\widetilde{\boldsymbol\gamma}$ be the tuple obtained by choosing the tautological lift on each gerbe factor. Write $\rho^{(m)}$ for the product of the $m$ central rigidification maps. Under the sectorwise isomorphisms of Step~2, define
\begin{equation}
  \label{eq:rigidified-normalized-tau} \tau^{\mathrm{nor}}_{\Gamma,\mathrm{rig}}([K])_{\boldsymbol\gamma} :=n^m\rho^{(m)}_* \tau^{\mathrm{nor}}([K])_{\widetilde{\boldsymbol\gamma}}.
\end{equation}
The factor $n^m$ makes this inverse to pullback along the product of the $m$ gerbes: for a class $\beta$ on the rigidification, $\rho^{(m)}_*(\rho^{(m)})^*\beta=n^{-m}\beta$, since each $\mu_n$-gerbe has stack-theoretic degree $1/n$.

Let $\varpi:\mathcal W\to\mathcal W_{\mathrm{rig}}$ be a product of auxiliary gerbe projections on a relative product. If the central $\mu_n$ of each gerbe factor acts trivially on $K$, then weight-zero descent gives
\begin{equation}
  \label{eq:gerbe-weight-zero-descent} \overline K:=R\varpi_*K, \qquad L\varpi^*\overline K\simeq K;
\end{equation}
see \cite[Subsection~4.4, Step~1]{MSY23}. The ambient derived restriction and its eigensummands are then pulled back from $\mathcal W_{\mathrm{rig}}$. Flat pullback for the localized Chern character and the equality of the tangent factors therefore give
\begin{equation}
  \label{eq:gerbe-weight-zero-tau} \tau^{\mathrm{nor}}_{\Gamma,\mathrm{rig}}([K]) =\tau^{\mathrm{nor}}([R\varpi_*K]).
\end{equation}
This also applies when only some gerbe factors are rigidified, retaining the remaining coefficient systems.

For the Fourier kernels, use \eqref{eq:twisted-Prym-kernel-local-form}. Let $\ch^{\mathrm{rig}}_\gamma(\mathscr M_0)$ denote the class in $H^{2*}([\check P_\gamma/\Gamma], \mathcal L_{0,\gamma}^{(d_L)}|_B)$ whose pullback to the torsor $\mathcal T_\gamma$ in \eqref{eq:twisted-inertia-lift-torsor} is $\operatorname{Tr}(\widetilde\gamma\mid\mathscr M_0) \exp(c_1(\mathscr M_0))$. Its inverse is $\ch^{\mathrm{rig}}_\gamma(\mathscr M_0^{-1})$, with dual coefficients. Under $\check\sigma_L$ and $\widehat\sigma_0$, the module formula gives
\begin{equation}
  \label{eq:twisted-Prym-tau-local-form}
\begin{split}
  \left.\tau^{\mathrm{nor}}_{\Gamma,\mathrm{rig}} ([\mathscr F_{L,0}])_{1,\gamma}\right|_B &=p_2^*\ch^{\mathrm{rig}}_\gamma(\mathscr M_0) \cap\tau^{\mathrm{nor}}_B([\mathscr F])_{1,\gamma},\\
  \left.\tau^{\mathrm{nor}}_{\Gamma,\mathrm{rig}} ([\mathscr F_{L,0}^{-1}])_{\gamma,1}\right|_B &=p_1^*\ch^{\mathrm{rig}}_\gamma(\mathscr M_0^{-1}) \cap\tau^{\mathrm{nor}}_B([\mathscr F^{-1}])_{\gamma,1}.
\end{split}
\end{equation}
Indeed, after tensoring with the inverse external line, each equality reduces to \eqref{eq:gerbe-weight-zero-tau}.

In each convolution, the central weights on the middle factor are opposite, so the tensor integrand descends through the middle gerbe. Pairing the lift characters contributes $\frac{1}n\sum_{\lambda\in\mu_n}\lambda^w\lambda^{-w}=1$. The proof of Proposition~\ref{prop:normalized-tau-composition}, with the same middle $\Theta$-factor, consequently gives
\begin{equation}
  \label{eq:rigidified-tau-composition} \tau^{\mathrm{nor}}_{\Gamma,\mathrm{rig}} ([K_{23}\circ K_{12}]) =\tau^{\mathrm{nor}}_{\Gamma,\mathrm{rig}}([K_{23}]) \circ\tau^{\mathrm{nor}}_{\Gamma,\mathrm{rig}}([K_{12}]).
\end{equation}
The diagonal kernels representing the identity on the weight categories have opposite central weights on their two factors; after the normalization in \eqref{eq:rigidified-normalized-tau}, their classes are the identity on the local systems in \eqref{eq:twisted-sector-target}. On the weight-zero source, the small diagonal realizes the ordinary cup product.

\smallskip
\noindent\textbf{Step 4. Tensor kernels and codimension estimates.}

On the fourfold relative product ordered as $\widehat{\mathcal P}^{\,0},\widehat{\mathcal P}^{\,0}, \check{\mathcal P}^{L},\widehat{\mathcal P}^{\,0}$, define
\begin{equation}
  \label{eq:twisted-multiplication-kernel} \mathscr K_{L,0}:= Rp_{124*}\left( Lp_{13}^*\mathscr F_{L,0}^{-1} \otimes^L Lp_{23}^*\mathscr F_{L,0}^{-1} \otimes^L Lp_{34}^*\mathscr F_{L,0}\right).
\end{equation}
The central weight of the tensor integrand along the middle $\check{\mathcal P}^{L}$-factor is zero. By \eqref{eq:twisted-Prym-kernel-local-form}, its pullback to $B$ is the pullback of $\mathscr K$ tensored by $\mathscr M_0^{-1}\boxtimes\mathscr M_0^{-1}\boxtimes\mathscr M_0$. These external factors do not change the support. The estimate in Proposition~\ref{convolution kenel} therefore gives
\begin{equation}
  \label{eq:twisted-ternary-support} \dim\operatorname{Supp} (\mathscr K_{L,0}|_{\gamma_1,\gamma_2,\gamma_3}) \leq\dim\mathbf B+2g_P-a_{\gamma_1}-a_{\gamma_2}-a_{\gamma_3},
\end{equation}
where sector restrictions are taken in the smooth ambient product with support, as before. Thus $\mathfrak C^L:=\tau^{\mathrm{nor}}_{\Gamma,\mathrm{rig}} ([\mathscr K_{L,0}])$ has correspondence degrees at least $-g_P$ and coefficients $p_1^*\mathcal L_{0,\gamma_1}^{(-d_L)}\otimes p_2^*\mathcal L_{0,\gamma_2}^{(-d_L)}\otimes p_3^*\mathcal L_{0,\gamma_3}^{(d_L)}$.

\smallskip
\noindent\textbf{Step 5. Adams operations.}

Let $\iota$ embed $\widehat{\mathcal P}^{\,0}\times_{\mathbf B}\check{\mathcal P}^{L}$ in its smooth absolute product, and set
\begin{equation}
  \label{eq:twisted-Adams-class} \mathscr A_N^L:=[\mathscr F_{L,0}]\circ \psi^N([R\iota_*\mathscr F_{L,0}^{-1}]).
\end{equation}
Supported Adams operations are available on these quotient stacks as in \cite[Subsections~3.5.2 and~4.4]{MSY23}. The central weight of the convolution integrand along the middle $\check{\mathcal P}^{L}$-factor remains zero. Over $B$, each Newton monomial acquires the external factor $\mathscr M_0^{-N}\boxtimes\mathscr M_0$. Hence Proposition~\ref{codim for FV} and Lemma~\ref{lem:Adams-support-for-FV} show that the $(\gamma_1,\gamma_2)$-restriction of $\mathscr A_N^L$ lies in the subgroup generated by classes supported in dimension at most
\begin{equation}
  \label{eq:twisted-Adams-support} \dim\mathbf B+g_P-a_{\gamma_1}-a_{\gamma_2}.
\end{equation}
Its normalized class has no negative correspondence degrees.

For Fourier vanishing, take
\begin{equation}
  \label{eq:twisted-Adams-congruence} N\equiv1\pmod{n^2}.
\end{equation}
If $\gamma$ has order $m\mid n$, every lift has order dividing $mn$, since $\widetilde\gamma^m\in\mu_n$. Thus this congruence preserves both the lift eigenvalues and the coefficient systems, and
\begin{equation}
  \label{eq:twisted-Adams-scaling} \ch_r^{\mathrm{loc}} \bigl(t_{\widetilde{\boldsymbol\gamma}}\psi^N(\alpha)\bigr) =N^r\ch_r^{\mathrm{loc}} \bigl(t_{\widetilde{\boldsymbol\gamma}}\alpha\bigr),
\end{equation}
where $r$ is Chern codimension before normalization. We therefore apply the Vandermonde argument and the normalization calculation in the proof of Proposition~\ref{Prop:FV vanishing}.

Set
\begin{equation}
  \label{eq:global-rigidified-Fourier-classes} \mathfrak F^L:=\tau^{\mathrm{nor}}_{\Gamma,\mathrm{rig}} ([\mathscr F_{L,0}]), \qquad \mathfrak F^{L,-1}:=\tau^{\mathrm{nor}}_{\Gamma,\mathrm{rig}} ([\mathscr F_{L,0}^{-1}]).
\end{equation}
Write their sector components as $\mathfrak F^L_{\gamma,i}$ and $\mathfrak F^{L,-1}_{\gamma,i}$, of correspondence degree $i-g_P$, and put $\mathfrak F_i^L=\sum_\gamma\mathfrak F^L_{\gamma,i}$ and $\mathfrak F_i^{L,-1}=\sum_\gamma\mathfrak F^{L,-1}_{\gamma,i}$. Set
\begin{equation}
  \label{eq:global-twisted-sector-projectors} \mathfrak p_{k,\gamma}^L :=\sum_{i\leq k} \mathfrak F_{\gamma,i}^{L,-1} \circ\mathfrak F_{\gamma,2g_P-i}^L.
\end{equation}
For $\delta\in\Gamma$, let $\mathscr L_\delta$ also denote the weight-zero line on $\widehat{\mathcal P}^{\,0}$ obtained by descending the normalized determinant pairing with $\pi^*\delta$. Its pullback to $B$ is the line used in Lemma~\ref{the scalar on trivial line bundle}. The determinant-of-cohomology construction gives
\[
  (T_\delta\times\mathrm{id})^*\mathscr F_{L,0} \simeq\mathscr F_{L,0}\otimes p_2^*\mathscr L_\delta.
\]
As in Lemma~\ref{sectorwise isomorphism}, this identity implies that $\mathfrak F^L_{\gamma,i}$ vanishes on all character summands other than $\kappa_\gamma$, while the image of $\mathfrak F^{L,-1}_{\gamma,i}$ lies in the $\kappa_\gamma$-summand. Hence $\mathfrak p_{k,\gamma}^L$ acts only on that summand.

\begin{theorem}\label{prop:global-twisted-Prym-projectors}\label{coro:multi for twisted}
The normalized Fourier correspondences are mutually inverse and satisfy
\begin{equation}
  \label{eq:global-twisted-FV} \mathfrak F^L_{\gamma_2,i}\circ \mathfrak F^{L,-1}_{\gamma_1,j}=0 \qquad(i+j<2g_P).
\end{equation}
The correspondences
\begin{equation}
  \label{eq:global-twisted-projectors} \mathfrak p_k^L =\sum_{i\leq k}\mathfrak F_i^{L,-1}\circ\mathfrak F_{2g_P-i}^L,\qquad \mathfrak q_{k+1}^L =\sum_{i\geq k+1}\mathfrak F_i^{L,-1}\circ\mathfrak F_{2g_P-i}^L
\end{equation}
are complementary orthogonal projectors on $R\check h^L_*\CC_{\check{\mathbf P}^{L}}$.

Under $\check{\mathbf P}^{L}\simeq\mathbf M$, cup product restricts to
\begin{equation}
  \label{eq:moduli-projector-multiplicativity} \cup:\Im\mathfrak p_k^L\otimes^L\Im\mathfrak p_l^L \longrightarrow\Im\mathfrak p_{k+l}^L \qquad(k,l\geq0)
\end{equation}
in $D_c^b(\mathbf B)$, with the images in $Rh_*\CC_{\mathbf M}$.
\end{theorem}

\begin{proof}
The inverse identities follow from \eqref{eq:twisted-Prym-derived-equivalence} and Step~3; Fourier vanishing follows from Step~5. Taking the degree-zero part of $\mathfrak F^{L,-1}\circ\mathfrak F^L=\mathrm{id}$ gives $\mathfrak p_k^L+\mathfrak q_{k+1}^L=1$, while \eqref{eq:global-twisted-FV} gives $\mathfrak q_{k+1}^L\circ\mathfrak p_k^L=0$. Since the two endomorphisms sum to the identity, they are idempotent and orthogonal in both orders.

Steps~3 and~4 give
\[
  \mathfrak F^{L,-1}\circ\mathfrak C^L\circ (\mathfrak F^L\otimes\mathfrak F^L) =[\Delta^{\mathrm{sm}}_{\check{\mathbf P}^{L}/\mathbf B}],
\]
with $\mathfrak C^L$ of correspondence degrees at least $-g_P$. As in the proof of Proposition~\ref{prop:local-multiplicative-projectors}, Fourier vanishing implies that $\mathfrak q_{k+l+1}^L\circ [\Delta^{\mathrm{sm}}_{\check{\mathbf P}^{L}/\mathbf B}] \circ(\mathfrak p_k^L\otimes\mathfrak p_l^L)$ has degree at least $(k+l+1-g_P)-g_P+(g_P-k)+(g_P-l)=1$. It also has degree zero, so it vanishes.
\end{proof}

\begin{lemma}\label{lem:twisted-untwisted-projector-comparison}
We also denote by $\check\sigma_L:\check{\mathbf P}^{L}_B\xrightarrow{\sim}\check P$ the isomorphism induced by the morphism of Proposition~\ref{prop:twisted-prym-etale-comparison} after rigidification by the central $\mu_n$. For every geometric point $b\to B$, every $\gamma\in\Gamma$, and every $k$, the projectors $\mathfrak p_{k,\gamma}^L|_b$ and $\mathfrak p_{k,\gamma}|_b$, identified by $\check\sigma_L$, have the same image. Consequently, on an open subset on which the relevant perverse cohomology sheaves are local systems, they define the same sub-local systems.
\end{lemma}

\begin{proof}
By \eqref{eq:twisted-Prym-tau-local-form}, after pullback to $B$ the twisted Fourier correspondences differ from the untwisted ones by multiplication by $\ch^{\mathrm{rig}}_\gamma(\mathscr M_0)$, and the inverse correspondences differ by its inverse. On a geometric fiber, \eqref{eq:external-kernel-line-fiberwise-trivial} shows that this class is the nonzero stabilizer character of $\mathscr M_0$ in degree zero. The two scalar factors cancel in every summand defining the source projector. Hence the two projectors have the same action on every geometric fiber. The last assertion follows by applying this equality to the stalks of the corresponding local systems.
\end{proof}

\section{Fourier--Mukai transform and endoscopic correspondence}

We retain the notation of the preceding section and assume that $d=\deg D>2g-2$ is even.

\subsection{The endoscopic correspondence}

The $\Gamma$-action on $\mathbf M_{n,L}^D$ preserves the fibers of $h^D$ and gives the character decomposition
\[
  Rh^{D}_*\underline{\mathbb{C}}=\bigoplus_{\kappa\in \hat{\Gamma}}(Rh^{D}_*\underline{\mathbb{C}})_{\kappa}.
\]

Fix $\kappa=\kappa_\gamma\in\widehat\Gamma$, with the Weil-pairing convention fixed in Section~2. Let $\mathbf M_\gamma^D$ be the $\gamma$-fixed locus in the full moduli space and let $\mathbf B_\gamma^D=h^D(\mathbf M_\gamma^D)$. Write $h_\gamma^D:\mathbf M_\gamma^D\to\mathbf B_\gamma^D$ and $i_\gamma^D:\mathbf B_\gamma^D\hookrightarrow\mathbf B^D$ for the induced map and the inclusion, and put $\mathbf B_\gamma=\mathbf B_\gamma^D\cap\mathbf B$. The results of Ng\^{o} \cite[Theorems 7.8.3 and 7.8.5]{Ngo10} and de Cataldo \cite[Theorem 1.0.2]{deC17} (see also \cite[Corollary 2.2]{MS21}) give the following support theorem.

\begin{proposition}\label{support theorem}
For $d>2g-2$, the only strict support of the nonzero perverse cohomology sheaves of $(Rh^{D}_*\underline{\mathbb{C}})_{\kappa_\gamma}$ is $\mathbf{B}_{\gamma}^{D}$, where $\gamma$ corresponds to $\kappa$ via the Weil pairing. Moreover, for any $\gamma$, the generic point of $\mathbf{B}_{\gamma}^{D}$ is contained in $\mathbf{B}_\gamma$, the elliptic locus.
\end{proposition}

We recall the endoscopic construction of \cite[Section~3]{MS21}. Put $m=\operatorname{ord}(\gamma)$ and $r=n/m$, and let
\begin{equation}
  \label{etale cover piH} \pi_{\langle\gamma\rangle}: \Sigma_{\langle\gamma\rangle}\longrightarrow\Sigma
\end{equation}
be the connected cyclic étale cover determined by $\gamma$, with Galois group $G_\gamma=\langle\gamma\rangle^\vee$. The space $\mathbf M_{r,L}^D(\gamma)$ parametrizes rank-$r$, $\pi_{\langle\gamma\rangle}^*D$-twisted Higgs bundles on $\Sigma_{\langle\gamma\rangle}$ whose pushforward to $\Sigma$ is stable, trace-free, and has determinant $L$. Its Hitchin base $\mathbf B^D(\gamma)$ consists of the spectral coefficients with trace-free pushforward, and it carries the spectral family
\[
  \mathbf C^D(\gamma)\longrightarrow\mathbf B^D(\gamma).
\]
The group $G_\gamma$ acts on these spaces. The pushforward construction identifies the quotients of $\mathbf M_{r,L}^D(\gamma)$ and $\mathbf B^D(\gamma)$ with $\mathbf M_\gamma^D$ and $\mathbf B_\gamma^D$, respectively; its action on the moduli space is free. Write $q_{\mathbf B}:\mathbf B^D(\gamma)\to\mathbf B_\gamma^D$ for the finite quotient map, and let $\mathbf B_\gamma^{D*}$ be the open locus over which this quotient is \'etale. Its inverse image is denoted by $\mathbf B^D(\gamma)^*$.

Since $q_{\mathbf B}$ is finite and the generic point of $\mathbf B_\gamma^D$ lies in $\mathbf B_\gamma$, the dimension formula of Lemma~\ref{lem:H-fixed-dimensions} gives
\begin{equation}
  \label{eq:cyclic-endoscopic-base-dimension} \dim\mathbf B^D(\gamma)=\dim\mathbf B_\gamma^D =\dim\mathbf B^D-a_\gamma.
\end{equation}

Let
\begin{equation}
  \label{first open} \mathbf{B}^\heartsuit(\gamma)\subset \mathbf{B}^D(\gamma)
\end{equation}
be the largest open subset on which the \(G_\gamma\)-action is free and the spectral curves are smooth. The quotient
\begin{equation}
  \label{open in B gamma} \mathbf{C}_\gamma^\heartsuit = \mathbf C^\heartsuit(\gamma) / G_\gamma \longrightarrow \mathbf{B}_\gamma^\heartsuit := \mathbf{B}^\heartsuit(\gamma)/G_\gamma
\end{equation}
defines a family of smooth integral curves over the open subset \(\mathbf{B}_\gamma^\heartsuit \subset \mathbf{B}_\gamma^D\). Let $\mathbf C_{\gamma}^\circ\rightarrow \mathbf{B}_{\gamma}^\heartsuit$ be the restriction of the universal spectral curve $\mathbf C^D\rightarrow \mathbf{B}^D$. By \cite[Lemma~3.4]{MS21}, the natural morphism
\begin{equation}
  \label{family-wise normalization} \xymatrix{ \mathbf{C}_\gamma^\heartsuit \ar[rr]^{u_{\sC}} \ar[dr] & & \mathbf C_{\gamma}^\circ \ar[dl] \\ & \mathbf{B}_{\gamma}^\heartsuit, & }
\end{equation}
restricts to the normalization on every fiber over a closed point $a\in\mathbf B_\gamma^\heartsuit$.

By the BNR correspondence, pullback along $u_{\sC}$ defines a morphism
\[
  g_u: \mathbf{M}_{n,L}^{D,\mathrm{reg}} \times_{\mathbf{B}^D} \mathbf{B}^\heartsuit(\gamma) \longrightarrow \mathbf{M}_{r,L'}^{D,\mathrm{reg}}(\gamma) \times_{\mathbf{B}^D(\gamma)} \mathbf{B}^\heartsuit(\gamma),
\]
Here $\mathbf{M}_{n,L}^{D,\mathrm{reg}}$ and $\mathbf{M}_{r,L'}^{D,\mathrm{reg}}(\gamma)$ are the loci represented by line bundles on the corresponding spectral curves. The determinant $L'$ is the one induced by pullback to the normalization; it satisfies $\deg L'-\deg L=a_\gamma$, as in \cite[Lemma~3.6]{MS21}.

Taking the graph of \(g_u\) and then its reduced Zariski closure inside the product
\[
  \mathbf{M}_{r,L'}^D(\gamma) \times_{\mathbf{B}^D(\gamma)} \bigl( \mathbf{M}_{n,L}^D \times_{\mathbf{B}^D} \mathbf{B}^D(\gamma) \bigr)
\]
gives a closed subvariety \(\mathscr C\) and a commutative diagram; see \cite[Equation~(75)]{MS21}:
\[
  \xymatrix{ & \mathscr{C} \ar[dl] \ar[dr] & \\ \mathbf{M}_{n,L}^D \times_{\mathbf{B}^D} \mathbf{B}^D(\gamma) \ar[dr]_{h^D} & & \mathbf{M}_{r,L'}^D(\gamma) \ar[dl]^{h_{\pi,L^\prime}^D} \\ & \mathbf{B}^D(\gamma). & }
\]

The cycle of $\mathscr C$ defines a $\Gamma$-equivariant cohomological correspondence
\begin{align}
  \label{correspondence on cover} [\mathscr{C}]_\# : q_{\mathbf B}^* \bigl( \mathrm{R}h_*^D \underline{\mathbb{C}} \bigr) \longrightarrow \mathrm{R}h_{\pi,L'*}^D \underline{\mathbb{C}}[-2a_\gamma]
\end{align}
in the bounded derived category of constructible sheaves on \(\mathbf{B}^D(\gamma)\). Here $a_\gamma$ is the age computed in \eqref{computation of age}, which equals $\operatorname{codim}_{\mathbf B^D}(\mathbf B_\gamma^D)$ by \eqref{eq:cyclic-endoscopic-base-dimension}, and $q_{\mathbf B}$ is the quotient map defined above. When pulling back a complex on $\mathbf B^D$, $q_{\mathbf B}^*$ denotes $(i_\gamma^D\circ q_{\mathbf B})^*$.

To identify the relative moduli spaces with determinants $L'$ and $L$ by tensor product, we need $n\mid a_\gamma$. Equivalently, $(n-r)d$ must be even, which holds because $d$ is even. We may therefore choose a line bundle
\begin{align}
  \label{the line bundle N0} \sN_0\in \Pic^k(\Sigma)\qquad k=\frac{\deg (L^\prime)-\deg (L)}{n}\in \Z
\end{align}
with $L^\prime\simeq L\otimes\sN_0^{\otimes n}$. Tensoring by $\pi_{\langle\gamma\rangle}^*\sN_0^{-1}$ gives an isomorphism of moduli spaces
\begin{equation}
  \label{the iso phi} \xymatrix{ \mathbf{M}_{r,L'}^{D}(\gamma) \ar[rr]^{\phi_{\sN_0}} \ar[dr] & & \mathbf{M}_{r,L}^{D}(\gamma) \ar[dl] \\ & \mathbf{B}^D(\gamma) & }
\end{equation}

The endoscopic isomorphism of \cite{MS21} takes the following form.

\begin{theorem}[\cite{MS21}, Theorem~3.2 and~3.10]\label{MS's endoscopic}
There is an isomorphism
\[
  c_{\kappa}^D: \big(\mathrm{R}h_*^D\underline{\mathbb{C}}\big)_{\kappa}\xrightarrow{\cong} i_{\gamma *}^D \big(\mathrm{R}h_{\gamma *}^D\underline{\mathbb{C}}\big)_{\kappa}[-2a_\gamma],
\]
which is canonical up to scaling. After restriction to $\mathbf{B}_{\gamma}^{D*}$ and pullback along the \'{e}tale morphism $q_{\mathbf B}$, it is the following composite:
\[
  q_{\mathbf B}^*(\mathrm{R}h_*^D\underline{\mathbb{C}})_{\kappa}\big|_{\mathbf{B}^{D}(\gamma)^*} \xrightarrow{[\mathscr{C}]_{\#,\kappa}\big|_{\mathbf{B}^{D}(\gamma)^*}} (\mathrm{R}h_{\pi,L'*}^D\underline{\mathbb{C}})_{\kappa}\big|_{\mathbf{B}^{D}(\gamma)^*}[-2a_\gamma] \xrightarrow{(\phi_{\sN_0}^{-1})^*} (\mathrm{R}h_{\pi,L*}^D\underline{\mathbb{C}})_{\kappa}\big|_{\mathbf{B}^{D}(\gamma)^*}[-2a_\gamma]
\]
\end{theorem}

We recall why $[\mathscr{C}]_{\#,\kappa}\big|_{\mathbf{B}^{D}(\gamma)^*}$ is an isomorphism. Since $q_{\mathbf B}$ is \'{e}tale on $\mathbf{B}^{D}(\gamma)^*$, both the source and target of $[\mathscr{C}]_{\#,\kappa}\big|_{\mathbf{B}^{D}(\gamma)^*}$ are direct sums of shifted semisimple perverse sheaves. It suffices to show that the induced map on each perverse cohomology sheaf,
\[
  {}^p\sH^i([\mathscr{C}]_{\#,\kappa})\big|_{\mathbf{B}^{D}(\gamma)^*}: q_{\mathbf B}^*{}^p\sH^i\big(\mathrm{R}h_*^D\underline{\mathbb{C}}\big)_{\kappa}\big|_{\mathbf{B}^{D}(\gamma)^*} \xrightarrow{\;\simeq\;} {}^p\sH^{i-2a_\gamma}\big(\mathrm{R}h_{\pi,L'*}^D\underline{\mathbb{C}}\big)_{\kappa}\big|_{\mathbf{B}^{D}(\gamma)^*}
\]
is an isomorphism for every $i$. By \cite[Corollary~2.2, Theorem~2.3(b) and Proposition~2.10]{MS21}, both sides have $\mathbf{B}^{D}(\gamma)^*$ as the only support. They are therefore intermediate extensions of local systems on a dense open subset. It suffices to show that
\[
  {}^p\sH^i([\mathscr{C}]_{\#,\kappa})\big|_{U}: q_{\mathbf B}^*{}^p\sH^i\big(\mathrm{R}h_*^D\underline{\mathbb{C}}\big)_{\kappa}\big|_{U} \xrightarrow{\;\simeq\;} {}^p\sH^{i-2a_\gamma}\big(\mathrm{R}h_{\pi,L'*}^D\underline{\mathbb{C}}\big)_{\kappa}\big|_{U}
\]
is an isomorphism for some Zariski open subset $U\subset \mathbf{B}^{D}(\gamma)^*$. After shrinking $U$ so that these sheaves are shifted local systems, it is enough to show that the fiber correspondence
\[
  [\mathscr{C}_b]_{\#,\kappa, b}: H^i\big((\mathbf{M}_{n,L}^D)_b, \CC \big)_{\kappa}\longrightarrow H^{i-2a_\gamma}\big((\mathbf{M}_{r,L^\prime}^D(\gamma))_b, \CC \big)_{\kappa}
\]
is an isomorphism for every $b\in U$ and every $i$.

Take $U\subset\mathbf B^\heartsuit(\gamma)\subset\mathbf B^D(\gamma)^*$ to be the open subset on which the family
\begin{equation}
  \label{the open U} \mathbf C_{U}:=\mathbf C^D\times_{\mathbf{B}^D}U\longrightarrow U
\end{equation}
has at worst nodal singularities. Then
\[
  \mathbf C_{U}(\gamma):=\mathbf C^D(\gamma) \times_{\mathbf{B}^D(\gamma)}U\longrightarrow U
\]
is the simultaneous normalization of $\mathbf C_U$. On this locus, $[\mathscr C_b]_{\#,\kappa}$ is the isomorphism described in \cite[Lemma~A.3.1 and Lemma~3.4.1]{Yun09}. We give an equivalent description in Lemma~\ref{lem:identification to the image}.

\subsection{Parabolic modules over the nodal locus}\label{subsec:nodal-parabolic-modules}

Take an \'{e}tale cover
\begin{equation}
  \label{the final etale cover} \mathbb{V}\longrightarrow U\times_{\mathbf B} B
\end{equation}
with the following properties.
\begin{enumerate}
\item For the base change
\[
  \mathbf C_{\mathbb{V}}:= \mathbf C_U\times_U \mathbb{V}\longrightarrow \mathbb{V},
\]
there are sections $s_i: \mathbb{V}\rightarrow \mathbf C_{\mathbb{V}}$, $0\leq i \leq a_\gamma$, so that $s_0$ is the pullback of the section $s:B\to C$ fixed in \eqref{first etale cover}, and hence lies in the smooth locus of every fiber. The images of $s_i$, for $1\leq i\leq a_\gamma$, are exactly the nodal points of the fibers of $\mathbf C_{\mathbb V}$.
\item For the base change
\[
  \mathbf C_{\mathbb{V}}(\gamma):= \mathbf C_U(\gamma)\times_U \mathbb{V}\longrightarrow \mathbb{V},
\]
there are $2a_\gamma$ sections $s_i^{\pm}: \mathbb{V}\rightarrow \mathbf{C}_{\mathbb V}(\gamma)$, $1\leq i \leq a_\gamma$, so that $u_{\mathbf C_{\mathbb{V}}}\circ s_i^\pm=s_i$, where $u_{\mathbf C_{\mathbb{V}}}$ is the base change of $u_{\sC}$ in~\eqref{family-wise normalization}. Since $U$ is an open subset of $\mathbf{B}^\heartsuit(\gamma)$, $u_{\mathbf C_{\mathbb{V}}}$ is a simultaneous normalization, and the images of the $s_i^\pm$ form the inverse image of the nodal sections of $\mathbf C_{\mathbb{V}}$.
\end{enumerate}

Let $\check{\mathbb P}$ be the compactified norm-zero Prym of $\mathbf C_{\mathbb V}\to\mathbb V$, and let $\check{\mathbb P}(\gamma)$ be the norm-zero Prym of the smooth family $\mathbf C_{\mathbb V}(\gamma)\to\mathbb V$. The latter is a smooth proper group scheme, not necessarily connected. Write $\pi_{\mathbb V}:\mathbf C_{\mathbb V}\to\Sigma\times\mathbb V$ and $\pi_{\mathbb V}(\gamma):\mathbf C_{\mathbb V}(\gamma)\to \Sigma\times\mathbb V$ for the spectral maps. We denote by $\check{\mathbb P}^{-a_\gamma}(\gamma)$ the norm fiber over
\[
  \sO_{\Sigma\times\mathbb V} \left(-\sum_{i=1}^{a_\gamma} \operatorname{Im}(\pi_{\mathbb V}\circ s_i)\right).
\]
It is a torsor under $\check{\mathbb P}(\gamma)$ and has the same component structure fiberwise.

Pull back the section $s$ and the line bundle $\mathscr N_L$ from \eqref{first etale cover} and \eqref{eq:norm-correction-line} to $\mathbb V$. Correcting the two spectral sheaves by $\sO(-d_L\operatorname{Im}s_0)\otimes\pi_{\mathbb V}^*\mathscr N_L$ and by its pullback to the normalization gives the compatible BNR identifications
\begin{equation}
  \label{identifications} \mathbf M_{r,L'}^D(\gamma) \times_{\mathbf B^D(\gamma)}\mathbb V \simeq\check{\mathbb P}(\gamma), \qquad \mathbf M_{n,L}^D\times_{\mathbf B^D}\mathbb V \simeq\check{\mathbb P}.
\end{equation}
By \cite[Lemma~5.14(2)]{FHHO26}, pushforward along $u_{\mathbf C_{\mathbb{V}}}$ gives an embedding
\begin{equation}
  \xymatrix{ \check{\mathbb{P}}^{-a_\gamma}(\gamma) \ar[dr] \ar@{^{(}->}[rr]^{\nu} & & \check{\mathbb{P}} \ar[dl] \\ & \mathbb{V}. & }
\end{equation}

Under the same degree and norm correction, $\check{\mathbb P}^{-a_\gamma}(\gamma)$ is identified with $\mathbf M_{r,L}^D(\gamma) \times_{\mathbf B^D(\gamma)}\mathbb V$. The morphism $\nu$ corresponds to pushforward from the normalization.

In fact, on this nodal open, $\nu$ identifies its source with the scheme-theoretic fixed locus $\check{\mathbb P}_\gamma :=\check P_\gamma\times_B\mathbb V$:
\begin{equation}
  \label{eq:normalization-identifies-fixed-sector} \check{\mathbb P}^{-a_\gamma}(\gamma) \xrightarrow[\ \nu\ ]{\ \simeq\ } \check{\mathbb P}_\gamma.
\end{equation}

The map $\mathbb V\to\mathbf B_\gamma\times_{\mathbf B}B$ is \'{e}tale by the choice of $U$ and \eqref{the final etale cover}. Since $\check P_\gamma\to B$ factors through this endoscopic base, $\check{\mathbb P}_\gamma\to\check P_\gamma$ is \'{e}tale. The fixed locus $\check P_\gamma$ is smooth, so $\check{\mathbb P}_\gamma$ is reduced. Moreover, $\pi_{\mathbb V}(\gamma)^*\gamma$ is trivial, so $\gamma$ acts trivially on $\check{\mathbb P}^{-a_\gamma}(\gamma)$ and the equivariant closed immersion $\nu$ factors through $\check{\mathbb P}_\gamma$. The cyclic-cover description of fixed Higgs bundles shows that this closed immersion is surjective on geometric points: their spectral sheaves are pushforwards of line bundles on the normalization, with the norm prescribed above. It is therefore an isomorphism onto the reduced scheme $\check{\mathbb P}_\gamma$. In particular, $[\check{\mathbb P}^{-a_\gamma}(\gamma)/\Gamma]\to[\check P_\gamma/\Gamma]$ is representable and \'{e}tale.

The section $s_0$ lies in the smooth locus and therefore lifts uniquely to a section $s_0(\gamma)$ of the simultaneous normalization $\mathbf C_{\mathbb V}(\gamma)\to\mathbb V$. Let $\mathscr{U}_{\mathbf C_{\mathbb{V}}(\gamma)}$ be the universal line bundle on $\mathbf C_{\mathbb{V}}(\gamma)\times_{\mathbb{V}}\check{\mathbb{P}}(\gamma)$, rigidified by
\begin{equation}
  \label{rigidification} \mathscr{U}_{\mathbf C_{\mathbb{V}}(\gamma)}\mid_{\Im(s_0(\gamma))\times_{\mathbb{V}} \check{\mathbb{P}}(\gamma)} \simeq \mathcal{O}_{\check{\mathbb{P}}(\gamma)}.
\end{equation}
Restricting $\mathscr U_{\mathbf C_{\mathbb V}(\gamma)}$ to $\Im(s_i^\pm)\times_{\mathbb V}\check{\mathbb P}(\gamma)$ gives line bundles $\mathscr U_{\mathbf C_{\mathbb V},i}(\gamma)^\pm$ on $\check{\mathbb P}(\gamma)$. For each $i$, consider the associated projective bundle
\[
  \mathbb P\bigl( \mathscr U_{\mathbf C_{\mathbb V},i}(\gamma)^+ \oplus \mathscr U_{\mathbf C_{\mathbb V},i}(\gamma)^- \bigr)\longrightarrow\check{\mathbb P}(\gamma).
\]
We form their fiber product
\begin{equation}
  \label{fiber product of PPM} \mathbb{PPM}_{\mathbf C_{\mathbb{V}}} := \mathbb{P}\bigl(\mathscr{U}_{\mathbf C_{\mathbb{V}},1}(\gamma)^{+}\oplus \mathscr{U}_{\mathbf C_{\mathbb{V}},1}(\gamma)^{-}\bigr) \times_{\check{\mathbb{P}}(\gamma)} \cdots \times_{{\check{\mathbb{P}}(\gamma)}} \mathbb{P}\bigl(\mathscr{U}_{\mathbf C_{\mathbb{V}},a_\gamma}(\gamma)^{+}\oplus \mathscr{U}_{\mathbf C_{\mathbb{V}},a_\gamma}(\gamma)^{-}\bigr).
\end{equation}
This is the Prym version of the moduli space of parabolic modules of $(\mathbf C_{\mathbb{V}}(\gamma), \mathbf C_{\mathbb{V}})$.

For \(1\leq i\leq a_\gamma\), let \(\mathcal V_{\mathbf{C}_{\mathbb V},i}\) be the tautological line bundle on the \(i\)-th projective-bundle factor of \(\mathbb{PPM}_{\mathbf{C}_{\mathbb V}}\), and set
\[
  \mathcal V_{\mathbf{C}_{\mathbb V}} := \bigotimes_{i=1}^{a_\gamma} \mathcal V_{\mathbf{C}_{\mathbb V},i}.
\]
The projective-bundle construction of $\mathbb{PPM}_{\mathbf C_{\mathbb V}}$ is independent of the rigidification of $\mathscr U_{\mathbf C_{\mathbb V}(\gamma)}$, whereas the tautological line $\mathcal V_{\mathbf C_{\mathbb V}}$ depends on it.

For a parabolic module $(\tilde{\sL},\{\ell_i\})$, define the rank-one torsion-free sheaf $\sL$ by
\begin{equation}
  \label{defn of rho} 0\longrightarrow \sL \longrightarrow u_{\mathbf C_{\mathbb{V}}, *} \tilde{\sL}\longrightarrow u_{\mathbf C_{\mathbb{V}}, *}\Bigl(\bigoplus_{i=1}^{a_\gamma}\bigl((\tilde{\sL}\big|_{\Im s_i^+}\oplus \tilde{\sL}\big|_{\Im s_i^-})/\ell_i\bigr)\Bigr)\longrightarrow 0.
\end{equation}
This construction defines a morphism $\rho: \mathbb{PPM}_{\mathbf C_{\mathbb{V}}}\rightarrow \check{\mathbb P}$.

After replacing \(\mathbb V\) by a further étale cover if necessary, fix
\[
  L_0\in\Pic^0(\Sigma\times\mathbb V/\mathbb V)
\]
together with an isomorphism
\[
  L_0^{\otimes n} \simeq \sO_{\Sigma\times\mathbb V} \left( \sum_{i=1}^{a_\gamma} \Im(\pi_{\mathbb V}\circ s_i) - a_\gamma\Im(\pi_{\mathbb V}\circ s_0) \right).
\]
The line $\sN_0$ in \eqref{the line bundle N0} identifies the endoscopic moduli spaces with determinants $L'$ and $L$. By contrast, $L_0$ corrects the norm after twisting by $\sO(a_\gamma\operatorname{Im}s_0(\gamma))$. Thus
\[
  \mathsf T_{L_0,s_0}: \check{\mathbb P}^{-a_\gamma}(\gamma) \xrightarrow{\sim} \check{\mathbb P}(\gamma)
\]
is given by
\[
  \mathscr L\longmapsto \mathscr L\otimes \sO_{\mathbf{C}_{\mathbb V}(\gamma)} \bigl(a_\gamma\Im s_0(\gamma)\bigr) \otimes\pi_{\mathbb V}(\gamma)^*L_0.
\]

These constructions give the following diagram over $\mathbb V$:
\begin{equation}
  \label{morphismsssss} \xymatrix{ & & \mathbb{PPM}_{\mathbf C_{\mathbb{V}}} \ar[dll]_{\mathsf P} \ar[drr]^{\rho} & & \\ \check{\mathbb{P}}(\gamma) & & \check{\mathbb{P}}^{-a_\gamma}(\gamma) \ar[ll]_{\mathsf T_{L_0,s_0}} \ar@{^{(}->}[rr]^{\nu} & & \check{\mathbb P}. }
\end{equation}
Here $\mathsf P$ is the projection to $\check{\mathbb{P}}(\gamma)$ in~\eqref{fiber product of PPM}.

Taking determinants after pushing the normalization exact sequence and \eqref{defn of rho} forward to $\Sigma\times\mathbb V$ gives $\Nm_{\mathbf C_{\mathbb V}/\Sigma}(\sL)\simeq\Nm_{\mathbf C_{\mathbb V}(\gamma)/\Sigma}(\tilde{\sL})$. Indeed, the two quotients are sums of line bundles supported on the same nodal sections, so their determinants after pushforward agree. Thus the parabolic-module construction commutes with restriction to the norm fibers.

\begin{lemma}\label{properties of rho}
The morphism
\[
  \rho: \mathbb{PPM}_{\mathbf{C}_{\mathbb V}} \longrightarrow \check{\mathbb P}
\]
is finite and is a resolution of singularities.
\end{lemma}

\begin{proof}
For a geometric point $v\to\mathbb V$, the morphism $\rho_v$ is the restriction to the Prym loci of the parabolic-module morphism associated with the normalization
\[
  u_{\mathbf{C}_{\mathbb V},v}: \mathbf C_v(\gamma)\longrightarrow\mathbf C_v.
\]
By \cite[Equation~(3.5) and Lemma~3.7]{FHHO26}, its fibers are finite, and every rank-one torsion-free sheaf on \(\mathbf C_v\) is obtained from a parabolic module. The norm identity above shows that a preimage of a point of \(\check{\mathbb P}_{\mathbf C_v}\) belongs to \(\mathbb{PPM}_{\mathbf C_v}\). Thus \(\rho\) is surjective and has finite geometric fibers. Since $\mathbb{PPM}_{\mathbf C_{\mathbb V}}$ is projective over $\mathbb V$ and $\check{\mathbb P}$ is separated over $\mathbb V$, the morphism $\rho$ is proper. A proper quasi-finite morphism is finite.

The variety \(\mathbb{PPM}_{\mathbf{C}_{\mathbb V}}\) is an iterated \(\mathbb P^1\)-bundle over \(\check{\mathbb P}(\gamma)\), and hence is smooth. Let $\mathbb{PPM}_{\mathbf{C}_{\mathbb V}}^{\mathrm{lf}} \subset \mathbb{PPM}_{\mathbf{C}_{\mathbb V}}$ be the open locus on which, at every node, the tautological line has nonzero projection to both branches. The exact sequence \eqref{defn of rho} identifies its image with the locus of line bundles in \(\check{\mathbb P}\), and the gluing line at each node is uniquely recovered from the resulting line bundle. Thus $\rho$ is an isomorphism over this dense open subset. It is therefore birational and is a resolution of singularities.
\end{proof}

We define
\begin{equation}
  r: \mathbb{PPM}_{\mathbf C_{\mathbb V}} \xrightarrow{(\mathsf P,\rho)} \check{\mathbb P}(\gamma)\times_{\mathbb V}\check{\mathbb P}.
\end{equation}
This morphism is finite: it is the composite of the graph of $\mathsf P$ with the base change of the finite morphism $\rho$.

\begin{lemma}\label{lem:identification to the image}
Consider the base change $\mathscr C_{\mathbb V}:= \mathscr C\times_{\mathbf B^D(\gamma)} \mathbb V$. Under the identifications in~\eqref{identifications}, \(\mathscr C_{\mathbb V}\) is the scheme-theoretic image of $r:\mathbb{PPM}_{\mathbf{C}_{\mathbb V}} \rightarrow \check{\mathbb P}(\gamma) \times_{\mathbb V}\check{\mathbb P}$. Consequently,
\[
  r_*[\mathbb{PPM}_{\mathbf{C}_{\mathbb V}}] = [\mathscr C_{\mathbb V}]\in H_*^{\mathrm{BM}}\left( \check{\mathbb P}(\gamma) \times_{\mathbb V} \check{\mathbb P}, \mathbb C \right).
\]
\end{lemma}

\begin{proof}
Let \(\mathbb{PPM}_{\mathbf{C}_{\mathbb V}}^{\mathrm{lf}}\) be the dense open subset introduced in the proof of Lemma~\ref{properties of rho}. For a point
\[
  (\tilde{\sL},\{\ell_i\}) \in \mathbb{PPM}_{\mathbf{C}_{\mathbb V}}^{\mathrm{lf}},
\]
the sheaf \(\sL=\rho(\tilde{\sL},\{\ell_i\})\) is a line bundle on \(\mathbf{C}_{\mathbb V}\), and the construction gives $\tilde{\sL}\simeq u_{\mathbf{C}_{\mathbb V}}^*\sL$.

Conversely, a line bundle \(\sL\) determines a unique collection of gluing lines \(\{\ell_i\}\). Under the BNR identifications \eqref{identifications}, the restriction of \(r\) to \(\mathbb{PPM}_{\mathbf{C}_{\mathbb V}}^{\mathrm{lf}}\) is therefore an isomorphism onto the base change to \(\mathbb V\) of the graph of \(g_u\).

By definition, \(\mathscr C\) is the reduced Zariski closure of this graph. Since
\[
  \mathbb V\longrightarrow U \longrightarrow\mathbf B^D(\gamma)
\]
is the composition of an étale morphism and an open immersion, it is flat. Hence the base-changed graph is schematically dense in \(\mathscr C_{\mathbb V}\), and
\[
  \mathscr C_{\mathbb V} = \overline{r\bigl( \mathbb{PPM}_{\mathbf{C}_{\mathbb V}}^{\mathrm{lf}} \bigr)} = r \bigl( \mathbb{PPM}_{\mathbf{C}_{\mathbb V}} \bigr).
\]
The scheme-theoretic image of $r$ is reduced, since its structure sheaf injects into $r_*\sO_{\mathbb{PPM}_{\mathbf C_{\mathbb V}}}$ and the source is reduced. The same is true of $\mathscr C_{\mathbb V}$, since it is an \'etale base change of the reduced graph closure. The preceding equality of supports therefore identifies these closed subschemes. Since $r$ is an isomorphism over the dense graph, its generic degree is one on every component of its image. Proper pushforward of fundamental cycles gives
\[
  r_*[\mathbb{PPM}_{\mathbf{C}_{\mathbb V}}] = [\mathscr C_{\mathbb V}],
\]
which proves the last assertion.
\end{proof}

\subsection{Relation between Poincar\'e sheaves}\label{sec:poincare-relation}

Consider the Poincar\'e sheaf
\begin{equation}
  \label{the poincare sheaf R V} \mathscr{R}_{\mathbb{V}} \quad \text{ on }\quad \check{\mathbb P}\times_{\mathbb V} \check{\mathbb P},
\end{equation}
and the Poincar\'e line bundle
\begin{equation}
  \label{the Poincare line bundle R V} \mathscr{R}_{\mathbb{V}}(\gamma )\quad \text{ on }\quad \check{\mathbb{P}}(\gamma )\times_{\mathbb{V}} \check{\mathbb{P}}(\gamma ),
\end{equation}
which are defined as in~\eqref{the sheaf R}.

Let \(e:\mathbb V\to\check{\mathbb P}\) and \(e(\gamma):\mathbb V\to\check{\mathbb P}(\gamma)\) denote the unit sections. We choose the Poincaré sheaves with their normalized rigidifications
\begin{align}
  (e\times_{\mathbb V}\id)^*\mathscr R_{\mathbb V} &\simeq \sO_{\check{\mathbb P}}, & (\id\times_{\mathbb V}e)^*\mathscr R_{\mathbb V} &\simeq \sO_{\check{\mathbb P}}, \label{eq:relative-Poincare-rigidification}\\
  (e(\gamma)\times_{\mathbb V}\id)^* \mathscr R_{\mathbb V}(\gamma) &\simeq \sO_{\check{\mathbb P}(\gamma)}, & (\id\times_{\mathbb V}e(\gamma))^* \mathscr R_{\mathbb V}(\gamma) &\simeq \sO_{\check{\mathbb P}(\gamma)}. \nonumber
\end{align}
These rigidifications are chosen compatibly with \eqref{rigidification} and with the symmetry of the Poincaré sheaves. We use the parabolic-module space $\mathbb{PPM}_{\mathbf C_{\mathbb V}}$ and the maps in~\eqref{morphismsssss} to compare these sheaves, following \cite{FHHO26}.

\[
\begin{tikzcd}[column sep=small]
  \mathbb{PPM}_{\mathbf{C}_{\mathbb V}}\times_{\mathbb V}\check{\mathbb P}^{-a_\gamma}(\gamma) \arrow[r,"\rho\times\id"]\arrow[d,"q_1"'] &\check{\mathbb P}\times_{\mathbb V}\check{\mathbb P}^{-a_\gamma}(\gamma) \arrow[r,"\id\times\nu"]\arrow[d,"\operatorname{pr}_1"'] &\check{\mathbb P}\times_{\mathbb V}\check{\mathbb P}\arrow[d,"\operatorname{pr}_1"]\\
  \mathbb{PPM}_{\mathbf{C}_{\mathbb V}}\arrow[r,"\rho"'] &\check{\mathbb P}\arrow[r,"\id"']&\check{\mathbb P}.
\end{tikzcd}
\]

\[
\begin{tikzcd}[column sep=large]
  \mathbb{PPM}_{\mathbf{C}_{\mathbb V}} \times_{\mathbb V} \check{\mathbb P}^{-a_\gamma}(\gamma) \arrow[r,"\mathsf P\times\mathsf T_{L_0,s_0}"] \arrow[d,"q_1"'] & \check{\mathbb P}(\gamma) \times_{\mathbb V} \check{\mathbb P}(\gamma) \arrow[d,"\operatorname{pr}_1"] \\
  \mathbb{PPM}_{\mathbf{C}_{\mathbb V}} \arrow[r,"\mathsf P"'] & \check{\mathbb P}(\gamma).
\end{tikzcd}
\]

\begin{proposition}\label{prop: comparision of Poincare family version}
There is an isomorphism of sheaves on $\check{\mathbb P}\times_{\mathbb V} \check{\mathbb P}^{-a_\gamma}(\gamma)$:
\begin{equation}
  \label{change of Poincare sheaf} (\id\times_{\mathbb{V}} \nu)^{*}\mathscr{R}_{\mathbb{V}}\; \cong\; (\rho \times_{\mathbb{V}} \id)_{*}\Bigl( q_1^{*}\mathcal{V}_{\mathbf C_{\mathbb{V}}} \;\otimes\; (\mathsf P\times_{\mathbb{V}} \mathsf T_{L_0,s_0})^{*}\mathscr{R}_{\mathbb{V}}(\gamma) \Bigr),
\end{equation}
where \(q_1: \mathbb{PPM}_{\mathbf C_{\mathbb{V}}} \times_{\mathbb{V}} \check{\mathbb{P}}^{-a_\gamma}(\gamma) \to \mathbb{PPM}_{\mathbf C_{\mathbb{V}}}\) is the projection to the first factor and $\mathcal{V}_{\mathbf C_{\mathbb{V}}}$ is the tautological line bundle defined above.
\end{proposition}

\begin{proof}
Let $\overline{\mathscr P}_{\mathbb V}$ be the normalized Poincar\'e sheaf on $\overline J_{\mathbf{C}_{\mathbb V}}^0 \times_{\mathbb V}\overline J_{\mathbf{C}_{\mathbb V}}^0$, and let $\overline{\mathscr P}_{\mathbb V}(\gamma)$ be the normalized Poincar\'e line bundle on $J_{\mathbf{C}_{\mathbb V}(\gamma)}^0 \times_{\mathbb V}J_{\mathbf{C}_{\mathbb V}(\gamma)}^0$. Write $\overline{\mathbb{PM}}_{\mathbf{C}_{\mathbb V}}$ for the parabolic module space over $J_{\mathbf{C}_{\mathbb V}(\gamma)}^0$, and use $\overline\rho$, $\overline{\mathsf P}$, and $\overline{\mathcal V}_{\mathbf{C}_{\mathbb V}}$ for the corresponding maps and tautological line bundle. Let
\[
  \tau_{a_\gamma,s_0(\gamma)}: J_{\mathbf{C}_{\mathbb V}(\gamma)}^{-a_\gamma} \longrightarrow J_{\mathbf{C}_{\mathbb V}(\gamma)}^0, \qquad \mathscr L\longmapsto \mathscr L\otimes \sO_{\mathbf{C}_{\mathbb V}(\gamma)} \bigl(a_\gamma\Im s_0(\gamma)\bigr).
\]
Define the sheaf
\[
  \overline{\mathscr G} := (\overline\rho\times_{\mathbb V}\id)_* \Bigl( \overline q_1^*\overline{\mathcal V}_{\mathbf{C}_{\mathbb V}} \otimes (\overline{\mathsf P} \times_{\mathbb V}\tau_{a_\gamma,s_0(\gamma)})^* \overline{\mathscr P}_{\mathbb V}(\gamma) \Bigr)
\]
on $\overline J_{\mathbf{C}_{\mathbb V}}^0 \times_{\mathbb V}J_{\mathbf{C}_{\mathbb V}(\gamma)}^{-a_\gamma}$, where $\overline q_1$ is the projection to $\overline{\mathbb{PM}}_{\mathbf{C}_{\mathbb V}}$.

The sheaf under the pushforward is a line bundle. Since $\overline{\mathbb{PM}}_{\mathbf{C}_{\mathbb V}}$ is a product of projective bundles over the smooth relative Jacobian $J_{\mathbf{C}_{\mathbb V}(\gamma)}^0$, this line bundle is flat over $J_{\mathbf{C}_{\mathbb V}(\gamma)}^{-a_\gamma}$. As in the proof of Lemma~\ref{properties of rho}, $\overline\rho$ is proper with finite geometric fibers, hence finite. Its pushforward therefore preserves flatness over this parameter and commutes with arbitrary base change. Thus $\overline{\mathscr G}$ is a flat family of rank-one torsion-free sheaves.

By the universal property, $\overline{\mathscr G}$ defines a morphism
\[
  \varphi_{\overline{\mathscr G}}: J_{\mathbf{C}_{\mathbb V}(\gamma)}^{-a_\gamma} \longrightarrow \overline{\operatorname{Pic}} \bigl(\overline J_{\mathbf{C}_{\mathbb V}}^0/\mathbb V\bigr).
\]
The Poincar\'e family defines the autoduality morphism
\[
  \beta: \overline J_{\mathbf{C}_{\mathbb V}}^0 \longrightarrow \overline{\operatorname{Pic}} \bigl(\overline J_{\mathbf{C}_{\mathbb V}}^0/\mathbb V\bigr).
\]
By the relative form of \cite[Theorem~B]{Ari13}, discussed in the remarks following its Theorem~C, $\overline{\mathscr P}_{\mathbb V}$ is universal on the component identified by $\beta$.

Let
\[
  \overline\nu_{\mathrm{Jac}}: J_{\mathbf{C}_{\mathbb V}(\gamma)}^{-a_\gamma} \longrightarrow \overline J_{\mathbf{C}_{\mathbb V}}^0
\]
be induced by pushforward under $u_{\mathbf{C}_{\mathbb V}}$. For a geometric point $b\to\mathbb V$ and $\mathscr L\in J_{\mathbf C_b(\gamma)}^{-a_\gamma}$, finite base change and \cite[Theorem~5.3]{FHHO26}, with the factors interchanged, identify $\overline{\mathscr G}|_{\overline J_{\mathbf C_b}^0\times\{\mathscr L\}}$ with the Poincar\'e sheaf corresponding to $u_{\mathbf{C}_{\mathbb V},b*}\mathscr L$. Hence $\varphi_{\overline{\mathscr G}}$ and $\beta\circ\overline\nu_{\mathrm{Jac}}$ agree on all geometric points. The parameter $J_{\mathbf{C}_{\mathbb V}(\gamma)}^{-a_\gamma}$ is smooth over the smooth base $\mathbb V$, and is therefore reduced. Since the target is separated, their equalizer is a closed subscheme containing every geometric point of the reduced parameter scheme, and is therefore the whole parameter scheme. Thus
\[
  \varphi_{\overline{\mathscr G}} =\beta\circ\overline\nu_{\mathrm{Jac}}.
\]
The universal property of $\overline{\mathscr P}_{\mathbb V}$ now gives
\[
  \overline{\mathscr G} \simeq (\id\times_{\mathbb V}\overline\nu_{\mathrm{Jac}})^* \overline{\mathscr P}_{\mathbb V} \otimes p_2^*\mathscr N
\]
for a line bundle $\mathscr N$ on $J_{\mathbf{C}_{\mathbb V}(\gamma)}^{-a_\gamma}$. The rigidifications in~\eqref{eq:relative-Poincare-rigidification} give $\mathscr N\simeq\sO$.

We now restrict to $\check{\mathbb P}\times_{\mathbb V} \check{\mathbb P}^{-a_\gamma}(\gamma)$. The norm identity preceding Lemma~\ref{properties of rho} identifies the inverse image of $\check{\mathbb P}$ under $\overline\rho$ with $\mathbb{PPM}_{\mathbf{C}_{\mathbb V}}$. Ordinary base change for the finite morphism $\overline\rho\times\id$ therefore restricts the ambient pushforward to the corresponding Prym pushforward.

The map $\mathsf T_{L_0,s_0}$ differs from $\tau_{a_\gamma,s_0(\gamma)}$ by tensoring with $\pi_{\mathbb V}(\gamma)^*L_0$. The translation formula for the Poincar\'e line bundle changes the kernel by the Poincar\'e line associated with $\pi_{\mathbb V}(\gamma)^*L_0$ on the first factor. By compatibility of the Poincar\'e pairing with the norm, this line is the pullback, under the norm map, of the Poincar\'e line associated with $L_0$ on $J_\Sigma^0$. Its restriction to the norm-zero locus $\check{\mathbb P}(\gamma)$ is canonically trivial. Thus the restricted formula is unchanged when $\tau_{a_\gamma,s_0(\gamma)}$ is replaced by $\mathsf T_{L_0,s_0}$, as in \cite[proof of Proposition~5.16]{FHHO26}. This gives \eqref{change of Poincare sheaf}.
\end{proof}

\subsection{Cohomological comparison and Riemann--Roch correction}\label{subsec:endoscopic-RR-comparison}

For the Fourier--Mukai compositions below, put $r^{\mathrm t}:=(\rho,\mathsf P)$, the transpose of $r=(\mathsf P,\rho)$. We transport $[\mathscr C_{\mathbb V}]$ by the same exchange of factors when it is viewed as a correspondence from $\check{\mathbb P}$ to $\check{\mathbb P}(\gamma)$.

On $\mathscr R_{\mathbb V}$ we use the restriction of the $\Gamma$-linearization defining $\mathscr F$. On $\mathscr R_{\mathbb V}(\gamma)$ we use the $\Gamma$-linearization obtained by restricting the $J_\Sigma^0$-descent structure on the normalized Poincar\'e line bundle over
\[
  \check{\mathbb P}(\gamma)\times_{\mathbb V} J^0_{\mathbf C_{\mathbb V}(\gamma)},
\]
as in the construction following \cite[Proposition~4.4]{GS22}. Denote the corresponding descents by $\mathscr F_{\mathbb V}$ and $\mathscr F_{\mathbb V}(\gamma)$. The group acts only on the second factors; $q_1^*\mathcal V_{\mathbf C_{\mathbb V}}$ has the linearization pulled back from the trivial action on $\mathbb{PPM}_{\mathbf C_{\mathbb V}}$.

For $\chi\in\widehat\Gamma$, write $\mathbb C_\chi$ for the trivial line bundle with $\Gamma$ acting through $\chi$, or for its descent from $B\Gamma$ to the indicated quotient stack. We work on a connected component of $\mathbb V$ and retain the ambient sheaf $\overline{\mathscr G}$ from the proof of Proposition~\ref{prop: comparision of Poincare family version}. We compare the fixed descents using the ordinary isomorphism constructed there.

The two translation formulas have the same norm-pullback Poincar\'e factor and the same composition law, by the determinant-of-cohomology norm identity \cite[Proposition~4.4]{GS22} and the norm identity for parabolic modules.

Every automorphism of $\overline{\mathscr G}$ is a scalar pulled back from $\mathbb V$: indeed, $\overline{\mathscr G}$ is the finite birational pushforward of a line bundle from $\overline{\mathbb{PM}}_{\mathbf C_{\mathbb V}} \times_{\mathbb V} J_{\mathbf C_{\mathbb V}(\gamma)}^{-a_\gamma}$, which is smooth and proper over $\mathbb V$ with geometrically connected fibers. Thus the ratios of the two translation structures are units on $\mathbb V$, and their cocycle condition defines a character $\chi_\gamma\in\widehat\Gamma$ on each connected component of $\mathbb V$. Restricting to the Prym loci, with the norm-pairing trivialization used for $\mathsf T_{L_0,s_0}$, gives the following comparison of the fixed descents.
\begin{equation}
  \label{eq:quotient-Poincare-comparison} (\id\times_{\mathbb V}\overline\nu)^*\mathscr F_{\mathbb V} \simeq{} (\rho\times_{\mathbb V}\id)_* \Bigl( q_1^*\mathcal V_{\mathbf C_{\mathbb V}} \otimes (\mathsf P\times_{\mathbb V} \overline{\mathsf T}_{L_0,s_0})^* \mathscr F_{\mathbb V}(\gamma) \Bigr) \otimes p_2^*\mathbb C_{\chi_\gamma}
\end{equation}
on $\check{\mathbb P}\times_{\mathbb V} [\check{\mathbb P}^{-a_\gamma}(\gamma)/\Gamma]$, where $p_2$ is the second projection, and $\overline\nu$ and $\overline{\mathsf T}_{L_0,s_0}$ are induced by $\nu$ and $\mathsf T_{L_0,s_0}$ on the corresponding quotient stacks.

As a convolution of kernels, the right-hand side is
\[
  \Bigl( (\id\times_{\mathbb V}\overline{\mathsf T}_{L_0,s_0})^* \mathscr F_{\mathbb V}(\gamma) \otimes p_2^*\mathbb C_{\chi_\gamma} \Bigr) \circ Rr^{\mathrm t}_*\mathcal V_{\mathbf C_{\mathbb V}},
\]
where $p_2$ projects from $\check{\mathbb P}(\gamma)\times_{\mathbb V} [\check{\mathbb P}^{-a_\gamma}(\gamma)/\Gamma]$ to the second factor. The convolution is taken along the following diagram:
\[
  \xymatrix{ \mathbb{PPM}_{\mathbf C_{\mathbb{V}}} \ar[r]^{r^{\mathrm t}} & \check{\mathbb P}\times_{\mathbb{V}} \check{\mathbb{P}}(\gamma) \ar[dl] \ar[dr] & & \check{\mathbb{P}}(\gamma)\times_{\mathbb{V}} [\check{\mathbb{P}}^{-a_\gamma}(\gamma)/\Gamma] \ar[dl] \ar[dr] & \\ \check{\mathbb P} & & \check{\mathbb{P}}(\gamma) & & [\check{\mathbb{P}}^{-a_\gamma}(\gamma)/\Gamma]. }
\]
The source $\check{\mathbb P}$ may be singular. We use the first definition in \eqref{eq:normalized-tau-supported-definition}; its target is smooth. Homological realization and composition require only the target and the intermediate space to be smooth \cite[Subsection~2.2.2]{MSY23}.

\begin{proposition}\label{realization of compositin of functors}
We have
\begin{equation}
  \label{composition of functors} \tau^{\mathrm{nor}}_{\mathbb V} \bigl([(\mathrm{id}\times_{\mathbb V}\overline\nu)^* \mathscr F_{\mathbb V}]\bigr)_{1,\gamma}= \chi_\gamma(\gamma)\, \overline{\mathsf T}_{L_0,s_0}^{\,*} \circ\tau^{\mathrm{nor}}_{\mathbb V} (\mathscr F_{\mathbb V}(\gamma))_{1,\gamma} \circ\tau^{\mathrm{nor}} \bigl([Rr^{\mathrm t}_*\mathcal V_{\mathbf C_{\mathbb V}}]\bigr).
\end{equation}
\end{proposition}

\begin{proof}
The pullback in \eqref{eq:quotient-Poincare-comparison} agrees with derived pullback, since $\mathscr F_{\mathbb V}$ is flat over its second factor. The morphism $\mathsf P$ is smooth and projective, $\rho$ is finite, and $\mathscr F_{\mathbb V}(\gamma)$ is a line bundle. Flat base change and the projection formula therefore identify \eqref{eq:quotient-Poincare-comparison} with the displayed bounded convolution.

Since $\gamma$ acts trivially on the normalized Prym and its torsor, Riemann--Roch on the $(1,\gamma)$-sector is ordinary equivariant Riemann--Roch after applying $t_\gamma$. The character line has inertial Chern character
\[
  \widetilde{\ch}(\mathbb C_{\chi_\gamma})_\gamma =\chi_\gamma(\gamma).
\]

Proper covariance, the module formula, and smooth pullback \cite[Theorem~3.1(b), (c), and (d)(i)]{EG00} identify both sides of \eqref{composition of functors} with
\[
  \chi_\gamma(\gamma)\, (\rho\times_{\mathbb V}\mathrm{id})_* \Bigl(q_1^*\bigl(\ch(\mathcal V_{\mathbf C_{\mathbb V}}) \td(T_{\mathsf P})\bigr)\cap{} (\mathsf P\times_{\mathbb V}\overline{\mathsf T}_{L_0,s_0})^* \tau^{\mathrm{nor}}_{\mathbb V} (\mathscr F_{\mathbb V}(\gamma))_{1,\gamma}\Bigr).
\]
For the right-hand side, use $\operatorname{pr}_2\circ r^{\mathrm t}=\mathsf P$ and the refined pullback along the absolute diagonal of the smooth intermediate scheme $\check{\mathbb P}(\gamma)$ \cite[Section~6.2]{Ful98}; the target Todd factors agree under $\overline{\mathsf T}_{L_0,s_0}$.

We also compare with the ambient Fourier class. Lemma~\ref{lem:Poincare-sector-actual-pullback} and \eqref{eq:normalized-tau-supported-definition} give
\[
  \tau^{\mathrm{nor}}_B([\mathscr F])_{1,\gamma} =\operatorname{pr}_2^*\td(T_{[\check P_\gamma/\Gamma]})^{-1} \cap\tau_{\mathrm{BM}}\bigl(t_\gamma([L\ell_\gamma^*\mathscr F])\bigr),
\]
since the moving-normal factor cancels in the target normalization. Pulling back on the second factor along the \'{e}tale map following \eqref{eq:normalization-identifies-fixed-sector} introduces no Todd correction. Compatibility of Riemann--Roch with this pullback and proper base change give
\begin{equation}
  \label{eq:nodal-ambient-Fourier-restriction} \tau^{\mathrm{nor}}_{\mathbb V} \bigl([(\mathrm{id}\times_{\mathbb V}\overline\nu)^* \mathscr F_{\mathbb V}]\bigr)_{1,\gamma} =\overline\nu^*\circ \left(\sum_i\mathfrak F_{\gamma,i}\right)\big|_{\mathbb V}
\end{equation}
as total graded morphisms between the unshifted direct images. The ambient age shift is retained when indexing the homogeneous components.
\end{proof}

We now compute $\tau^{\mathrm{nor}}([Rr^{\mathrm t}_*\mathcal{V}_{\mathbf C_{\mathbb{V}}}])$. The kernel is supported on
\[
  \check{\mathbb P} \times_{\mathbb V}\check{\mathbb P}(\gamma).
\]
The same convention applies to $[\mathscr C_{\mathbb V}]$. Write $p_2$ for the projection to $\check{\mathbb P}(\gamma)$. By the definition of the target normalization,
\[
  \tau^{\mathrm{nor}}([Rr^{\mathrm t}_*\mathcal V_{\mathbf{C}_{\mathbb V}}]) = p_2^*\operatorname{td}(T_{\check{\mathbb P}(\gamma)})^{-1} \cap\tau_{\operatorname{BM}} ([Rr^{\mathrm t}_*\mathcal V_{\mathbf{C}_{\mathbb V}}]).
\]

\begin{lemma}\label{key lemma}
There is a $\Gamma$-equivariant invertible correspondence
\[
  \mathfrak U_\gamma =\sum_{s\geq0}\mathfrak U_{\gamma,s}, \qquad \mathfrak U_{\gamma,s}\in \operatorname{Corr}^{s}_{\mathbb V} \bigl(\check{\mathbb P}(\gamma),\check{\mathbb P}(\gamma)\bigr), \qquad \mathfrak U_{\gamma,0}=\mathrm{id},
\]
such that
\begin{equation}
  \label{eq:endoscopic-RR-correction} \tau^{\mathrm{nor}}([Rr^{\mathrm t}_*\mathcal V_{\mathbf{C}_{\mathbb V}}]) = \mathfrak U_\gamma\circ[\mathscr C_{\mathbb V}].
\end{equation}
\end{lemma}

\begin{proof}
The graph factorization of $r^{\mathrm t}$ is a closed immersion followed by a base change of the finite morphism $\rho$. Thus $r^{\mathrm t}$ is proper. The variety $\mathbb{PPM}_{\mathbf{C}_{\mathbb V}}$ is smooth, and $\mathsf P$ is smooth of relative dimension $a_\gamma$. Proper covariance of $\tau_{\operatorname{BM}}$, the projection formula, and $p_2\circ r^{\mathrm t}=\mathsf P$ give
\begin{align*}
  \tau^{\mathrm{nor}}([Rr^{\mathrm t}_*\mathcal V_{\mathbf{C}_{\mathbb V}}]) &=r^{\mathrm t}_*\Bigl( \ch(\mathcal V_{\mathbf{C}_{\mathbb V}}) \td(T_{\mathbb{PPM}_{\mathbf{C}_{\mathbb V}}}) \mathsf P^*\td(T_{\check{\mathbb P}(\gamma)})^{-1} \cap[\mathbb{PPM}_{\mathbf{C}_{\mathbb V}}]\Bigr)\\
  &=r^{\mathrm t}_*\Bigl( \ch(\mathcal V_{\mathbf{C}_{\mathbb V}}) \td(T_{\mathsf P}) \cap[\mathbb{PPM}_{\mathbf{C}_{\mathbb V}}]\Bigr).
\end{align*}
By Lemma~\ref{lem:identification to the image}, $r^{\mathrm t}_*[\mathbb{PPM}_{\mathbf{C}_{\mathbb V}}] =[\mathscr C_{\mathbb V}]$. We compute the correction $\ch(\mathcal V_{\mathbf{C}_{\mathbb V}})\td(T_{\mathsf P})$ using the projective-bundle presentation.

Put
\[
  \mathbb{PPM}_{\mathbf{C}_{\mathbb V},i} := \mathbb P\bigl( \mathscr U_{\mathbf{C}_{\mathbb V},1}(\gamma)^+ \oplus\mathscr U_{\mathbf{C}_{\mathbb V},1}(\gamma)^-\bigr) \times_{\check{\mathbb P}(\gamma)}\cdots \times_{\check{\mathbb P}(\gamma)} \mathbb P\bigl( \mathscr U_{\mathbf{C}_{\mathbb V},i}(\gamma)^+ \oplus\mathscr U_{\mathbf{C}_{\mathbb V},i}(\gamma)^-\bigr),
\]
with $\mathbb{PPM}_{\mathbf{C}_{\mathbb V},0}=\check{\mathbb P}(\gamma)$. Let $p_i:\mathbb{PPM}_{\mathbf{C}_{\mathbb V},i}\to \mathbb{PPM}_{\mathbf{C}_{\mathbb V},i-1}$ be the natural projection and $q_i=p_1\circ\cdots\circ p_i$. Set
\[
  \xi_i=c_1(\sO_{p_i}(1)), \qquad \alpha_i^\pm =c_1\bigl(\mathscr U_{\mathbf{C}_{\mathbb V},i}(\gamma)^\pm\bigr).
\]
With the tautological sub-line convention, the projective-bundle relation is
\begin{equation}
  \label{relation of cohomology in a P1 bundle} \xi_i^2+q_i^*(\alpha_i^++\alpha_i^-)\xi_i +q_i^*(\alpha_i^+\alpha_i^-)=0.
\end{equation}
Moreover,
\[
  T_{p_i} \simeq \sO_{p_i}(2)\otimes q_i^*\bigl( \mathscr U_{\mathbf{C}_{\mathbb V},i}(\gamma)^+ \otimes\mathscr U_{\mathbf{C}_{\mathbb V},i}(\gamma)^-\bigr),
\]
and there is an exact sequence
\[
  0\longrightarrow T_{p_i}\longrightarrow T_{q_i} \longrightarrow p_i^*T_{q_{i-1}}\longrightarrow0.
\]
We use the same notation for the pullbacks of these bundles and classes to $\mathbb{PPM}_{\mathbf C_{\mathbb V}}$. By definition,
\begin{equation}
  \label{canonical line bundle of PMod} \mathcal V_{\mathbf{C}_{\mathbb V}} \simeq \bigotimes_{i=1}^{a_\gamma}\sO_{p_i}(-1).
\end{equation}

The projective-bundle relation gives
\[
  \bigl(2\xi_i+q_i^*(\alpha_i^++\alpha_i^-)\bigr)^2 =q_i^*\bigl((\alpha_i^+-\alpha_i^-)^2\bigr).
\]
Put $z=c_1(T_{p_i})=2\xi_i+q_i^*(\alpha_i^++\alpha_i^-)$. Since $\td(T_{p_i})=z/(1-e^{-z})$ and $e^{-\xi_i}e^{z/2}=e^{q_i^*(\alpha_i^++\alpha_i^-)/2}$, the first equality in \eqref{eq:RR-correction-descends} is the identity $e^{-\xi_i}z/(1-e^{-z})= e^{q_i^*(\alpha_i^++\alpha_i^-)/2}z/(2\sinh(z/2))$. For the second equality, use the fact that $z/(2\sinh(z/2))$ is even and the projective-bundle relation gives $z^2=q_i^*((\alpha_i^+-\alpha_i^-)^2)$. Hence
\begin{equation}
  \label{eq:RR-correction-descends}
\begin{aligned}
  \ch(\sO_{p_i}(-1))\td(T_{p_i}) &= e^{q_i^*(\alpha_i^++\alpha_i^-)/2} \frac{c_1(T_{p_i})}{2\sinh(c_1(T_{p_i})/2)}\\
  &= q_i^*\left( e^{(\alpha_i^++\alpha_i^-)/2} \frac{\alpha_i^+-\alpha_i^-}{2\sinh((\alpha_i^+-\alpha_i^-)/2)} \right).
\end{aligned}
\end{equation}
The quotients are interpreted as power series.

Define $\mathfrak U_\gamma$ to be the diagonal correspondence given by cup product on $\check{\mathbb P}(\gamma)$ with
\[
  \prod_{i=1}^{a_\gamma} e^{(\alpha_i^++\alpha_i^-)/2} \frac{\alpha_i^+-\alpha_i^-}{2\sinh((\alpha_i^+-\alpha_i^-)/2)}.
\]
The relative tangent sequences give $\td(T_{\mathsf P})=\prod_{i=1}^{a_\gamma}\td(T_{p_i})$. Together with \eqref{canonical line bundle of PMod} and \eqref{eq:RR-correction-descends}, this identifies the pullback of the displayed class by $\mathsf P$ with $\ch(\mathcal V_{\mathbf{C}_{\mathbb V}})\td(T_{\mathsf P})$. The projection formula and Lemma~\ref{lem:identification to the image} therefore give \eqref{eq:endoscopic-RR-correction}.

Now fix a geometric point $b\to\mathbb V$. On the smooth connected curve $\mathbf C_b(\gamma)$, the first Chern class of an evaluation line bundle is independent of the point of evaluation. It vanishes at $s_0(\gamma)(b)$ by \eqref{rigidification}. Hence
\[
  \alpha_i^\pm|_{\check{\mathbb P}(\gamma)_b}=0 \quad\text{in }H^2(\check{\mathbb P}(\gamma)_b,\Q).
\]
Restricting \eqref{eq:RR-correction-descends} to the fiber consequently gives
\begin{equation}
  \label{claim} \left. \ch(\mathcal V_{\mathbf{C}_{\mathbb V}}) \td(T_{\mathsf P}) \right|_{\mathbb{PPM}_{\mathbf{C}_{\mathbb V},b}} =1 \quad\text{in } H^*(\mathbb{PPM}_{\mathbf{C}_{\mathbb V},b},\Q).
\end{equation}

Thus $\mathfrak U_{\gamma,0}=\mathrm{id}$, so $\mathfrak U_\gamma$ is invertible: its positive-degree part is nilpotent in cohomology. Finally, tensoring by an element of $\Gamma$ changes each evaluation line $\mathscr U_{\mathbf C_{\mathbb V},i}(\gamma)^\pm$ by a torsion line pulled back from $\mathbb V$. Hence the classes $\alpha_i^\pm$, and therefore the class defining $\mathfrak U_\gamma$, are $\Gamma$-invariant in rational cohomology. Thus $\mathfrak U_\gamma$ is $\Gamma$-equivariant.
\end{proof}

In Lemma~\ref{key lemma}, $\mathfrak U_\gamma$ is invertible in the graded algebra of cohomological correspondences. On the $\kappa_\gamma$-summand where the preceding endoscopic result applies, $[\mathscr C_{\mathbb V}]_{\#,\kappa_\gamma}$ is the homogeneous endoscopic isomorphism, whereas $\tau^{\mathrm{nor}}([Rr^{\mathrm t}_*\mathcal V_{\mathbf{C}_{\mathbb V}}])$ is its composite with the graded automorphism $\mathfrak U_\gamma$.

\subsection{Comparison of Fourier projectors}\label{subsec:comparison-Fourier-projectors}

Consider the fibrations
\[
  \xymatrix{ \check{\mathbb{P}}(\gamma) \ar[dr]^{h(\gamma)} & & [\check{\mathbb{P}}(\gamma)/\Gamma] \ar[dl] \\ & \mathbb V .& }
\]
and their Fourier correspondences
\begin{align*}
  \tau^{\mathrm{nor}}(\mathscr{F}_{\mathbb V}(\gamma)) & =\sum_{j}\mathfrak{F}_{\mathbb V,j}(\gamma)=\sum_{j, \eta\in \Gamma}\mathfrak{F}_{\mathbb V,\eta,j}(\gamma) \\
  \tau^{\mathrm{nor}}(\mathscr{F}_{\mathbb V}^{-1}(\gamma)) & =\sum_{j}\mathfrak{F}_{\mathbb V,j}^{-1}(\gamma)=\sum_{j, \eta\in \Gamma}\mathfrak{F}_{\mathbb V,\eta,j}^{-1}(\gamma),
\end{align*}

Here $(\gamma)$ specifies the relative endoscopic family, $\eta$ labels an inertia sector, and $j$ records correspondence degree $j-(g_P-a_\gamma)$. Define the Fourier projectors by
\begin{align*}
  \mathfrak{p}_{\mathbb V , k}(\gamma)= \sum_{j\leq k}\mathfrak{F}_{\mathbb V,j}^{-1}(\gamma)\circ \mathfrak{F}_{\mathbb V,2(g_P-a_\gamma)-j}(\gamma)
\end{align*}

The kernel $\mathscr F_{\mathbb V}(\gamma)$ retains the fixed descent structure specified in Subsection~\ref{subsec:endoscopic-RR-comparison}. The equivariant structure on $\mathscr F_{\mathbb V}^{-1}(\gamma)$ is induced by relative duality. Tensoring the forward kernel by $p_2^*\mathbb C_\chi$ tensors its adjoint kernel by $p_1^*\mathbb C_{\chi^{-1}}$. On the $\eta$-sector, their homogeneous Fourier components are multiplied by $\chi(\eta)$ and $\chi(\eta)^{-1}$, respectively. These degree-zero factors cancel in every summand defining $\mathfrak p_{\mathbb V,k}(\gamma)$.

The spectral algebra over $\Sigma_{\langle\gamma\rangle}\times\mathbb V$ is, as a module, the pullback of $\bigoplus_{\ell=0}^{n/\operatorname{ord}(\gamma)-1} \sO_{\Sigma_{\langle\gamma\rangle}} (-\ell\pi_{\langle\gamma\rangle}^*D)$. Hence the Lie bundle of its Jacobian has the $K_0$-class of a trivial vector bundle. Since the differential of the norm splits by $\frac1n\pi_{\mathbb V}(\gamma)^*$, we obtain
\[
  [\operatorname{Lie}\check{\mathbb P}(\gamma)] =(g_P-a_\gamma)[\sO_{\mathbb V}].
\]
Thus the relative Todd class and the Chern character of the relative dualizing line are both $1$. A translation has a fixed point only when it is trivial, so every nonempty sector has zero moving normal bundle. The normalized Fourier components are therefore the eigenvalue-weighted Chern components of the Poincar\'e line and its dual, with the shift in the inverse kernel retained.

The homomorphism $[n]-\pi_{\mathbb V}(\gamma)^*\circ\Nm$ from $J^0_{\mathbf C_{\mathbb V}(\gamma)}$ has image in $\check{\mathbb P}(\gamma)^0$ and restricts to $[n]$ on $\check{\mathbb P}(\gamma)$. Its component group is consequently killed by $n$. Choose $N>1$ with $N\equiv1\pmod n$. Multiplication by $N$ preserves the components, commutes with $\Gamma$, and induces a representable finite \'etale map of degree $N^{2(g_P-a_\gamma)}$ on every nonempty quotient sector.

The biextension identities, with the fixed descent structure, give
\[
  ([N]\times\id)^*\mathscr F_{\mathbb V}(\gamma) \simeq \mathscr F_{\mathbb V}(\gamma)^{\otimes N} \simeq (\id\times[N])^*\mathscr F_{\mathbb V}(\gamma).
\]
Since $N\equiv1\pmod n$, the sector eigenvalues are unchanged. Consequently, the degree-$j$ normalized class scales by $N^j$ under either pullback, as does the corresponding inverse class. This is the homogeneity used in \cite[Subsection~3.4]{MSY23}.

Fix a nonempty $\eta$-sector. Pullback on the first factor realizes precomposition by $[N]_*$, and pullback on the second realizes postcomposition by $[N]^*$. The finite \'etale trace identity gives
\begin{align*}
  N^{j_1+j_2} \mathfrak F_{\mathbb V,\eta,j_1}^{-1}(\gamma) \circ\mathfrak F_{\mathbb V,\eta,j_2}(\gamma) &= \mathfrak F_{\mathbb V,\eta,j_1}^{-1}(\gamma) \circ[N]_*\circ[N]^* \circ\mathfrak F_{\mathbb V,\eta,j_2}(\gamma)\\
  &= N^{2(g_P-a_\gamma)} \mathfrak F_{\mathbb V,\eta,j_1}^{-1}(\gamma) \circ\mathfrak F_{\mathbb V,\eta,j_2}(\gamma).
\end{align*}
The same calculation for the reverse composition uses $[N]$ on the intermediate scheme $\check{\mathbb P}(\gamma)$. Thus
\begin{equation}
  \label{eq:relative-reverse-FV}
\begin{aligned}
  \mathfrak F_{\mathbb V,\gamma,j_1}^{-1}(\gamma) \circ\mathfrak F_{\mathbb V,\gamma,j_2}(\gamma)&=0,\\
  \mathfrak F_{\mathbb V,\gamma,j_1}(\gamma) \circ\mathfrak F_{\mathbb V,\gamma,j_2}^{-1}(\gamma)&=0
\end{aligned}
  \qquad\text{if }j_1+j_2\neq2(g_P-a_\gamma).
\end{equation}
The same identities hold on every nonempty $\eta$-sector.

For the source inverse identity, apply \cite[Proposition~4.5]{GS22} to the smooth family $\mathbf C_{\mathbb V}(\gamma)\to\Sigma\times\mathbb V$. Its descended kernel is $\mathscr F_{\mathbb V}(\gamma)$, using the quotient presentation of Lemma~\ref{etale cover of J bar}. Proposition~\ref{prop:normalized-tau-composition} therefore gives
\[
  \left(\sum_j\mathfrak F_{\mathbb V,j}^{-1}(\gamma)\right) \circ \left(\sum_j\mathfrak F_{\mathbb V,j}(\gamma)\right) =\id.
\]
The translation calculation of Lemma~\ref{sectorwise isomorphism}, applied to the kernel and its dual, gives the same sector--character compatibility.

\begin{proposition}\label{abelian fact}
We have
\begin{equation}
  \label{image on endoscopic locus} {}^p\mathcal{H}^{i+\dim \mathbb V}\bigl(\mathrm{Im}\mathfrak{p}_{\mathbb V , k}(\gamma)\bigr)\simeq {}^p\mathcal{H}^{i+\dim \mathbb V} \bigl(Rh(\gamma)_*\mathbb C_{\check{\mathbb{P}}(\gamma)}\bigr)
\end{equation}
for every $i\leq k$, while ${}^p\mathcal{H}^{i+\dim \mathbb V} \bigl(\mathrm{Im}\mathfrak{p}_{\mathbb V,k}(\gamma)\bigr)=0$ for $i>k$.
\end{proposition}
\begin{proof}
The source inverse identity and \eqref{eq:relative-reverse-FV} show that $\mathfrak p_{\mathbb V,k}(\gamma)$ is a projector. Since $h(\gamma)$ is smooth and proper, it suffices to check its action on the cohomology of geometric fibers.

Multiplication by $N\equiv1\pmod n$ preserves each connected component. After choosing a point on such a component, its restriction is identified with multiplication by $N$ on the identity component followed by a translation of the identity component. Such translations act trivially on cohomology. Hence $[N]^*$ acts on the $i$-th cohomology of each component by $N^i$, and $[N]_*$ acts by $N^{2(g_P-a_\gamma)-i}$.

$\mathfrak F_{\mathbb V,\eta,2(g_P-a_\gamma)-j}(\gamma)$ vanishes on this cohomology unless $j=i$. Thus $\mathfrak p_{\mathbb V,k}(\gamma)$ acts as the identity for $i\leq k$ and as zero for $i>k$. This is the usual Fourier-projector description of \cite[Section~2]{Ku93}, with homological realization as in \cite[Remarks after Corollary~3.2]{DM91}; see also \cite[Subsection~2.5.2]{MSY23}. Finally, $R^i h(\gamma)_*\CC[-i]$ lies in perverse degree $i+\dim\mathbb V$, proving the proposition.
\end{proof}

Recall that by Theorem~\ref{MS's endoscopic},
\[
  [\mathscr C_{\mathbb V}]_{\#,\kappa_\gamma}: \bigl((Rh_*\mathbb C_{\mathbf M})\big|_{\mathbb V}\bigr)_{\kappa_\gamma} \xrightarrow{\ \simeq\ } \bigl(Rh(\gamma)_*\mathbb C_{\check{\mathbb P}(\gamma)}\bigr) _{\kappa_\gamma}[-2a_\gamma].
\]

\begin{proposition}\label{prop:comparison-of-Fourier-projectors}
We have
\begin{equation}
  \label{eq:projector-comparison-over-V} [\mathscr C_{\mathbb V}]_{\#,\kappa_\gamma} \circ\mathfrak p_{k,\gamma}\big|_{\mathbb V} = \mathfrak p_{\mathbb V,k-a_\gamma}(\gamma) \circ[\mathscr C_{\mathbb V}]_{\#,\kappa_\gamma}.
\end{equation}
\end{proposition}

\begin{proof}
Throughout this proof, we restrict $\mathfrak U_\gamma$, its inverse, and their homogeneous components to the $\kappa_\gamma$-summand.

The isomorphism \eqref{eq:normalization-identifies-fixed-sector} induces an isomorphism $\overline\nu$ between the corresponding quotient sectors.

The age shift $[-2a_\gamma]$ on the ambient sector is retained under this transport. By Lemma~\ref{key lemma} and Proposition~\ref{realization of compositin of functors}, the comparison on the $\kappa_\gamma$-summand is
\begin{equation}
  \label{compo of functors used} \overline\nu^*\circ \left(\sum_i\mathfrak F_{\gamma,i}\right)\Big|_{\mathbb V} = \chi_\gamma(\gamma)\, \overline{\mathsf T}_{L_0,s_0}^{\,*}\circ \left(\sum_j\mathfrak F_{\mathbb V,\gamma,j}(\gamma)\right) \circ\mathfrak U_\gamma \circ[\mathscr C_{\mathbb V}]_{\#,\kappa_\gamma}.
\end{equation}

The endoscopic correspondence lowers ordinary cohomological degree by $2a_\gamma$. The scalar $\chi_\gamma(\gamma)$ has degree zero. Taking homogeneous components, including the ambient age shift, gives
\begin{equation}
  \label{composition1} \overline\nu^*\circ\mathfrak F_{\gamma,i}|_{\mathbb V} = \chi_\gamma(\gamma)\, \overline{\mathsf T}_{L_0,s_0}^{\,*}\circ \sum_{s\geq0} \mathfrak F_{\mathbb V,\gamma,i-a_\gamma-s}(\gamma) \circ\mathfrak U_{\gamma,s} \circ[\mathscr C_{\mathbb V}]_{\#,\kappa_\gamma}.
\end{equation}
The left-hand side of \eqref{compo of functors used} is invertible on the $\kappa_\gamma$-summand. Since $\chi_\gamma(\gamma)\neq0$ and $\overline{\mathsf T}_{L_0,s_0}^{\,*}$, $\mathfrak U_\gamma$, and $[\mathscr C_{\mathbb V}]_{\#,\kappa_\gamma}$ are invertible, the same holds for $\sum_j\mathfrak F_{\mathbb V,\gamma,j}(\gamma)$ on this summand, and its inverse is $\sum_j\mathfrak F_{\mathbb V,\gamma,j}^{-1}(\gamma)$. Inverting \eqref{compo of functors used} and taking homogeneous components therefore gives
\begin{equation}
  \label{composition2} [\mathscr C_{\mathbb V}]_{\#,\kappa_\gamma} \circ\mathfrak F_{\gamma,i}^{-1}|_{\mathbb V} = \chi_\gamma(\gamma)^{-1} \sum_{r\geq0}(\mathfrak U_\gamma^{-1})_r \circ\mathfrak F_{\mathbb V,\gamma,i-a_\gamma-r}^{-1}(\gamma) \circ\bigl(\overline{\mathsf T}_{L_0,s_0}^{\,*}\bigr)^{-1} \circ\overline\nu^*.
\end{equation}
Combining \eqref{composition1} and \eqref{composition2}, the maps $\overline\nu^*$ and $\overline{\mathsf T}_{L_0,s_0}^{\,*}$ cancel with their inverses. The character factors also cancel. We obtain
\begin{align*}
  [\mathscr C_{\mathbb V}]_{\#,\kappa_\gamma} \circ\mathfrak p_{k,\gamma}|_{\mathbb V} = \sum_{\substack{i\leq k\\r,s\geq0}} (\mathfrak U_\gamma^{-1})_r \circ\mathfrak F_{\mathbb V,\gamma,i-a_\gamma-r}^{-1}(\gamma) \circ\mathfrak F_{\mathbb V,\gamma,2g_P-i-a_\gamma-s}(\gamma) \circ\mathfrak U_{\gamma,s} \circ[\mathscr C_{\mathbb V}]_{\#,\kappa_\gamma}.
\end{align*}
The sum of the two middle Fourier indices is
\[
  (i-a_\gamma-r)+(2g_P-i-a_\gamma-s) =2(g_P-a_\gamma)-(r+s).
\]
If $r+s>0$, the middle composition vanishes by \eqref{eq:relative-reverse-FV}. The remaining term has $r=s=0$, giving
\[
  \sum_{i\leq k} \mathfrak F_{\mathbb V,\gamma,i-a_\gamma}^{-1}(\gamma) \circ\mathfrak F_{\mathbb V,\gamma,2g_P-i-a_\gamma}(\gamma) \circ[\mathscr C_{\mathbb V}]_{\#,\kappa_\gamma}.
\]
By the definition of $\mathfrak p_{\mathbb V,k-a_\gamma}(\gamma)$ and the sector--character compatibility above, this is the right-hand side of \eqref{eq:projector-comparison-over-V}.
\end{proof}

\begin{theorem}\label{multiplicativity of ellptic locus}
Assume $d>2g-2$ is even. For the Fourier projectors $\mathfrak p_k^L$ in \eqref{eq:global-twisted-projectors}, there is an identification compatible with the inclusion in $Rh_*\mathbb C_{\mathbf M}$:
\begin{equation}
  \mathrm{Im} \mathfrak{p}^L_k= {}^p\tau_{\leq k+\dim \mathbf B}Rh_*\mathbb C_{\mathbf{M}}.
\end{equation}
Consequently, Theorem~\ref{coro:multi for twisted} gives a factorization of the cup product
\begin{equation}
  \cup: {}^p\tau_{\leq k+\dim \mathbf B}Rh_*\mathbb C_{\mathbf{M}}\otimes {}^p\tau_{\leq l+\dim \mathbf B}Rh_*\mathbb C_{\mathbf{M}} \longrightarrow {}^p\tau_{\leq k+l+\dim \mathbf B}Rh_*\mathbb C_{\mathbf{M}}.
\end{equation}
\end{theorem}

\begin{proof}
The decomposition $Rh_*\mathbb C_{\mathbf{M}}=\oplus_{\kappa\in \hat{\Gamma}}(Rh_*\mathbb C_{\mathbf{M}})_\kappa$ is compatible with perverse truncation:
\[
  {}^p\tau_{\leq k+\dim \mathbf B}Rh_*\mathbb C_{\mathbf{M}}=\bigoplus_{\kappa\in \hat{\Gamma}}{}^p\tau_{\leq k+\dim \mathbf B}(Rh_*\mathbb C_{\mathbf{M}})_\kappa.
\]
By the sector--character compatibility following \eqref{eq:global-twisted-sector-projectors}, similar to Lemma~\ref{sectorwise isomorphism}, $\mathfrak p_k^L=\sum_{\gamma\in\Gamma}\mathfrak p_{k,\gamma}^L$, and $\mathfrak p_{k,\gamma}^L$ acts only on $(Rh_*\mathbb C_{\mathbf M})_{\kappa_\gamma}$. Thus it remains to show that
\begin{equation}
  \mathrm{Im} \mathfrak{p}^L_{k, \gamma}= {}^p\tau_{\leq k+\dim \mathbf B}(Rh_*\mathbb C_{\mathbf{M}})_{\kappa_\gamma}.
\end{equation}
The natural inclusion $\mathrm{Im} \mathfrak{p}^L_{k, \gamma}\rightarrow (Rh_*\mathbb C_{\mathbf{M}})_{\kappa_\gamma}$ identifies its source with the truncation if
\begin{equation}
  \label{eq:projector image} {}^p\mathcal{H}^{i+\dim \mathbf{B}}(\mathrm{Im} \mathfrak{p}^L_{k, \gamma})\longrightarrow {}^p\mathcal{H}^{i+\dim \mathbf{B}}((Rh_*\mathbb C_{\mathbf{M}})_{\kappa_\gamma})
\end{equation}
is an isomorphism for all $i\leq k$ and its source vanishes for $i>k$.

Since $\mathfrak{p}^L_{k, \gamma}$ is a projector on $(Rh_*\mathbb C_{\mathbf{M}})_{\kappa_\gamma}$, by Proposition~\ref{support theorem}, both sides of \eqref{eq:projector image} are semisimple perverse sheaves whose nonzero constituents have strict support $\mathbf{B}_\gamma$. It therefore suffices to verify these assertions after pullback to an \'{e}tale cover of a dense open subset of $\mathbf{B}_\gamma$.

We take the étale cover $\mathbb V\to U\to\mathbf B^\heartsuit(\gamma)\to\mathbf B_\gamma$ constructed above, together with its map to $B$. By Lemma~\ref{lem:twisted-untwisted-projector-comparison}, the images of $\mathfrak p_{k,\gamma}^L$ and $\mathfrak p_{k,\gamma}$ define the same sub-local systems after restriction to $\mathbb V$. By Proposition~\ref{prop:comparison-of-Fourier-projectors}, the endoscopic isomorphism of Theorem~\ref{MS's endoscopic} identifies the latter with the image of $\mathfrak p_{\mathbb V,k-a_\gamma}(\gamma)$ on $Rh(\gamma)_*\CC[-2a_\gamma]$.

Proposition~\ref{abelian fact} shows that this relative projector projects onto precisely the summands indexed by $j\leq k-a_\gamma$. Since $\dim\mathbb V=\dim\mathbf B-a_\gamma$, the shift $[-2a_\gamma]$ places such a summand in ambient perverse degree
\[
  (j+\dim\mathbb V)+2a_\gamma =(j+a_\gamma)+\dim\mathbf B.
\]
Thus it gives exactly the perverse degrees $i+\dim\mathbf B$ with $i\leq k$, and none with $i>k$. This proves \eqref{eq:projector image} over $\mathbb V$. The two sides are semisimple perverse sheaves with strict support $\mathbf B_\gamma$. Hence the isomorphism and vanishing extend from this dense \'etale open to $\mathbf B_\gamma$. Summing over $\gamma$ proves the first assertion, and the multiplicativity statement follows from Theorem~\ref{coro:multi for twisted}.
\end{proof}

\section{Multiplicativity of the perverse filtration}

We first show that products of classes in nontrivial character summands vanish.

\begin{proposition}\label{square zero}
Let $D=K_\Sigma$, or let $D$ be effective with $\deg D>2g-2$. Under the decomposition
\[
  H^*(\bf{M}_{n,L}^D,\mathbb C) =\bigoplus_{\kappa\in\widehat\Gamma} H^*(\bf{M}_{n,L}^D,\mathbb C)_\kappa,
\]
the sum of the nontrivial character spaces is a square-zero ideal. In particular, if $\kappa_1\neq1$, $\kappa_2\neq1$, and $c_i\in H^*(\bf{M}_{n,L}^D,\mathbb C)_{\kappa_i}$, then
\[
  c_1\cup c_2=0.
\]
\end{proposition}

\begin{proof}
Fix $\gamma\in\Gamma$ and the corresponding character $\kappa_\gamma\in\widehat\Gamma$. The endoscopic shift of \cite[Theorem~3.2]{MS21}, together with its odd-degree variant in \cite[Remark~4.13, arXiv version~5]{MS21}, gives
\[
  H^m(\bf M_{n,L}^D,\mathbb C)_{\kappa_\gamma}=0 \quad\text{for }m<2a_\gamma.
\]
For $\gamma\neq1$, we have $2a_\gamma=n^2d(1-\frac{1}{\operatorname{ord}(\gamma)})\geq n^2d/2$.

The $\mathbb C^*$-action scaling the Higgs field makes $\bf M_{n,L}^D$ semiprojective. Its core is contained in the global nilpotent cone $(h^D)^{-1}(0)$, and the scaling flow retracts the moduli space onto the core by \cite[Corollary~1.3.6]{HRV15}. The nilpotent cone has pure dimension $g_P$; for $d>2g-2$ this follows from \cite[Proposition~2.4.9]{deC17}, while for $D=K_\Sigma$ it follows from the Lagrangian property of the global nilpotent variety \cite{Gin01}. It follows that
\[
  H^m(\bf M_{n,L}^D,\mathbb C)=0\qquad\text{for }m>2g_P.
\]
On the other hand, for $\gamma_1,\gamma_2\neq1$,
\[
  2a_{\gamma_1}+2a_{\gamma_2}\geq n^2d>2g_P.
\]
The product of two classes in nontrivial character spaces therefore vanishes. The sum of these character spaces is an ideal because multiplication by an invariant class preserves every character summand and products of two nontrivial summands vanish.
\end{proof}

Let $p_\Sigma$ and $p_M$ be the projections from $\Sigma\times\bf M_{n,L}^{K_\Sigma}$, and let $\mathbb U$ be a universal family. Choose a decomposable rational class $\delta=p_\Sigma^*\delta_\Sigma+p_M^*\delta_M$ such that the K\"unneth components of bidegree $(0,2)$ and $(2,0)$ in $\operatorname{ch}_1^\delta(\mathbb U)$ vanish, where $\operatorname{ch}^\delta(\mathbb U)=\operatorname{ch}(\mathbb U)e^\delta$. For $\eta\in H^*(\Sigma,\mathbb C)$, set
\begin{equation}
  \label{eq:normalized-tautological-class} c_k(\eta) :=p_{M*}\bigl(p_\Sigma^*\eta\cup \operatorname{ch}_k^\delta(\mathbb U)\bigr).
\end{equation}
These classes are independent of the choice of universal family and are $\Gamma$-invariant.

We use the subscript $\mathrm{st}$ to denote the $\Gamma$-invariant summand, both in cohomology and in the Hitchin direct image.

\begin{proposition}\label{prop:stable-filtered-generation}
For the canonical Hitchin moduli space,
\[
  P_iH^*(\bf M_{n,L}^{K_\Sigma},\mathbb C)_{\mathrm{st}} =\operatorname{Span}\left\{ \prod_{a=1}^s c_{k_a}(\eta_a) \;\middle|\; \sum_{a=1}^s k_a\leq i \right\}.
\]
In particular, the perverse filtration on the stable summand is multiplicative.
\end{proposition}

\begin{proof}
Put $\ell=\deg L$, let $\widetilde{\bf M}_{n,\ell}^{K_\Sigma}$ be the corresponding $\operatorname{GL}_n$-Higgs moduli space, and let $j:\bf M_{n,L}^{K_\Sigma}\hookrightarrow \widetilde{\bf M}_{n,\ell}^{K_\Sigma}$ be the natural inclusion. For the trivial character, the results of \cite[Proposition~5.2 and Corollary~5.3]{MS21} give a strict filtered surjection
\[
  j^*:H^*(\widetilde{\bf M}_{n,\ell}^{K_\Sigma},\mathbb C) \longrightarrow H^*(\bf M_{n,L}^{K_\Sigma},\mathbb C)_{\mathrm{st}}.
\]
By \cite[Theorem~0.4 and Subsection~5.4]{MSY23}, after extending scalars to $\mathbb C$, the perverse filtration on the source is its Chern filtration. Choose the normalized universal families compatibly. Then $j^*c_k^{\operatorname{GL}}(\eta)=c_k(\eta)$, and the assertion follows because $j^*$ is a ring homomorphism.
\end{proof}

By Propositions~\ref{square zero} and~\ref{prop:stable-filtered-generation}, it suffices to prove that multiplication by each tautological class $c_k(\eta)\in H^*(\bf M_{n,L}^{K_\Sigma},\mathbb C)_{\mathrm{st}}$ satisfies
\[
  c_k(\eta)\cup : P_jH^*(\bf M_{n,L}^{K_\Sigma},\mathbb C) \longrightarrow P_{k+j}H^*(\bf M_{n,L}^{K_\Sigma},\mathbb C).
\]
We adapt the reduction steps of Hausel--Mellit--Minets--Schiffmann \cite[Section~8]{HMMS} to fixed determinant and trace-free Higgs fields, following \cite[Subsections~5.4.2--5.4.3]{MSY23}.

Fix distinct points $p,q\in\Sigma$ and put $K_{p,q}=K_\Sigma+p+q$. Choose generic full-flag weights at $p$ and $q$, sufficiently close to zero that a zero-residue parabolic Higgs bundle is stable if and only if its underlying Higgs bundle is stable. Such a chamber exists because $\gcd(n,\deg L)=1$.

Following \cite[Subsection~5.4.2]{MSY23}, consider the $K_{p,q}$-twisted moduli spaces:
\begin{equation}
  \label{eq:parabolic-chain} \xymatrix@C=1.25em{ \bf M_{n,L}^{K_\Sigma} &\bf M_{n,L}^{K_{p,q},0}\ar[l]_{f}\ar@{^{(}->}[r]^{\iota_0} &\bf M_{n,L}^{K_{p,q},\mathrm{spar}}\ar@{^{(}->}[r]^{\iota^{\mathrm{spar}}} &\bf M_{n,L}^{K_{p,q},\mathrm{wpar}} &\bf M_{n,L}^{K_{p,q},\mathrm{wpar,ell}}\ar@{_{(}->}[l]_{\iota} }
\end{equation}
Here ``$\mathrm{wpar}$'' denotes the weakly parabolic moduli space with full flags at $p$ and $q$, ``$\mathrm{spar}$'' is the strongly parabolic locus, ``$\mathrm{wpar,ell}$'' is the open locus where the spectral curve is integral and the two residues are regular semisimple, with their eigenlines ordered by the flags, and the superscript $0$ denotes zero residue. The degree $d_L$ is defined by \eqref{eq:spectral-sheaf-degree} with $D=K_{p,q}$.

On each space, define normalized tautological classes by~\eqref{eq:normalized-tautological-class}; we again denote them by $c_k(\eta)$.

Let $W\subset \bf{B}^{K_{p,q}, \mathrm{ell}}$ be the locus of integral spectral curves with fibers over $p$ and $q$ consisting of $n$ distinct points, and let $\mathbf{C}_W\rightarrow W$ be the universal curve. There is a Cartesian diagram
\[
  \xymatrix{ \bf M_{n,L}^{K_{p,q},\mathrm{wpar,ell}}\ar[r]^{t}\ar[d]^{h^{\mathrm{wpar,ell}}}&\check{\mathbf P}^{L}_{\mathbf{C}_W}\ar[d]\\ \widetilde W\ar[r]&W, }
\]
Here $t$ forgets the flags at $p$ and $q$, and $\widetilde W$ parametrizes spectral curves in $W$ together with an ordering of the points above $p$ and $q$. Both horizontal arrows are finite \'etale $S_n\times S_n$-torsors.

\begin{lemma}\label{lem:tautological-perversity-on-W}
Let $D=K_{p,q}$ and let $W$ be the open subset defined above. Then
\[
  c_k(\eta)\in P_kH^*(\check{\mathbf P}^{L}_{\mathbf{C}_W},\CC).
\]
\end{lemma}

\begin{proof}
First consider the corresponding $\mathrm{GL}_n$ spectral family. The proof of \cite[Corollary~5.7]{MSY23}, using Corollary~4.6 and Subsection~5.3.3 there, applies with the ruled surface
\[
  \mathbb P_\Sigma\bigl(\sO_\Sigma(K_{p,q})\oplus\sO_\Sigma\bigr).
\]
Indeed, the spectral curves are integral and planar, and their compactified Jacobian has smooth total space by the coprime Higgs moduli description. Grothendieck--Riemann--Roch for the projection to $\Sigma$ expresses $c_k(\eta)$ through normalized spectral Chern-character components of indices at most $k+1$; the estimate in that proof places the component of index $j$ in $P_{j-1}$. The normalization is the one in \eqref{eq:normalized-tautological-class}. Thus $c_k(\eta)\in P_k$ for this $\mathrm{GL}_n$ family.

Write $\sigma=\frac1n\tr(\theta)$ and $\theta_0=\theta-\sigma\mathrm{id}_E$. The mutually inverse assignments
\[
  (E,\theta)\longmapsto((E,\theta_0),\sigma), \qquad ((E,\theta_0),\sigma)\longmapsto (E,\theta_0+\sigma\mathrm{id}_E)
\]
identify the $\mathrm{GL}_n$ moduli space with the product of its trace-free locus and $H^0(\Sigma,\sO_\Sigma(K_{p,q}))$. They preserve stability because $\theta$ and $\theta_0$ have the same invariant subbundles. The identity
\[
  \det(\lambda\mathrm{id}_E-\theta) = \det((\lambda-\sigma)\mathrm{id}_E-\theta_0)
\]
gives the corresponding product coordinates on the Hitchin base. In these coordinates, the Hitchin morphism is the product of its trace-free restriction with the identity on $H^0(\Sigma,\sO_\Sigma(K_{p,q}))$; see \cite[Lemma~4.3.1]{deC17}.

Spectral translation preserves integrality and the distinctness of the points above $p$ and $q$, and does not change the underlying universal vector bundle. Removing this affine-space factor preserves the normalized perverse filtration. We obtain the same bound on $\overline{\mathbf J}^{d_L}_{\mathbf C_W}$.

Let $j:\check{\mathbf P}^{L}_{\mathbf C_W}\hookrightarrow \overline{\mathbf J}^{d_L}_{\mathbf C_W}$ be the norm-fiber inclusion. The finite \'etale presentation
\[
  \overline{\mathbf J}^{d_L}_{\mathbf C_W} \simeq \bigl(J_\Sigma^0\times\check{\mathbf P}^{L}_{\mathbf C_W}\bigr)/\Gamma
\]
identifies its direct image with $H^*(J_\Sigma^0,\CC)\otimes(R\check h^L_*\CC|_W)_{\mathrm{st}}$, where the first factor is a graded complex: translations act trivially on $H^*(J_\Sigma^0,\CC)$. Under this identification, $j^*$ is projection to the degree-zero factor and is strictly filtered. It carries the normalized universal classes to those on the Prym, giving the result.
\end{proof}

\begin{proposition}\label{prop:wparell}
For every normalized tautological class \(c_k(\eta)\) in $H^*(\bf M_{n,L}^{K_{p,q},\mathrm{wpar,ell}},\CC)$, the cup product satisfies
\[
  c_k(\eta)\cup-: P_iH^*(\bf M_{n, L}^{K_{p,q},\mathrm{wpar,ell}},\CC) \longrightarrow P_{i+k}H^*(\bf M_{n, L}^{K_{p,q},\mathrm{wpar,ell}},\CC).
\]
\end{proposition}

\begin{proof}

The preceding Cartesian diagram is a finite \'etale base change, and $\deg K_{p,q}=2g>2g-2$ is even. Pull back the sheaf-theoretic cup-product factorization of Theorem~\ref{multiplicativity of ellptic locus} along $\widetilde W\to W$. Proper base change and perverse $t$-exactness of \'etale pullback identify this with the corresponding factorization for $\bf M_{n,L}^{K_{p,q},\mathrm{wpar,ell}}$. Thus its perverse filtration is multiplicative. Since $c_k(\eta)$ is pulled back from $\check{\mathbf P}^{L}_{\mathbf{C}_W}$, Lemma~\ref{lem:tautological-perversity-on-W} gives $c_k(\eta)\in P_kH^*(\bf M_{n, L}^{K_{p,q},\mathrm{wpar,ell}},\CC)$, and the result follows.
\end{proof}

Taking residues at $p$ and $q$ gives a morphism
\[
  \chi_{\mathrm{res}}:\mathbf M_{n,L}^{K_{p,q},\mathrm{wpar}} \longrightarrow\mathfrak t_0\times\mathfrak t_0, \qquad \mathfrak t_0=\{(\mu_1,\ldots,\mu_n)\in\mathbb C^n:\sum_i\mu_i=0\},
\]
which is surjective and has the strongly parabolic locus as its zero fiber. Requiring the full residues to vanish instead gives $\mathbf M_{n,L}^{K_{p,q},0}$.

\begin{lemma}\label{smoothness of residue map}
The map $\chi_{\mathrm{res}}$ is smooth. For every $\mu\in\mathfrak t_0\times\mathfrak t_0$, restriction induces a $\Gamma$-equivariant isomorphism
\[
  H^*(\mathbf M_{n,L}^{K_{p,q},\mathrm{wpar}},\mathbb C) \xrightarrow{\ \simeq\ }H^*(\chi_{\mathrm{res}}^{-1}(\mu),\mathbb C).
\]
\end{lemma}

\begin{proof}
At $(E,\theta,F_p^\bullet,F_q^\bullet)$, the deformation complex is
\[
  \operatorname{ParEnd}_0(E)\xrightarrow{[\theta,-]} \operatorname{ParEnd}_0(E)\otimes K_{p,q}.
\]
Fixing the ordered residue eigenvalues gives the complex
\[
  \operatorname{ParEnd}_0(E)\xrightarrow{[\theta,-]} \operatorname{SParEnd}_0(E)\otimes K_{p,q}.
\]
The trace pairing and Serre duality identify its second hypercohomology with the dual of its zeroth hypercohomology. The latter vanishes by stability and the trace-free condition. Thus the relative obstruction space vanishes and $\chi_{\mathrm{res}}$ is smooth.

The space $\mathbf M_{n,L}^{K_{p,q},\mathrm{wpar}}$ is semiprojective. All residue eigenvalues have weight one under scaling of the Higgs field. The inverse image of a line through zero in $\mathfrak t_0\times\mathfrak t_0$ is smooth and semiprojective and has the same core as $\mathbf M_{n,L}^{K_{p,q},\mathrm{wpar}}$ since every point of the core has zero residue eigenvalues. Restriction to this inverse image induces an isomorphism on cohomology by the core retraction of \cite[Corollary~1.3.6]{HRV15}. The induced map to the line is smooth and surjective, so \cite[Corollary~1.3.3]{HRV15} identifies its cohomology with that of each fiber. Taking a line through $\mu$ proves the assertion. All restriction maps are $\Gamma$-equivariant.
\end{proof}

\begin{proposition}\label{prop:SL-parabolic-restrictions}
The pullback
\[
  (\iota^{\mathrm{spar}})^*: H^*(\bf M_{n,L}^{K_{p,q},\mathrm{wpar}},\mathbb C) \xrightarrow{\;\simeq\;} H^*(\bf M_{n,L}^{K_{p,q},\mathrm{spar}},\mathbb C)
\]
is a strict filtered isomorphism, while
\[
  \iota^*: H^*(\bf M_{n,L}^{K_{p,q},\mathrm{wpar}},\mathbb C) \longrightarrow H^*(\bf M_{n,L}^{K_{p,q},\mathrm{wpar,ell}},\mathbb C)
\]
is a strict filtered injection. Both maps preserve the normalized tautological classes and are $\Gamma$-equivariant.
\end{proposition}

\begin{proof}
The first map is an isomorphism by Lemma~\ref{smoothness of residue map} with $\mu=0$. For strictness, we use the projectivization argument of \cite[Proposition~8.9]{HMMS}. The ordered Hitchin base is a positive-weight affine space: for $i\geq2$, the evaluations of $H^0(\Sigma,\sO_\Sigma(iK_{p,q}))$ at $p$ and $q$ are surjective, since $H^1(\Sigma,\sO_\Sigma(iK_{p,q}-p-q))=0$. Thus the ordered residue eigenvalues, together with the kernels of these evaluation maps, give its coordinates. The strongly parabolic base is cut out by the linear equations $\mu=0$. The total space is smooth and semiprojective, the Hitchin map is projective, and the residue map is smooth by the preceding lemma. The projectivization argument therefore applies and proves that $(\iota^{\mathrm{spar}})^*$ is strict.

For the second map, choose $\mu=(\mu_p,\mu_q)$ with distinct entries at each point and such that
\[
  \sum_{i\in I}\mu_{p,i}+\sum_{j\in J}\mu_{q,j}\ne0 \quad\text{whenever }1\leq|I|=|J|<n.
\]
Each forbidden equality defines a proper hyperplane in $\mathfrak t_0\times\mathfrak t_0$, so such a $\mu$ exists. A proper invariant subbundle of rank $r$ would give a meromorphic trace form whose residues are the two sums for subsets of cardinality $r$, contrary to the residue theorem. The spectral polynomial is therefore irreducible; distinct residue eigenvalues also exclude a repeated factor. Hence
\[
  \chi_{\mathrm{res}}^{-1}(\mu) \subset \bf M_{n,L}^{K_{p,q},\mathrm{wpar,ell}}
\]
Since restriction to this fiber is an isomorphism on cohomology, $\iota^*$ is injective. The restriction to the corresponding open Hitchin base is compatible with the relative Lefschetz action. The argument of \cite[Proposition~8.7 and the paragraph following it]{HMMS}, using its Theorem~3.6(i) and Proposition~3.3, gives strictness.

The universal families pull back from the weakly parabolic space, so both maps preserve the normalized tautological classes. The maps are $\Gamma$-equivariant because tensoring by an element of $\Gamma$ preserves the residues and the flags.
\end{proof}

\begin{proposition}\label{prop:SL-zero-residue-splitting}
The maps $A=f_*\iota_0^*$ and $B=(\iota_0)_*f^*$ preserve every $\Gamma$-character summand. Their filtered degrees are given by
\begin{align*}
  B:\ P_jH^m(\mathbf M_{n,L}^{K_\Sigma}, \CC) &\longrightarrow P_{j+n(n-1)}H^{m+2n(n-1)}(\mathbf M_{n,L}^{K_{p,q},\mathrm{spar}}, \CC),\\
  A:\ P_jH^m(\mathbf M_{n,L}^{K_{p,q},\mathrm{spar}}, \CC) &\longrightarrow P_{j-n(n-1)}H^{m-2n(n-1)}(\mathbf M_{n,L}^{K_\Sigma}, \CC).
\end{align*}
They satisfy
\[
  AB=\pm(n!)^2\operatorname{id}.
\]
For classes $\alpha$ and $\beta$ on the canonical and strongly parabolic spaces, respectively,
\begin{align*}
  B\bigl(c_k(\eta)\cup\alpha\bigr) &=c_k(\eta)\cup B(\alpha),\\
  A\bigl(c_k(\eta)\cup\beta\bigr) &=c_k(\eta)\cup A(\beta).
\end{align*}
\end{proposition}

\begin{proof}
Let $G=\operatorname{PGL}_n$. The trace pairing identifies $T^*\operatorname{Fl}_n$ with the space of pairs consisting of a full flag and an endomorphism strictly lowering it. Taking residues at $p$ and $q$ gives the diagram
\[
  \xymatrix@C=1.2em@R=1.8em{ \mathbf M_{n,L}^{K_\Sigma}\ar[d] &\mathbf M_{n,L}^{K_{p,q},0} \ar[l]_{f}\ar@{^{(}->}[r]^{\iota_0}\ar[d] &\mathbf M_{n,L}^{K_{p,q},\mathrm{spar}} \ar[d]^{\mathrm{ev}_{p,q}}\\ [\operatorname{pt}/G]^2 &[\operatorname{Fl}_n/G]^2\ar[l]\ar@{^{(}->}[r] &[T^*\operatorname{Fl}_n/G]^2. }
\]
The bottom arrows are the structure morphism and the product of the zero sections. The vertical maps record the projective fibers at $p$ and $q$, together with the flags and residues when present. Both squares are $2$-Cartesian: the left square adds arbitrary full flags, while the right square imposes zero residues. Here the small-weight assumption identifies zero-residue parabolic stability with stability of the underlying canonical Higgs bundle. The bottom row is the product of the local diagrams in \cite[Section~3.2, Equation~(44)]{MSY26}.

The evaluation map $\mathrm{ev}_{p,q}$ is smooth along the zero-residue locus. Indeed, its relative deformation complex there is
\[
  \operatorname{End}_0(E)(-p-q) \xrightarrow{[\theta,-]} \operatorname{End}_0(E)\otimes K_\Sigma.
\]
By Serre duality, its second hypercohomology is dual to the space of trace-free Higgs endomorphisms of $(E,\theta)$, which vanishes by stability. This is the two-point, fixed-determinant version of the deformation argument used in \cite[Proposition~3.2]{MSY26}. Consequently, $f$ is a smooth proper $\operatorname{Fl}_n\times\operatorname{Fl}_n$-bundle of relative dimension $n(n-1)$, and $\iota_0$ is a regular closed immersion of the same codimension. Its normal bundle restricts on each flag fiber to $T^*(\operatorname{Fl}_n\times\operatorname{Fl}_n)$.

The cycle of $\mathbf M_{n,L}^{K_{p,q},0}$ and its transpose realize $B=(\iota_0)_*f^*$ and $A=f_*\iota_0^*$, respectively, as in \cite[Section~4.2]{MSY26}. They are proper correspondences over $\mathbf B^{K_{p,q},\mathrm{spar}}$, where the canonical Hitchin base is included by its natural closed embedding, and their realizations on the unshifted Hitchin direct images have shifts $2n(n-1)$ and $-2n(n-1)$. Since
\[
  \mathbf B^{K_{p,q},\mathrm{spar}} =\bigoplus_{i=2}^n H^0\bigl(\Sigma,\sO_\Sigma(iK_{p,q}-p-q)\bigr),
\]
Riemann--Roch gives
\[
  \dim\mathbf B^{K_{p,q},\mathrm{spar}} -\dim\mathbf B^{K_\Sigma} =\sum_{i=2}^n2(i-1)=n(n-1).
\]
Pushforward by the closed embedding of Hitchin bases is perverse $t$-exact. The same truncation calculation as in \cite[Section~4.3.1]{MSY26}, with this dimension difference, gives the filtered degrees $n(n-1)$ for $B$ and $-n(n-1)$ for $A$.

The local self-intersection calculation in \cite[Lemma~3.1(b)]{MSY26} applies to each factor of the lower diagram and pulls back to the present diagram, as in \cite[Proposition~3.5(a)]{MSY26}. Explicitly,
\[
  AB(\alpha) =f_*\bigl(c_{n(n-1)}(N_{\iota_0})\cup f^*\alpha\bigr) =\pm\chi(\operatorname{Fl}_n\times\operatorname{Fl}_n)\alpha =\pm(n!)^2\alpha.
\]
Compatibility of the normalized universal classes gives $\iota_0^*c_k^{\mathrm{spar}}(\eta) =f^*c_k^{K_\Sigma}(\eta)$, so both module identities follow from the projection formula. Finally, tensoring by an element of $\Gamma$ commutes with $f$ and $\iota_0$. Hence $A$ and $B$ are $\Gamma$-equivariant and preserve every character summand.
\end{proof}

\begin{theorem}\label{main theorem}
The perverse filtration for the canonical $\operatorname{SL}_n$ Hitchin map is multiplicative:
\[
  P_kH^*(\bf M_{n,L}^{K_\Sigma},\mathbb C)\cup P_jH^*(\bf M_{n,L}^{K_\Sigma},\mathbb C) \subseteq P_{k+j}H^*(\bf M_{n,L}^{K_\Sigma},\mathbb C).
\]
\end{theorem}

\begin{proof}
By Propositions~\ref{square zero} and~\ref{prop:stable-filtered-generation}, it suffices to check that
\[
  c_k(\eta)\cup : P_jH^*(\bf M_{n,L}^{K_\Sigma},\mathbb C) \longrightarrow P_{k+j}H^*(\bf M_{n,L}^{K_\Sigma},\mathbb C).
\]
For $\alpha\in P_jH^*(\mathbf M_{n,L}^{K_\Sigma})$, Proposition~\ref{prop:SL-zero-residue-splitting} gives $B\alpha\in P_{j+n(n-1)}$ on the strongly parabolic space. By the strict filtered isomorphism $(\iota^{\mathrm{spar}})^*$ of Proposition~\ref{prop:SL-parabolic-restrictions}, it has a unique lift in $P_{j+n(n-1)}$ on the weakly parabolic space. Restrict this lift to the elliptic open and apply Proposition~\ref{prop:wparell}. The product with $c_k(\eta)$ lies in $P_{j+n(n-1)+k}$ there. Strictness and injectivity of $\iota^*$ place the product of the lifted class with $c_k(\eta)$ in the same filtration step on the weakly parabolic space. Restricting back by $(\iota^{\mathrm{spar}})^*$, which preserves tautological classes, yields $c_k(\eta)\cup B\alpha\in P_{j+n(n-1)+k}$. Applying $A$ lowers the perverse index by $n(n-1)$, and hence Proposition~\ref{prop:SL-zero-residue-splitting} gives
\[
  A\bigl(c_k(\eta)\cup B\alpha\bigr) =\pm(n!)^2c_k(\eta)\cup\alpha\in P_{k+j}H^*(\bf M_{n,L}^{K_\Sigma},\mathbb C).
\]
Since $(n!)^2\neq0$ in $\mathbb C$, this is the required bound.
\end{proof}

\bibliographystyle{alpha}
\bibliography{ref}
\end{document}